%% file: Main1_Inexact__1_.tex
\documentclass{article}
\usepackage{graphicx} 
\usepackage{amsmath,amssymb,amsthm}
\usepackage{fullpage}

\usepackage{xcolor}
\usepackage{algorithm}
\usepackage{booktabs}
\usepackage{algorithmic}
\usepackage{hyperref}
\usepackage{soul}
\usepackage{float}
\makeatletter
\def\BState{\State\hskip-\ALG@thistlm}
\makeatother
 
\DeclareMathOperator{\KL}{KL}
\DeclareMathOperator{\TV}{TV}

\DeclareMathOperator{\CLS}{\boldsymbol{C}_{LS}}

\DeclareMathOperator{\I}{I}

\DeclareMathOperator{\RRT}{RRT}

\usepackage{comment}

\title{Langevin for Nonconvex Optimization:\\
Exact, Inexact and Zeroth-Order}
\author{Emanuele Naldi\thanks{MaLGa - DIMA, University of Genova, Genova, IT. emanuele.naldi@edu.unige.it; silvia.villa@unige.it} 
        \and 
        Marco Rando\thanks{Universit\'e C\^ote d'Azur, INRIA, CNRS, LJAD, Nice, FR. marco.rando@univ-cotedazur.fr}
        \and 
        Lorenzo Rosasco\thanks{MaLGa - DIBRIS, University of Genova, Genova, IT; Istituto Italiano di Tecnologia, Genova, IT. lorenzo.rosasco@unige.it}
        \and 
        Silvia Villa\footnotemark[1]
        }

\date{}

\input{MacroBook}

\begin{document}

\maketitle

\begin{abstract}
We study Langevin-based methods for non-convex optimization under smoothness and dissipativity assumptions, focusing on non-asymptotic expected excess-risk bounds rather than sampling guarantees for the full target distribution. A main methodological message is that relative-entropy sampling guarantees can be converted directly into expected objective-value guarantees, without passing through Wasserstein distance. This direct KL-to-objective route yields sharper dependence on the Log-Sobolev constant, which may scale exponentially with inverse temperature and dimension in non-convex problems. We first analyze the Unadjusted Langevin Algorithm with exact gradients and derive explicit bounds in terms of inverse temperature, dimension, stepsize, smoothness and dissipativity parameters, and the Log-Sobolev constant. We then extend the result to inexact ULA, allowing biased and stochastic gradient surrogates whose mean-square error grows at most quadratically in the state. This framework covers stochastic gradients and zeroth-order estimators based on function evaluations. We show that Gaussian and spherical finite-difference estimators fit the theory and obtain explicit function-evaluation complexity bounds. To the best of our knowledge, these are the first non-asymptotic global non-convex optimization complexity bounds for zeroth-order ULA. We also provide numerical experiments illustrating the spherical zeroth-order scheme.
\end{abstract}


\section{Introduction}

Non-convex optimization is a central problem in modern machine learning, statistics, signal processing, inverse problems, and many other applications. In this setting, classical descent methods may be trapped by local geometry, while global optimization requires mechanisms that can explore the landscape beyond a single basin of attraction. Global optimization is generally subject to the curse of dimensionality: for black-box Lipschitz functions on a compact subset of $\mathbb R^d$, worst-case accuracy $\epsilon$ requires a number of function evaluations
of order $(L/\epsilon)^d$~~\cite{Bull2011,MalherbeVayatis2017}. Structured assumptions can lead to problem-dependent guarantees that expose how the geometry of the objective enters the complexity, without generally eliminating this dimensional dependence. Classical direct-search and random-search methods provide a general framework for derivative-free problems~\cite{ConnScheinbergVicente2009}. Under suitable regularity assumptions, these methods admit convergence guarantees, typically formulated in terms of stationarity or local optimality~\cite{AudetDennis2006}, or probabilistic convergence~\cite{SolisWets1981}. Another important family is given by consensus-based optimization and related swarm methods~\cite{PinnauTotzeckTseMartin2017,CarrilloJinLiZhu2021,FornasierKlockRiedl2024}. These methods evolve a population of interacting particles which explore the space stochastically while being attracted toward a consensus point, typically a weighted average favoring low objective values. Recent theory establishes global convergence and constrained variants under suitable assumptions on the objective, the particle system, and the consensus mechanism~\cite{FornasierKlockRiedl2024,BeddrichChencheneFornasierHuangWohlmuth2026}, and related work connects consensus dynamics with evolution strategies~\cite{FornasierHuangKlemencMalaspina2026}. Other related methods include recent zeroth-order proximal-point schemes, where the standard proximal map is replaced by a Gibbs-weighted local average of function values, thereby favoring regions with small objective values~\cite{NaldiLabarriereMolinariVilla2026}, with convergence guarantees to global minima~\cite{ZhangHanChowOsherSchaeffer2024}, as well as the Proximal Basin Hopping algorithm \cite{lauga2026proximalbasinhoppingglobal}, which combines proximal optimization with local minimization to obtain high-probability convergence to the global minimizer with finite samples. When first-order information is available, a classical stochastic approach to global non-convex optimization is simulated annealing. Its early theory established asymptotic convergence to global minimizers under logarithmic cooling schedules~\cite{GemanHwang1986,ChiangHwangSheu1987,Hajek1988}. More recent works provide quantitative high-probability bounds for objective-value errors, both in continuous and discrete time~\cite{TangZhou2023,TangWuZhou2024}. These guarantees are genuinely global, but their assumptions often rely on isolated global minima and a detailed metastable description of the energy landscape, involving wells, communication heights, saddle points, and the critical depth $E^\ast$; see also the metastability literature~\cite{BovierEckhoffGayrardKlein2004,BovierDenHollander2015}. 

\paragraph{Langevin dynamic.} This is complementary to the framework pursued in this paper. We do not design a cooling schedule, and we do not assume explicit knowledge of wells, barriers, or critical depths. Instead, we study fixed-temperature Langevin algorithms under smoothness and dissipativity, compressing the global geometry into functional-inequality constants. Given an objective function $F:\mathbb{R}^d\to\mathbb{R}$ and an inverse temperature $\beta>0$, the Langevin diffusion
\[
    dX_t=-\nabla F(X_t)\,dt+\sqrt{2\beta^{-1}}\,dB_t
\]
combines deterministic descent with isotropic stochastic exploration. Under suitable assumptions, its invariant measure is the Gibbs distribution $\pi_{\beta}$ defined by
\begin{equation}\label{eq:GIBBS}
    d \pi_{\beta}=Z^{-1}e^{-\beta F}\,dx,
    \qquad
    Z=\int_{\mathbb{R}^d}e^{-\beta F(y)}\,dy .
\end{equation}
As $\beta$ increases, $\pi_\beta$ concentrates near the global minimizers of $F$, so approximate sampling from $\pi_\beta$ can be used as a global optimization principle. In contrast with simulated annealing, where the temperature is decreased over time, we consider fixed-temperature Langevin algorithms and choose $\beta$ large enough to make the Gibbs bias compatible with the target optimization accuracy.

The algorithmic object we study is the Unadjusted Langevin Algorithm (ULA),
\[
    x_{k+1}
    =
    x_k-\gamma\nabla F(x_k)
    +
    \sqrt{2\gamma\beta^{-1}}\,z_k,
    \qquad
    z_k\sim\mathcal N(0,I_d),
\]
together with its inexact, stochastic, and zeroth-order variants. The non-asymptotic analysis of Langevin methods for non-convex optimization was systematically developed by Raginsky, Rakhlin and Telgarsky~\cite{Raginsky2017}, who obtained guarantees under smoothness and dissipativity by combining sampling estimates with Gibbs concentration bounds. Subsequent works refined this program in several directions: Wasserstein convergence bounds with non-accumulating discretization error~\cite{Zhang2023}, direct discrete-time analyses for Langevin-based non-convex optimization~\cite{XCZG2018}, Cheeger-type and total-variation approaches for non-log-concave sampling~\cite{Zou2021Faster}, variance-reduced Langevin dynamics~\cite{KinoshitaSuzuki2022}, and Lyapunov-potential methods yielding high-probability hitting-time guarantees~\cite{Chen2024}. Related recent work also considers sampling-based approaches to global non-convex optimization under different assumptions~\cite{Cortild2025}. Closely related sampling results analyze ULA through functional inequalities, isoperimetry, or tail-growth conditions~\cite{VempalaWibisono2019,ChewiErdogduLiShenZhang2022,ErdogduHosseinzadeh2021}. Other non-log-concave sampling analyses obtain convergence rates under stronger geometric assumptions, such as strong convexity or contractivity outside a compact set~\cite{CCABJ2019,majka2018nonasymptotic}, while further works study weakly smooth regimes, stationarity guarantees, or nonsmooth extensions~\cite{BCESZ2022,NguyenDangChen2023,Johnston2025}. Classical and log-concave ULA analyses, including inaccurate-gradient variants, provide important background~\cite{RobertsTweedie1996,Dalalyan2017,DurmusMoulines2017,Dalalyan2019}. Zeroth-order Langevin methods have also been studied mainly from a sampling perspective~\cite{Roy2022}. 

The quantitative analysis is delicate because of the competition between optimization bias and mixing. Taking $\beta$ large improves the Gibbs concentration around global minimizers, but it also makes the dynamics less noisy and may dramatically slow convergence to equilibrium. In non-convex landscapes, this difficulty is encoded in functional inequality constants such as the logarithmic Sobolev constant $\CLS(\beta,d)$, which may depend exponentially on both $\beta$ and the dimension. Most existing Langevin optimization guarantees are derived from sampling guarantees: one first bounds a distance between the law of the algorithm and the Gibbs measure, typically in Wasserstein distance, total variation, or KL divergence, and then converts this distributional estimate into an objective-value bound. While natural, this route may lose important factors when the final goal is optimization rather than sampling 

\paragraph{Goal and contribution.} The main goal of this paper is to control
\begin{equation}\label{eq}
\mathbb E[F(x_k)]-\min F
\end{equation}
under smoothness and dissipativity assumptions. In contrast with several previous analyses, our main methodological observation is that, whenever a sampling result provides a relative-entropy estimate, the quantity in \eqref{eq} can be controlled directly. Specifically, using a weighted Csiszár--Kullback--Pinsker inequality, we avoid the usual intermediate step of converting relative entropy into Wasserstein distance through Talagrand's inequality, as in \cite{Raginsky2017,KinoshitaSuzuki2022,Roy2022,LiuWang2020}, or through related weighted transportation inequalities, as in \cite{XCZG2018}. As detailed below, this yields a sharper dependence on the dominant geometric quantity $\CLS(\beta,d)$. Beyond this application, the resulting KL-to-objective transfer principle is independent of ULA and can be combined with other relative-entropy sampling bounds under suitable exponential-moment conditions. A further feature of our analysis is that we track explicitly the dependence on the dimension, inverse temperature, smoothness and dissipativity parameters, initialization, and oracle accuracy. Making the resulting complexity guarantees fully explicit requires careful treatment of all these dependencies: quantities commonly treated as fixed
initialization constants may themselves conceal dimensional dependence
(see Remark~\ref{rem:gaussian_initialization}), while other steps
require substantial quantitative calculations. This tracking allows us, for any prescribed accuracy, to provide concrete choices of the inverse temperature, stepsize, number of iterations, and, when applicable, batch size and smoothing parameters. To the best of our knowledge, this is the first analysis of Langevin-based methods for global non-convex optimization that provides such fully explicit gradient and function-evaluation oracle-complexity guarantees, with dependencies that are often left implicit in previous analyses.

We first carry out this program for exact-gradient ULA, proving an explicit non-asymptotic bound on $\mathbb E[F(x_k)]-\min F$ which tracks the dependence on $\beta$, $d$, the structural constants of $F$, the stepsize, and $\CLS(\beta,d)$. The resulting parameter calibration gives the iteration complexity
\[
    \tilde O\!\left(
        \frac{\beta^2 d\,\CLS(\beta,d)^2}{\epsilon^2}
    \right).
\]
The key feature is the quadratic dependence on $\CLS(\beta,d)$, improving over previous bounds. We then extend the analysis to an inexact-gradient version of ULA, where $\nabla F(x)$ is replaced by a possibly biased stochastic surrogate $g(x,\xi)$. Our oracle assumption is flexible, covering bounded-variance stochastic gradients, mini-batch gradients, biased deterministic approximations, and finite-difference estimators.

As applications, we derive stochastic-gradient and zeroth-order Langevin optimization guarantees. The zeroth-order part is a central contribution of the paper. We consider both Gaussian and spherical finite-difference estimators, which are naturally biased because they approximate gradients of smoothed objectives and also introduce variance through random directions. We show that both estimators satisfy the quadratic-growth error condition required by our inexact-ULA theory and obtain explicit function-evaluation complexity bounds for the resulting zeroth-order Langevin algorithms. To the best of our knowledge, these are the first non-asymptotic global non-convex optimization complexity bounds for zeroth-order ULA under the standard smooth dissipative assumptions.

The rest of the paper is organized as follows. Section~\ref{sec:Langevin} studies exact-gradient ULA and proves the main excess-risk bound. Section~\ref{sec:Inexact} develops the inexact-gradient theory, including stability of the iterates and propagation of gradient approximation errors. Section~\ref{sec:Applications} applies the framework to mini-batch and zeroth-order Langevin algorithms. Detailed comparisons with existing bounds are given throughout the paper and in the appendices.

\subsection{Notations}

Throughout the paper, we employ standard asymptotic notation to characterize parameter choices and algorithmic complexity. For $f,g:\R\to \R$, we write:
\begin{itemize}
    \item $f(x) = O(g(x))$ if there exist a constant $c > 0$ and $x_0$ such that $|f(x)| \leq c|g(x)|$ for all $x \geq x_0$.
    \item $f(x) = \Theta(g(x))$ if there exist constants $c_1 > 0$, $c_2 > 0$, and $x_0$ such that $c_1|g(x)| \leq |f(x)| \leq c_2|g(x)|$ for all $x \geq x_0$.
    \item $f(x) = \tilde{O}(g(x))$ to denote an asymptotic upper bound that suppresses polylogarithmic factors, meaning $f(x) = O(g(x) \log^k(g(x)))$ for some integer $k \ge 0$.
    \item $f(x) = \tilde{\Theta}(g(x))$ to denote exact asymptotic order up to polylogarithmic factors, meaning $f(x) = \Theta(g(x) \log^k(g(x)))$ for some integer $k \ge 0$.
\end{itemize}

\section{Langevin for non-convex optimization}\label{sec:Langevin}

We begin with the exact-gradient setting. Let
\begin{equation}\label{eq:the_problem}
    \min_{x\in\mathbb{R}^d} F(x)
\end{equation}
be the optimization problem of interest. Throughout this section, \(F\) is assumed to satisfy the smoothness and dissipativity assumptions stated below. For a fixed inverse temperature \(\beta>0\), we denote by $\pi_\beta$ the associated Gibbs measure, defined as in \eqref{eq:GIBBS}. Our goal is to obtain a non-asymptotic bound on the expected excess risk of the Unadjusted Langevin Algorithm, rather than a sampling guarantee for the full distribution of the iterates. The algorithm is the Euler--Maruyama discretization of the overdamped Langevin diffusion and is given by
\[
    x_{k+1}
    =
    x_k-\gamma\nabla F(x_k)
    +
    \sqrt{2\gamma\beta^{-1}}\,z_k,
    \qquad
    z_k\sim\mathcal N(0,I_d).
\]
We denote by \(\mu_k\) the law of \(x_k\). The main result of this section controls $\mathbb{E}[F(x_k)]-\min F$ by combining a KL convergence estimate for ULA with a direct KL-to-objective comparison. This avoids the intermediate Wasserstein step and is the source of the improved dependence on the Log-Sobolev constant.

\begin{algorithm}[H] 
\caption{Standard ULA} \label{alg:ULA}
\begin{algorithmic}
  \STATE{{\bf Input:} $x_0 \sim \mu_0$, $\beta, \gamma \in \mathbb{R}_+$}
     \FOR{$k = 0,1,\dots$}
         \STATE{sample an independent $z_k$ from $\mathcal{N}(0, I_d)$}
         \STATE{$x_{k + 1} = x_k - \gamma \nabla F(x_k) +\sqrt{2\gamma\beta^{-1}} z_k$}
     \ENDFOR
 \end{algorithmic}    
\end{algorithm}

\subsection{Assumptions and preliminary results}

Under the following assumptions, the convergence of the Langevin dynamic and, specifically, bounds for the expected excess risk $\mathbb{E}[F(X_k)] - F^*$ have been analyzed in recent works such as \cite{Raginsky2017, XCZG2018}.

\begin{assumption}[$F$ is smooth]\label{ass:smoothness}
    The function $F$ is $M$-smooth; namely, it is continuously differentiable and its gradient is $M$-Lipschitz continuous:
    \begin{equation*}
        (\forall x,y \in \mathbb{R}^d) \qquad \| \nabla F(x) - \nabla F(y) \| \leq M \| x - y \|.
    \end{equation*}
\end{assumption}

\begin{assumption}[$F$ is dissipative]\label{ass:dissipativity}
    The function $F$ is $(m, b)$-dissipative; that is, for some constants $m, b > 0$,
    \begin{equation*}
        (\forall x \in \mathbb{R}^d) \qquad \langle x, \nabla F(x) \rangle \geq m \| x \|^2 - b.
    \end{equation*}
\end{assumption}

\begin{remark}
    Under both these hypothesis clearly it has to hold $M\geq m$.
\end{remark}

The dissipativity condition is the basic stability assumption in our work. It says that outside a sufficiently large ball the drift $-\nabla F(x)$ points back toward the origin, preventing the Langevin dynamics from escaping to infinity and ensuring confinement of the process. A canonical example is $F(x) = \frac{\lambda}{2}\|x\|^2 + \psi(x)$ where $\lambda > 0$ and $\|\nabla \psi\|_\infty < \infty$, which is dissipative even when $\psi$ is non-convex. Dissipativity is weaker than other geometric assumptions often used in non-log-concave sampling, such as strong convexity outside a compact set, or related contractivity-at-infinity conditions \cite{CCABJ2019,eberle2016reflection,majka2018nonasymptotic}. In contrast, we follow the dissipative framework of \cite{Raginsky2017}. We present here some useful results that apply to functions satisfying the previous Assumptions \ref{ass:smoothness}-\ref{ass:dissipativity}. We state the results for a general function $f$ because we will apply them to different functions in the rest of the paper. The proofs are given in Appendix \ref{app:proof_preliminaries}.

\begin{lemma}\label{lem:quad_growth}
    Let $f$ be a function satisfying Assumption \ref{ass:smoothness} (it is $L$-smooth) and Assumption \ref{ass:dissipativity} (it is $(m,b)$-dissipative). Then $f$ satisfies the following:
    \begin{itemize}
        \item[(i)] $\argmin\limits_{x\in \R^d} f(x)\neq \emptyset$,
        \item[(ii)] for all $x\in \R^d$ it holds $\|\nabla f(x)\|\leq L\|x\| + B$,
        \item[(iii)] for all $x\in \R^d$ it holds $\frac{m}{3}\|x\|^2-\frac{b}{2}\log 3 + C \leq f(x)\leq \frac{L}{2}\|x\|^2 +B\|x\|+A$,
    \end{itemize}
    where $A=|f(0)|$, $B=\|\nabla f(0)\|$ and $C = \min\limits_{x \in \mathbb{R}^d} f(x)$.
\end{lemma}

\begin{corollary}\label{cor:abs-quadratic-bound}
Let $f:\R^d\to\R$ satisfy Assumption~\ref{ass:smoothness} ($L$-smooth) and Assumption~\ref{ass:dissipativity} ($(m,b)$-dissipative). Then for all $x\in\R^d$,
\begin{equation*}
|f(x)|\leq \bar L\|x\|^2 + \bar A    
\end{equation*}
where $\bar L= \frac{L+1}{2}$, $\bar A=A+\frac{B^2}{2}
\;+\;
|C|,$ and $A,B,C$ are the constants from Lemma~\ref{lem:quad_growth}.
\end{corollary}

\begin{remark}\label{rem:diss_minprop}
    Notice that for any $x^* \in \mathbb{R}^d$ such that $\nabla f(x^*) =0$, we have that, by Assumption \ref{ass:dissipativity}, $0\geq m \|x^*\|^2 - b$. In particular, all critical points of $f$ are contained in a ball of radius $\sqrt{\frac{b}{m}}$ and we have $\|\nabla f(0)\|\leq L \|x^*\| \leq L\sqrt{\frac{b}{m}}$.    
\end{remark}

To perform our analysis, we also need the following assumption on the initialization.
\begin{assumption}\label{ass:mu_0}
    The starting distribution $\mu_0$ has a bounded, strictly positive density $p_0$ and satisfies
    \begin{equation*}
        \kappa_0 := \int e^{\|x\|^2} \, d\mu_0(x) < +\infty.
    \end{equation*}
    Consequently, because $\|x\|^2 \le e^{\|x\|^2}$ and $\|x\|^4 \le e^{\|x\|^2}$ for all $x \in \mathbb{R}^d$, the distribution $\mu_0$ has finite second and fourth moments bounded by $\kappa_0$, i.e., $\int \|x\|^2 \, d\mu_0(x) \leq \kappa_0$ and $\int \|x\|^4 \, d\mu_0(x) \leq \kappa_0$.
\end{assumption}
Such an assumption has already been considered in other works (see e.g. \cite{Raginsky2017, XCZG2018}) and it is satisfied, for instance, by assuming $\mu_0$ Gaussian.\\

We finally recall that, under the assumptions of smoothness and dissipativity, the Gibbs measure $\pi_{\beta}$ satisfies a Logarithmic Sobolev Inequality (LSI). While Raginsky et al.\ \cite[Proposition 9]{Raginsky2017} state the LSI result for potentials with merely Lipschitz continuous gradients (i.e., $C^{1,1}$), their proof invokes \cite{Cattiaux2010}, which necessitates a twice continuously differentiable ($C^2$) potential to pointwise bound the Hessian. To bridge this regularity gap, we provide a complete proof via mollification in Appendix \ref{sec:CLS_bounds}.

\begin{proposition}\label{prop:LSI}
    Let the function $F$ satisfy Assumptions \ref{ass:smoothness} and \ref{ass:dissipativity} (specifically, let $F$ be $M$-smooth and $(m,b)$-dissipative). Then, the Gibbs measure $d\pi_\beta(x) = Z^{-1} e^{-\beta F(x)} dx$ satisfies a Logarithmic Sobolev Inequality with constant $\CLS (\beta, d)$. More precisely,
\[
    \operatorname{Ent}_{\pi_\beta}(g^2)
    \leq
    2\CLS(\beta,d)
    \int_{\mathbb{R}^d}\|\nabla g\|^2\,d\pi_\beta,
    \qquad
    g\in H^1(\pi_\beta).
\]
Equivalently, for every probability measure $\mu\ll\pi_\beta$,
\[
    \KL(\mu\|\pi_\beta)
    \leq
    \frac{\CLS(\beta,d)}{2}
    I(\mu\|\pi_\beta),
\]
where $I(\mu\|\pi_\beta)$ is the extended-valued relative Fisher information defined in Appendix~\ref{app:inequalities}.
\end{proposition}

The constant $\CLS(\beta,d)$ governs the convergence of Langevin dynamics to the Gibbs measure. In non-convex landscapes, it may scale exponentially in $\beta$ and $d$, reflecting metastability and the difficulty of crossing energy barriers. Therefore, even reducing the dependence from $\CLS(\beta,d)^3$ to $\CLS(\beta,d)^2$ removes an entire exponential factor. This motivates our focus on tracking the dependence on $\beta$, $d$, and $\CLS(\beta,d)$ explicitly and on mitigating the impact of these exponential factors whenever possible. While many Langevin analyses are formulated as sampling guarantees, the precise dependence on these quantities is not always tracked at the level needed for optimization, nor is it necessarily optimal for objective-value guarantees. Since our goal is optimization rather than sampling, we control the expected excess risk $\mathbb{E}[F(x_k)]-\min F$ directly, avoiding Wasserstein-based intermediate bounds and obtaining sharper dependence on $\CLS(\beta,d)$.

\subsection{Main results}\label{sec:main_ris_ula}

Our main result is the following theorem, together with the correspondent corollary.

\begin{theorem}\label{thm:main_bound}
     Let $F$ satisfying Assumptions \ref{ass:smoothness} and \ref{ass:dissipativity}. Let $\mu_0$ be a starting distribution that satisfies Assumption \ref{ass:mu_0} and let
    \begin{equation*}
        \beta\geq\max\left\{d,\frac{4}{m}\right\}, \quad \gamma \le \frac{1}{4\beta M^2\CLS (\beta, d)}.
    \end{equation*}
    Let $(x_k)_{k \in \mathbb{N}}$ be a sequence of iterates generated by Algorithm \ref{alg:ULA}. Then, we have
    \begin{equation*}
        \begin{aligned}
            \mathbb{E}[F(x_k)] - \min F \leq \ &  C_0(M+1)\left(C_1\beta\,e^{-k\gamma/(2\beta \CLS (\beta,d))} + \sqrt{8\beta M^2 d \gamma \CLS (\beta,d)} + 4 \beta M^2 d\gamma \CLS (\beta,d)\right)\\
        & \hspace{1cm} +\frac{d}{2\beta}\log \left(\frac{e M}{m}\left(\frac{b \beta}{d}+1\right)\right)
         \end{aligned}
    \end{equation*}
    where $C_0$ and $C_1$ (explicit in the proof) are constants independent of $d$ and the algorithm parameters $\beta$ and $\gamma$.
\end{theorem}

The \hyperlink{proof_thm_ULA}{proof} is given in Section \ref{sec:proof_ingredients_ULA}.

\begin{remark}
    Since $M,m$ and $b$ are fixed quantities (once $F$ is fixed), then, under the same hypothesis of the previous theorem we have
    \[\E[F(x_k)]-\min F \leq \bar C \left( \beta e^{-k\gamma / (2 \beta \CLS(\beta,d))} + \sqrt{d\gamma \beta \CLS(\beta,d)}+d\gamma \beta \CLS(\beta,d)+ \frac{d}{\beta} \left(1+\log\left(\frac{\beta}{d}\right)\right)\right),\]
    where $\bar C$ is independent of $\beta, \gamma$ and $d$.
\end{remark}
In the following corollary, we study the complexity of Algorithm \ref{alg:ULA}, namely the number of iterations required to achieve $\E[F(x_k)] - \min F \leq \epsilon$ for $\epsilon \in (0,1)$. To simplify the presentation, we report only the dependence on $d$, $\beta$, $\CLS(\beta, d)$, and $\gamma$ in the choice of the parameters, while explicit constants are provided in the proof.
\begin{corollary}\label{cor:main}
   Under the same assumptions of Theorem~\ref{thm:main_bound}, fix an accuracy $\epsilon \in (0,1)$. Let
    \begin{equation*}
        \beta = \tilde \Theta\left(\frac{d}{\epsilon}\right), \quad \text{and} \quad \gamma= \Theta\left(
\frac{\epsilon^2}{\beta d \CLS (\beta, d)}
\right).
    \end{equation*}
    Let $x_k \in \mathbb{R}^d$ be generated by Algorithm \ref{alg:ULA} at timestep $k$. Fixing 
    \begin{equation}\label{eq:total_count_needed_ULA}
        k = O\left(\frac{\beta}{\gamma} \CLS (\beta, d) \log\left(\frac{\beta d}{\epsilon}\right)\right)= \tilde O\left(\frac{\beta^2 d\CLS (\beta,d)^2}{\epsilon^2}\right),
    \end{equation}
    we have $\E[F(x_k)]-\min F\leq \epsilon$.
    \end{corollary}

The \hyperlink{proof_cor_ULA}{proof} is given in Section \ref{sec:proof_ingredients_ULA}.

\begin{remark}[Probability guarantee]
    Since by definition $F(x) \geq \min F$ for all $x \in \mathbb{R}^d$, the excess risk is a non-negative random variable. Consequently, a bound on the expectation $\mathbb{E}[F(X_k)] - \min F \leq \epsilon$ immediately translates into a guarantee in probability via Markov's inequality. Specifically, for any confidence level $\alpha \in (0,1)$, we have
    \begin{equation*}
        \mathbb{P}\left( F(X_k) - \min F \geq \frac{\epsilon}{\alpha} \right) \leq \alpha.        
    \end{equation*}
\end{remark}
    

\subsection{Comparison with existing works}\label{sec:comparisons_ULA}

We compare Corollary~\ref{cor:main} with the closest existing guarantees for Langevin-based non-convex optimization. Detailed derivations and translations of the rates are deferred to Appendix~\ref{app:comparisons_ULA}. The key feature of our result is the quadratic dependence on \(\CLS(\beta,d)\) for the complexity $k$, together with the tight inverse-temperature calibration discussed in Remark~\ref{rem:exact_choice_of_beta}.

\begin{table}[H]
\centering
\begin{tabular}{l c l}
\toprule
Work & Complexity & Main comparison \\
\midrule
Corollary \ref{cor:main}
&
$\tilde O\!\left(
        \dfrac{\beta^2 d\,\CLS^2}{\epsilon^2}
    \right)$
&
our direct KL-to-objective route \\
Raginsky et al.~\cite{Raginsky2017}
&
$\tilde O\!\left(\dfrac{\beta^7\CLS^5}{\epsilon^4}\right)$
&
Wasserstein-to-objective route \\
\cite{Raginsky2017} refined
&
$\tilde O\!\left(\dfrac{\beta^5\CLS^4}{\epsilon^4}\right)$
&
worse in $\epsilon$ and $\CLS$ \\
Xu et al.~\cite{XCZG2018}
&
--
&
constants not explicit in $\beta,d,\CLS$ \\
Zou et al.~\cite{Zou2021Faster}
&
$\tilde O\!\left(\dfrac{\beta^2d^6\CLS^4}{\epsilon^2}\right)$
&
Cheeger/TV route gives worse $\CLS$ power \\
\midrule
\textbf{Require a larger $\beta$:}\\
Kinoshita--Suzuki~\cite{KinoshitaSuzuki2022}
&
$\tilde O\!\left(\dfrac{\beta^2d\,\CLS^3}{\epsilon}\right)$
&
better in $\epsilon$, worse in $\CLS$\\
Chen et al.~\cite{Chen2024}
&
$\tilde O\!\left(
\max\left\{
d^3 \CLS^3,\,
\dfrac{d^2 \CLS^2}{\epsilon^2}
\right\}
\right)$
&
hitting-time guarantee;  worse in $\CLS$ \\
\bottomrule
\end{tabular}
\caption{Comparison of exact-gradient Langevin optimization guarantees. Here $\CLS=\CLS(\beta,d)$ and polynomial factors and fixed problem-dependent constants are suppressed.}
\label{tab:comparison_exact_ula}
\end{table}

The bounds of \textbf{Raginsky, Rakhlin and Telgarsky~\cite{Raginsky2017}} proceed by first controlling a Wasserstein distance and then converting it into an objective-value estimate. Tracking the dependence on $\beta$, $d$, and $\CLS(\beta,d)$ gives the rate shown in Table~\ref{tab:comparison_exact_ula}; even a refined use of their intermediate estimates still yields a fourth power of $\CLS(\beta,d)$ and a worse dependence on $\epsilon$. The bounds of \textbf{Xu, Chen, Zou and Gu~\cite{XCZG2018}} are less directly comparable. Their rates are expressed in terms of a discrete-time spectral gap \(\lambda(\beta,d)\) and constants coming from geometric-ergodicity and Poisson-equation estimates. These constants are not made explicit in terms of \(\beta\), \(d\), and the structural parameters of the potential. Thus, their bound does not yield a fully explicit parameter calibration comparable to Corollary~\ref{cor:main}. Moreover, statements such as \(\lambda(\beta,d)=e^{-\tilde O(d)}\) are too coarse for our purposes, since different powers of \(\CLS(\beta,d)\) may all be hidden inside the same exponential notation. \textbf{Zou, Xu and Gu~\cite{Zou2021Faster}} obtain bounds in terms of a Cheeger constant. Translating this dependence through the relation between Cheeger, Poincar\'e, and logarithmic Sobolev constants gives a fourth power of $\CLS(\beta,d)$ in our setting. \textbf{Kinoshita and Suzuki~\cite{KinoshitaSuzuki2022}} provide one of the sharpest existing iteration bounds, with better direct dependence on $\epsilon$ but worse dependence on $\CLS(\beta,d)$. Since in global optimization one chooses $\beta=\tilde\Theta(d/\epsilon)$ and $\CLS(\beta,d)$ may depend exponentially on $\beta$, the power of $\CLS(\beta,d)$ is typically the dominant term. Moreover, their argument leads to a larger choice of the inverse temperature $\beta$ in the Gibbs concentration term (more details in Appendix \ref{app:comparisons_ULA}). Since $\CLS(\beta,d)$ may depend exponentially on $\beta$, this worsens the final bound beyond what is visible from the displayed dependence. Finally, \textbf{Chen, Sekhari and Sridharan~\cite{Chen2024}} obtain high-probability hitting-time guarantees based on Lyapunov potentials. Their result is complementary to ours: they control the time to hit an $\epsilon$-sublevel set under an additional self-bounding regularity assumption on the Lyapunov potential, whereas Corollary~\ref{cor:main} controls the expected excess risk of the iterate. Also in this case, their analysis requires a larger choice for $\beta$, which is not desirable, and makes the comparison more difficult.

\subsection{Proof ingredients and proof of the main results}\label{sec:proof_ingredients_ULA}

In order to prove our main result Theorem \ref{thm:main_bound} we make use of the following decomposition:
\begin{equation}\label{eqn:ula_main_bnd}
\begin{aligned}
    \mathbb{E}[F(x_k)] - \min F &= \underbrace{\mathbb{E}[F(x_k)-F(x^{\pi_{\beta}})]}_{(a)} + \underbrace{\mathbb{E}[F(x^{\pi_{\beta}})]-\min F}_{(b)},
    \end{aligned}
\end{equation}
where we write $x^{\pi_\beta}$ for a random vector with distribution $\pi_\beta$. While the treatment of (b) follows from classical concentration of the Gibbs measure \cite[Proposition 11]{Raginsky2017}, (a) is bounded using the convergence of ULA to the Gibbs measure. To control efficiently (a) we use the bound of the relative entropy given by Vempala and Wibisono \cite[Theorem 2]{VempalaWibisono2019} and we derive a new analysis that pass directly from relative entropy to expected objective error. To achieve this goal, we introduce the following crucial lemma. Usually, in the literature, a bound on $W_2$ is first provided and then a bound is derived in the expectations of functions that grow quadratically (see, e.g. \cite{Raginsky2017, XCZG2018}). We show instead that this lemma let us obtain a bound with better dependence on $\CLS(\beta,d)$ compared with the rest of the literature. This KL-to-observable transfer principle is independent of ULA and can be applied whenever a relative-entropy bound and the required exponential integrability are available.

\begin{lemma}[KL-to-observable transfer]\label{lem:BV_new}
    Let $\nu$ be a probability measure on $\mathbb{R}^d$ satisfying $\int \exp(c \|x\|^2) \, d\nu(x) < +\infty$ for some $c>0$. Let $f$ be a measurable function such that $|f(x)|\leq \bar L\|x\|^2 + \bar A$ for some $\bar L, \bar A > 0$. Then, for any probability measure $\mu$ on $\mathbb{R}^d$,
    \begin{equation}
        \left|\int f(x) \, d(\mu-\nu)(x)\right| \leq C_{\nu} \left(\sqrt{\KL(\mu||\nu)}+\frac{1}{2}\KL(\mu||\nu) \right),
    \end{equation}
    where $C_{\nu}=\frac{2 \bar L}{c}\left(\frac{3}{2}+\frac{c\bar A}{\bar L}+\log\int \exp(c\|x\|^2)\, d\nu(x) \right)$.
\end{lemma}

\begin{proof}
    Define $\varphi(x) := \frac{c}{2\bar L} |f(x)|$. Since $|f(x)| \leq \bar L \|x\|^2 + \bar A$, we have
    \[
        2\varphi(x) = \frac{c}{\bar L}|f(x)| \leq c\|x\|^2 + \frac{c \bar A}{\bar L}.
    \]
    Consequently, the exponential moment is bounded by
    \[
        \int e^{2\varphi(x)} \, d\nu(x) \leq \exp\left(\frac{c \bar A}{\bar L}\right)\int \exp(c\|x\|^2)\, d\nu(x) < +\infty.
    \]
    We apply the weighted Csiszár-Kullback-Pinsker inequality \cite[Theorem 2.1(i)]{BolleyVillani}. Note that in the cited theorem, $H(\mu|\nu)$ denotes the Kullback-Leibler divergence $\KL(\mu||\nu)$. We obtain:
    \begin{equation*}
    \begin{aligned}
        \|\varphi (\mu-\nu)\|_{\TV} &\leq \left(\frac{3}{2}+\log \int e^{2\varphi(x)}\, d\nu(x)\right)\left(\sqrt{\KL(\mu||\nu)}+\frac{1}{2}\KL(\mu||\nu) \right) \\
        &\leq \left(\frac{3}{2}+\frac{c\bar A}{\bar L}+\log\int \exp(c\|x\|^2)\, d\nu(x) \right)\left(\sqrt{\KL(\mu||\nu)}+\frac{1}{2}\KL(\mu||\nu) \right).
    \end{aligned}        
    \end{equation*}
    Finally, observing that 
    \begin{equation*}
        \left|\int f \, d(\mu-\nu)\right| \leq \int |f| \, d|\mu-\nu| = \||f|(\mu-\nu)\|_{\TV}= \frac{2\bar L}{c} \|\varphi (\mu-\nu)\|_{\TV},
    \end{equation*}
    we get the claim.
\end{proof}

Since the function $F$ grows quadratically, the following bound holds.

\begin{lemma}[Exponential moment of the target measure]
\label{lem:exp_bound_pi}
Let $F$ satisfy Assumptions \ref{ass:smoothness} and \ref{ass:dissipativity}. For $\beta \ge 4/m$, we have the bound
\begin{equation}
    \log \int_{\mathbb{R}^d} \exp(\|y\|^2) \, d\pi_{\beta}(y) \le \frac{2(d + \beta b)}{\beta m}.
\end{equation}
\end{lemma}
\begin{proof}
    See Appendix \ref{app:proof_exp_bound_pi}.
\end{proof}

The following corollary replaces the standard Otto--Villani and Talagrand route. Since \(\pi_{\beta}\) satisfies a logarithmic Sobolev inequality with constant \(\CLS(\beta,d)\), the Otto--Villani theorem \cite[Theorem~9.6.1]{Bakry2014} implies the Talagrand transportation inequality
$W_2^2(\mu_k,\pi_{\beta}) \leq 2\CLS(\beta,d)\KL(\mu_k\|\pi_{\beta})$ (see \cite[Definition~9.2.2]{Bakry2014}). Combining this estimate with a Wasserstein-to-objective comparison, such as \cite[Lemma~6]{Raginsky2017}, introduces an additional dependence on \(\CLS(\beta,d)\) in the final excess-risk bound. Instead, we apply the weighted Csisz\'ar--Kullback--Pinsker inequality directly to the objective \(F\), obtaining an objective-value bound in terms of \(\KL(\mu_k\|\pi_{\beta})\) and exponential moments of \(\pi_{\beta}\), without passing through \(W_2\).

\begin{corollary}\label{cor:discrete_to_gibbs_bound}
    Let Assumptions \ref{ass:smoothness} and \ref{ass:dissipativity} hold. Let $\beta \geq \max \left\{d,\frac{4}{m}\right\}$
    and $(x_k)_{k \in \mathbb{N}}$ be the sequence generated by Algorithm \ref{alg:ULA}. Then, for every $k \in \mathbb{N}$, we have 
\begin{equation}\label{eq:Expectation_bound_lemma}
        \left|\mathbb{E}[F(x_k) - F(x^{\pi_{\beta}})]\right|=\left|\int F(x) \, d(\mu_k-\pi_{\beta})(x)\right| \leq C_0 (M+1) \left(\sqrt{\KL(\mu_k||\pi_{\beta})}+\frac{1}{2}\KL(\mu_k||\pi_{\beta}) \right),
    \end{equation}
    where $C_0$ is a constant independent of $\beta$ and $d$ (explicit in the proof).
\end{corollary}
\begin{proof}
By Corollary \ref{cor:abs-quadratic-bound}, we have $|F(x)|\leq \bar L\|x\|^2 + \bar A$ with $\bar L = \frac{M+1}{2}$ and $\bar A = |F(0)|+ \frac{\|\nabla F(0)\|^2}{2}+ |\min F|$. Thus, by Lemma \ref{lem:BV_new} (with $c=1$) and Lemma \ref{lem:exp_bound_pi}, we get
\begin{equation}\label{eq:Expectation_bound_lemma}
        \left|\mathbb{E}[F(x_k) - F(x^{\pi_{\beta}})]\right|=\left|\int F(x) \, d(\mu_k-\pi_{\beta})(x)\right| \leq (M+1)\tilde K(\beta, d) \left(\sqrt{\KL(\mu_k||\pi_{\beta})}+\frac{1}{2}\KL(\mu_k||\pi_{\beta}) \right),
\end{equation}
where $\tilde K(\beta, d)=\frac{3}{2}+\frac{2\bar A}{M +1}+ \frac{2(d + \beta b)}{\beta m}$. Since $\beta \geq d$, and recalling the definition of $\bar A$, we have
\begin{equation}\label{eq:C0_def}
\tilde K(\beta, d)\leq  \frac{3}{2}+\frac{2|F(0)|+ \|\nabla F(0)\|^2  +2 |\min F|}{M+1}+ \frac{2 + 2 b}{ m}=: C_0.
\end{equation}
\end{proof}

At this point, is only left to bound the KL divergence between $\mu_k$ and $\pi_{\beta}$. To do this, we make use of the following result, which is a direct consequence of \cite[Theorem 2]{VempalaWibisono2019}. We believe this result leads to the best estimate in the literature, but other controls can be used, for example adaptations of \cite[Lemma 7]{Raginsky2017} combined with \cite[Equation (3.17)]{Raginsky2017}.

\begin{proposition}\label{prop:VW} Suppose that Assumptions \ref{ass:smoothness} and \ref{ass:dissipativity} hold. Let $(x_k)_{k \in \mathbb{N}}$ be the sequence generated by Algorithm \ref{alg:ULA}, and denote by $(\mu_k)_{k \in \mathbb{N}}$ the corresponding sequence of probability laws. Let $\pi_{\beta}$ be the Gibbs distribution related to $F$. Then, for $\gamma \leq \frac{1}{4 \beta M^2\CLS (\beta, d)}$, we have
\begin{equation*}
    \KL(\mu_k||\pi_{\beta})\leq e^{-k \gamma/(\beta \CLS (\beta, d)) } \KL(\mu_0||\pi_{\beta})+ 8 d \gamma \beta M^2\CLS (\beta, d).
\end{equation*}
\end{proposition}
\begin{proof}
By setting $\tilde \gamma = \gamma \beta^{-1}$ we can rescale the $x_k$ update into
\begin{equation}\label{eq:rescaled_ULA-Fh}
    x_{k + 1} = x_k - \tilde \gamma \beta  \nabla F(x_k) +\sqrt{2\tilde \gamma} z_k \qquad z_k \sim\mathcal{N}(0, I_{d}).\end{equation}
    Notice that $\mu_k$ is generated by the algorithm in \cite{VempalaWibisono2019} applied to $\beta F$ which is $(\beta M)$-smooth. By Proposition \ref{prop:LSI}, we have that $\pi_{\beta}$ satisfies the log-Sobolev inequality with constant $\CLS(\beta, d)$. Thus, under the condition $\tilde \gamma \leq \frac{1}{4 \beta^2 M^2 \CLS(\beta, d) }$, by  \cite[Theorem 2]{VempalaWibisono2019}, 
    we have that
    \begin{equation*}
    \KL(\mu_k||\pi_{\beta})\leq e^{-k\tilde \gamma/\CLS(\beta, d) } \KL(\mu_0||\pi_{\beta})+ 8 d \tilde \gamma \beta^2 M^2 \CLS(\beta, d).    
    \end{equation*}    
    Substituting back $\tilde \gamma = \gamma \beta^{-1}$, we obtain the claim.
\end{proof}

Here we also provide a bound on the initialization $\KL(\mu_0||\pi_{\beta})$ by fixing the initialization $\mu_0$ to satisfy the moment bound given in Assumption \ref{ass:mu_0}.

\begin{lemma}[Initialization bound $\KL(\mu_0\|\pi_{\beta})$]
\label{lem:KL_mu0_pi}
Assume Assumptions~\ref{ass:smoothness} and \ref{ass:dissipativity}. Suppose $\beta\ge d$.
Let $\mu_0$ satisfy Assumption \ref{ass:mu_0}. Then
\begin{equation}\label{eq:KL_factored}
\KL(\mu_0\|\pi_{\beta})\;\le\;\tilde C \beta,
\end{equation}
where $\tilde C$ is explicit expression in the proof.
\end{lemma}
\begin{proof}
    See Appendix \ref{app:proof_KL_mu0_pi}.
\end{proof}

A subtlety often left implicit in Langevin analyses,
including \cite{Raginsky2017}, is that initialization constants such as $\widetilde C$ may themselves depend on $d$. Although Lemma~\ref{lem:KL_mu0_pi} applies to any $\mu_0$ satisfying
Assumption~\ref{ass:mu_0}, one can make $\widetilde C$ independent of
both $d$ and $\beta$ by choosing $\mu_0$ properly, as for example $\mu_0=\mathcal N(0,\tau d^{-1}I_d)$ with $0<\tau<1/2$; see Remark~\ref{rem:gaussian_initialization}. Accordingly, our complexity
bounds treat $\widetilde C$ as a dimension- and temperature-independent constant. This follows the convention adopted in prior work, while here being explicitly justified by an admissible choice of initialization.

\paragraph{Proof of Theorem \ref{thm:main_bound} and Corollary \ref{cor:main}.}

\begin{proof}[\hypertarget{proof_thm_ULA}{Proof} of Theorem \ref{thm:main_bound}]
\textit{Bound of $(a)$ in \eqref{eqn:ula_main_bnd}:} By Corollary \ref{cor:discrete_to_gibbs_bound}, we have
\begin{equation*}
    |(a)| \leq C_0 (M+1)\left(\sqrt{\KL(\mu_k||\pi_{\beta})}+\frac{1}{2}\KL(\mu_k||\pi_{\beta}) \right).
\end{equation*}
By Proposition \ref{prop:VW}, we have
\begin{equation*}
    \KL(\mu_k||\pi_{\beta})\leq e^{-k \gamma/(\beta \CLS(\beta, d)) } \KL(\mu_0||\pi_{\beta})+ 8 d \gamma \beta M^2 \CLS (\beta, d).
\end{equation*}
Therefore, we get
\begin{equation*}
    \begin{aligned}
        |(a)| &\leq C_0 (M+1)\bigg( \sqrt{e^{-k \gamma/(\beta \CLS(\beta, d)) } \KL(\mu_0||\pi_{\beta})+ 8 d \gamma \beta M^2 \CLS (\beta, d)}\\
        &+\frac{1}{2} \bigg( e^{-k \gamma/(\beta \CLS(\beta, d)) } \KL(\mu_0||\pi_{\beta})+ 8 d \gamma \beta M^2 \CLS (\beta, d) \bigg)\bigg).
    \end{aligned}
\end{equation*}
Since for all $t_1,t_2 \geq 0$, $\sqrt{t_1+t_2}\leq \sqrt{t_1} + \sqrt{t_2}$, we have
\begin{equation*}
    \begin{aligned}
        |(a)| &\leq C_0 (M+1)\bigg( e^{ - k \gamma/(2\beta \CLS(\beta, d)) }\sqrt{ \KL(\mu_0||\pi_{\beta})} + \frac{1}{2} e^{-k \gamma/(\beta \CLS(\beta, d)) } \KL(\mu_0||\pi_{\beta})\\
        &+ \sqrt{8 d \gamma \beta M^2 \CLS (\beta, d)} +4 d \gamma \beta M^2 \CLS (\beta, d) \bigg).
    \end{aligned}
\end{equation*}
Since for every $\xi > 0$,  $e^{-\xi}\leq e^{-\xi/2}$, we have
\begin{equation*}
    \begin{aligned}
        |(a)| &\leq C_0 (M+1)\bigg( e^{ - k \gamma/(2\beta \CLS(\beta, d)) }\left(\sqrt{ \KL(\mu_0||\pi_{\beta})} + \frac{1}{2}  \KL(\mu_0||\pi_{\beta}) \right)\\
        &+ \sqrt{8 d \gamma \beta M^2 \CLS (\beta, d)} +4 d \gamma \beta M^2 \CLS (\beta, d) \bigg).
    \end{aligned}
\end{equation*}
Since $\beta \geq d \geq 1$, by Lemma \ref{lem:KL_mu0_pi}, we have $\sqrt{\KL(\mu_0||\pi_{\beta})} + \KL(\mu_0||\pi_{\beta})/2\leq \beta \left(\sqrt{\tilde C} + \tilde C/2\right)$. Let 
\begin{equation}\label{eq:C1_def}
    C_1:=\left(\sqrt{\tilde C} + \tilde C/2\right).
\end{equation}
Hence, we have 
\begin{equation}\label{eqn:ula_main_bnd_1}
    \begin{aligned}
        |(a)| &\leq C_0 (M+1)\bigg( C_1 \beta e^{ - k \gamma/(2\beta \CLS(\beta, d)) }+ \sqrt{8 d \gamma \beta M^2 \CLS (\beta, d)} +4 d \gamma \beta M^2 \CLS (\beta, d) \bigg).
    \end{aligned}
\end{equation}
\textit{Bound of $(b)$ in \eqref{eqn:ula_main_bnd}:} Since $\beta \geq \frac{4}{m} \geq \frac{2}{m}$, by \cite[Proposition 11]{Raginsky2017}, we have
\begin{equation}\label{eqn:ula_main_bnd_2}
    |(b)|\leq \frac{d}{2\beta}\log \left(\frac{e M}{m}\left(\frac{b \beta}{d}+1\right)\right).
\end{equation}
Combining \eqref{eqn:ula_main_bnd_1} and \eqref{eqn:ula_main_bnd_2} with \eqref{eqn:ula_main_bnd} yields the claim.
\end{proof}

\begin{proof}[\hypertarget{proof_cor_ULA}{Proof} of Corollary \ref{cor:main}] First of all, set $\epsilon_1, \epsilon_2, \epsilon_3 > 0$, with $\sum_{i=1}^3 \epsilon_i = \epsilon$. We advise choosing $\epsilon_1 \approx \epsilon$, since this will fix $\beta$ and the $\CLS(\beta, d)$ constant can depend badly on $\beta$. However, we leave the choice to the reader. For theoretical purposes we fix $\epsilon_1=\epsilon_2=\epsilon_3 = \frac{\epsilon}{3}$.\\
\textit{\underline{Choice of $\beta$}:} We choose $\beta$ as tightly as possible here, because any unnecessary increase in the inverse temperature may translate into an exponential deterioration of $\CLS(\beta,d)$; see Remark~\ref{rem:exact_choice_of_beta}. Recall that we need to choose $\beta$ so as to control the term
\begin{equation*}
T(\beta) \;:=\; \frac{d}{2\beta}\log \left(\frac{e M}{m}\left(\frac{b \beta}{d}+1\right)\right).    
\end{equation*}
Define $\theta := eM(1+b)/m$. Note that since $M \ge m$ and $b \ge 0$, we have $\theta \ge e$. 
Assuming $\epsilon_1 \le 1$, we choose
\[
\beta = \max\left\{d, \frac{4}{m}, \frac{d}{2\epsilon_1}\left( \log\left(\frac{\theta}{\epsilon_1}\right) + 2\log\left(\log\left(\frac{\theta e^2}{\epsilon_1}\right)\right) \right) \right\},
\]
so that $\beta$ is also in the hypothesis of Theorem \ref{thm:main_bound}. Since $\beta\ge d$, we can upper bound the argument of the logarithm$$\frac{b\beta}{d}+1 \;\le\; \frac{b\beta}{d}+\frac{\beta}{d} \;=\; (1+b)\frac{\beta}{d}.$$Therefore, we obtain the bound\begin{equation}\label{eq:inproof_epsilon_1}T(\beta) \;\le\; \frac{d}{2\beta}\log\left(\frac{eM}{m}(1+b)\frac{\beta}{d}\right) \;=\; \frac{d}{2\beta}\log\left(\theta\frac{\beta}{d}\right).\end{equation}Set $x := \beta/d$. Since $\beta \ge d$, then $x \ge 1$. Equation \eqref{eq:inproof_epsilon_1} then becomes $T(\beta) \le \frac{1}{2x}\log(\theta x)$. Consider the function
\begin{equation*}
    f(x) := 2\epsilon_1 x - \log(\theta x), \qquad x>0.
\end{equation*}
We want to show that $f(x) \geq 0$. The derivative of $f$ is $f'(x) = 2\epsilon_1 - 1/x$, which is non-negative for $x \ge \frac{1}{2\epsilon_1}$. Let $L := \log\left(\frac{\theta}{\epsilon_1}\right)$. Because $\theta \ge e$ and $\epsilon_1 \le 1$, we have $L \ge 1$. 
By our choice of $\beta$, the term $x$ satisfies $x \ge x_0 := \frac{1}{2\epsilon_1}\left( L + 2\log(L+2) \right)$.
Since $L \ge 1$, it is clear that $x_0 \ge \frac{1}{2\epsilon_1}$. Thus, $x$ lies in the region where $f$ is non-decreasing. It is thus sufficient to check that $f(x_0) \ge 0$. We compute
\begin{align*}
f(x_0) 
&= L + 2\log(L+2) - \log\left(\theta \cdot \frac{1}{2\epsilon_1} (L + 2\log(L+2))\right) \\
&= L + 2\log(L+2) - \left( \log\left(\frac{\theta}{\epsilon_1}\right) - \log 2 + \log(L + 2\log(L+2)) \right) \\
&= 2\log(L+2) + \log 2 - \log(L + 2\log(L+2)).
\end{align*}
We can rewrite the positive terms as $2\log(L+2) + \log 2 = \log(2(L+2)^2) = \log(2L^2 + 8L + 8)$. 
Furthermore, since $\log(L+2) \le L+1$ for $L \ge 1$, we can upper bound the argument of the negative logarithm:
\begin{equation*}
L + 2\log(L+2) \;\le\; L + 2(L+1) \;=\; 3L + 2.    
\end{equation*}
Because $2L^2 + 8L + 8 > 3L + 2$ strictly for all $L > 0$, it immediately follows that
\begin{equation*}
\log(2(L+2)^2) \;>\; \log(L + 2\log(L+2)).    
\end{equation*}
Therefore, $f(x_0) > 0$. Since $f(x) \ge f(x_0) > 0$, we have established that $\log(\theta x) \le 2\epsilon_1 x$. Substituting this back into \eqref{eq:inproof_epsilon_1} yields
\begin{equation*}
T(\beta) \;\le\; \frac{1}{2x}\log(\theta x) \;\le\; \frac{1}{2x}(2\epsilon_1 x) \;=\; \epsilon_1,    
\end{equation*}
which proves the desired bound.\\
\textit{\underline{Choice of $\gamma$}:} Since $\beta$ is now fixed as above, we have that from now on $\CLS (\beta,d)$ is also fixed. We can choose now     
\begin{equation*}
\gamma = \frac{\epsilon_2^2}{32 \beta M^2(M+1)^2 C_0^2 d \CLS (\beta, d)},    
\end{equation*}
with $C_0$ defined as in \eqref{eq:C0_def}. Since $C_0$ defined in \eqref{eq:C0_def} is greater than one, we have in particular that $\gamma \leq \frac{1}{4\beta M^2 \CLS (\beta, d)}$, so $\gamma$ is under the hypothesis of Theorem \ref{thm:main_bound}. Moreover, it holds $8d\gamma\beta M^2 \CLS (\beta,d)\leq 1$, and in this regime we have
\begin{equation*}
    4d\gamma \beta M^2 \CLS (\beta,d)\leq \sqrt{8d\gamma\beta M^2 \CLS (\beta,d)}.
\end{equation*}
Therefore
\begin{equation*}
    C_0 (M+1) \left(\sqrt{8d\gamma\beta M^2 \CLS (\beta,d)} +4d\gamma \beta M^2 \CLS (\beta,d)\right)\leq 2 M (M+1) C_0 \sqrt{8d\gamma\beta \CLS (\beta,d)} \leq \epsilon_2.
\end{equation*}
\textit{\underline{Choice of $k$}:} By choosing 
\begin{equation*}
    k \ge \frac{2\beta \CLS(\beta, d)}{\gamma}
\log\!\left(
\frac{C_0 C_1(M+1)\beta
}{\epsilon_3}
\right),
\end{equation*}
with $C_0$ defined as in \eqref{eq:C0_def} and with $C_1$ defined as in \eqref{eq:C1_def}, we have
\[\bigl(M+1\bigr) C_0 C_1 \beta e^{-k\gamma/(2\beta \CLS (\beta,d))} \leq \epsilon_3.\]
\end{proof}

\begin{remark}\label{rem:exact_choice_of_beta}
    Notice that not only we choose $\beta$ of the order of $\frac{d}{\epsilon}$, but our choice was $\beta = \frac{d}{2\epsilon_1}\tilde \Theta \left(\log \left(\frac{d}{\epsilon_1}\right)\right)$, with possibly $\epsilon_1 \approx \epsilon$. This is the best choice of $\beta$ possible since even for quadratic functions one has to choose $\frac{d}{2\epsilon_1}$ to get concentration and to have $\E[F(x^\pi_{\beta})]-\min F \leq \epsilon_1$, which is an unavoidable term. As discussed in Section \ref{sec:comparisons_ULA} (and in more details Appendix \ref{app:comparisons}) this tight choice distinguishes our work from some others. We recall that larger choices of $\beta$ can result in a much worse $\CLS(\beta,d)$ constant, since this usually depends exponentially on $\beta$.
\end{remark}

\section{Inexact ULA}\label{sec:Inexact}

We now move from exact-gradient ULA to the case where the algorithm only has access to an approximate gradient. We consider again problem~\eqref{eq:the_problem} under Assumptions~\ref{ass:smoothness}--\ref{ass:dissipativity}, and replace $\nabla F(x)$ by a possibly biased random surrogate $g(x,\xi)$, where $\xi\sim \mu_{\xi}$. The resulting inexact ULA scheme is
\[
    x_{k+1}
    =
    x_k-\gamma g(x_k,\xi_k)
    +
    \sqrt{2\gamma\beta^{-1}}z_k,
    \qquad
    z_k\sim\mathcal N(0,I_d).
\]
This framework covers stochastic and mini-batch gradients, deterministic numerical errors, approximate inner computations, and the zeroth-order finite-difference estimators considered later.

The central assumption is that the mean-square error of the surrogate grows at most quadratically with the state. This is natural under our standing assumptions, since Lemma~\ref{lem:quad_growth} implies that $\|\nabla F(x)\|$ grows at most linearly in $\|x\|$. Our analysis quantifies how such inexactness propagates along the Langevin dynamics. We first prove a uniform second-moment bound for the inexact iterates, then compare the inexact chain with the exact ULA chain driven by the same Gaussian noise, and finally combine this perturbation estimate with the exact-ULA result of Section~\ref{sec:Langevin}. The resulting bound separates the Langevin discretization error, the Gibbs concentration error, and the gradient-approximation error, without requiring the estimator $g(x,\xi)$ to be unbiased.

\subsection{Assumptions and preliminary results}

\begin{assumption}[Controlled growth of the gradient error]\label{ass:gradient_error}
    We assume that there exist non-negative constants $P$ and $Q$ and a precision parameter $\delta > 0$ such that for all $x\in \R^d$ it holds
\begin{equation}\label{eq:gradient_noise_bound}
        \mathbb{E}\left[ \|g(x, \xi) - \nabla F(x)\|^2 \right] \leq \delta \left( P\|x\|^2 + Q \right).
    \end{equation}
    Notice that we do not ask the estimator $g$ to be unbiased.
\end{assumption}

When $P=0$ conditions involving upper bounds proportional to $1/P$ are interpreted as void in the following.

\begin{remark}[Examples covered by Assumption~\ref{ass:gradient_error}]
Assumption~\ref{ass:gradient_error} is a gradient-error analogue of the classical Blum--Gladyshev growth conditions in stochastic approximation~\cite{Blum1954,Gladyshev1965}: the oracle error is allowed to have state-dependent second moment, with at most quadratic growth in $\|x\|$, while the standard bounded-variance case is recovered when $P=0$. This is natural in stochastic-gradient and zeroth-order settings, where the estimator variance may grow with $\|\nabla F(x)\|$ or with the objective value. In particular, it is compatible with the ABC-type conditions used in stochastic optimization~\cite{Gower2021,JunYe24}, for instance bounds of the form
\[
    \mathbb{E}\!\left[\|g(x,\xi)-\nabla F(x)\|^2\right]
    \leq
    \mathrm A\big(F(x)-\min F\big)
    +\mathrm B\|\nabla F(x)\|^2
    +\mathrm C .
\]
Indeed, Lemma~\ref{lem:quad_growth} implies that both $F(x)-\min F$ and $\|\nabla F(x)\|^2$ have at most quadratic growth in $\|x\|$, so such ABC bounds imply \eqref{eq:gradient_noise_bound}. The assumption also covers deterministic biased gradients, e.g. $g(x,\xi)=\nabla F(x)+r(x)$ with $\|r(x)\|^2\leq \delta(P\|x\|^2+ Q)$, as well as mixed stochastic and biased estimators $g(x,\xi)=g_{\mathrm{ex}}(x,\xi)+r(x)$, where $g_{\mathrm{ex}}$ is unbiased and $r$ is a bias term. In finite-sum problems, $g_{\mathrm{ex}}$ can be a single-sample gradient $\nabla f_{i(\xi)}(x)$ or a mini-batch version. Thus, $\delta$ represents the precision of the approximation, such as a squared bias level or an inverse batch size. 
\end{remark}

\begin{algorithm}[H] 
\caption{Inexact ULA} \label{alg:I-ULA}
\begin{algorithmic}
  \STATE{{\bf Input:} $x_0 \sim \mu_0$, $\beta, \gamma \in \mathbb{R}_+$}
     \FOR{$k = 0, 1, \dots$}
        \STATE{sample an independent $\xi_k$ from $\mu_\xi$}
         \STATE{sample an independent $z_k$ from $\mathcal{N}(0, I_{d})$}
         \STATE{$x_{k + 1} = x_k - \gamma g(x_k,\xi_k) +\sqrt{2\gamma\beta^{-1}} z_k$}
     \ENDFOR
 \end{algorithmic}    
\end{algorithm}

We first prove that the moments of the iterates remain bounded. The result relies on the interplay between the dissipativity of $F$ and the growth of the noise variance.

\begin{proposition}[Uniform bound on iterate second moment]\label{prop:iter-norm-bound-final}
Let Assumptions \ref{ass:smoothness}, \ref{ass:dissipativity}, and \ref{ass:gradient_error} hold and let $\mu_0$ satisfying Assumption \ref{ass:mu_0}. Assume moreover that
\[
\gamma \le \min\left\{1,\ \frac{m}{8 M^2}\right\}, \qquad \delta \le \frac{m^2}{4P}.
\]
Then $\E[\|x_k\|^2]$ is uniformly bounded and if $\beta \geq d$ then we have $\sup_{k\ge 0}\mathbb{E}\big[\|x_k\|^2\big]
\le  \Gamma$, with $\Gamma$ independent on $\gamma$, $\beta$, $\delta$ and $d$, and explicit in the proof.
\end{proposition}
\begin{proof}
    See Appendix \ref{app:proof_moment_bounds_IULA}.
\end{proof}

\subsection{Main Results}

\begin{theorem}\label{thm:main_bound_I}
    Let $F$ satisfy Assumption~\ref{ass:smoothness} (it is $M$-smooth) and Assumption~\ref{ass:dissipativity}
    (it is $(m,b)$-dissipative). Assume moreover to have a surrogate $g$ of the gradient satisfying Assumption \ref{ass:gradient_error}. Let $\mu_0$ satisfy Assumption \ref{ass:mu_0} and let \[
\beta\geq\max\left\{d, \frac{8}{m}\right\}, \quad \gamma \le \min\left\{1,\frac{1}{m},\frac{m}{8M^2},\frac{1}{4\beta M^2\CLS (\beta, d)}\right\}, \quad \delta \leq \frac{m^2}{4P}.
\] 
Then, we have the following bound for the sequence generated by Algorithm \ref{alg:I-ULA}
    \begin{equation*}
        \begin{aligned}
            \mathbb{E}[F(x_k)] - \min F \leq \ &   C_0'(M+1)\left(\sqrt{C_2 \gamma k \beta \delta} + \frac{C_2\gamma k \beta \delta}{2}\right) +  \\
        & +C_0(M+1)\left(C_1\beta e^{-k\gamma/(2\beta \CLS(\beta,d))} + \sqrt{8\beta M^2d\gamma \CLS(\beta,d)}+4\beta M^2 d\gamma \CLS(\beta,d)\right)\\
        & +\frac{d}{2\beta}\log \left(\frac{e M}{m}\left(\frac{b \beta}{d}+1\right)\right)
         \end{aligned}
    \end{equation*}
    where $C_0, C_0', C_1$ and $C_2$ are constants independent on $\beta, d,\gamma$ and $\delta$ and explicit in the proof.
\end{theorem}

The \hyperlink{proof_thm_IULA}{proof} is given in Section \ref{sec:proof_ingredients_IULA}.

\begin{remark}
    Since $M,m$ and $b$ are fixed quantities (once $F$ is fixed), then, under the same hypothesis of the previous theorem we have
    \[\begin{aligned}
        \E[F(x_k)]-\min F \leq & \bar C_{\I} \Bigg(\sqrt{\gamma k \beta \delta} + \gamma k \beta \delta +\beta e^{-k\gamma / (2 \beta \CLS(\beta,d))} \\ & \hspace{1cm}+ \sqrt{d\gamma \beta\CLS(\beta,d)}+d\gamma \beta  \CLS(\beta,d) + \frac{d}{\beta} \left(1+\log\left(\frac{\beta}{d}\right)\right)\Bigg),
        \end{aligned}\]
    where $\bar C_{\I}$ is independent of $\beta, \gamma, \delta$ and $d$.
    \end{remark}

In the following corollary, we study the complexity of Algorithm \ref{alg:I-ULA}, namely the number of iterations required to achieve $\E[F(x_k)] - \min F \leq \epsilon$ for $\epsilon \in (0,1)$. To simplify the presentation, we report only the dependence on $d$, $\beta$, $\CLS(\beta, d)$, and $\gamma$ in the choice of the parameters, while explicit constants are provided in the proof.

\begin{corollary}\label{cor:main_I}
   Under the same assumptions of Theorem~\ref{thm:main_bound_I}, fix an accuracy $\epsilon \leq 1$. Let
    \begin{equation*}
        \beta = \tilde \Theta\left(\frac{d}{\epsilon}\right),\quad \gamma= \Theta\left(
\frac{\epsilon^2}{\beta d \CLS (\beta, d)}
\right), \quad \text{and}\quad \delta = \tilde \Theta \left(\frac{\epsilon^2}{\beta^2 \CLS(\beta,d)}\right).
    \end{equation*}
    Let $x_k \in \mathbb{R}^d$ be generated by Algorithm \ref{alg:I-ULA} at timestep $k$. Fixing 
    \begin{equation}\label{eq:total_count_needed_I}
        k = O\left(\frac{\beta}{\gamma} \CLS (\beta, d) \log\left(\frac{\beta d}{\epsilon}\right)\right)= \tilde O\left(\frac{\beta^2 d\CLS (\beta,d)^2}{\epsilon^2}\right),
    \end{equation}
    we have $\E[F(x_k)]-\min F\leq \epsilon$.
\end{corollary}

The \hyperlink{proof_cor_IULA}{proof} is given in Section \ref{sec:proof_ingredients_IULA}.

\subsection{Comparison with existing works}\label{sec:comparisons_I}

We compare Corollary~\ref{cor:main_I} with the closest existing guarantees for inexact or stochastic-gradient Langevin methods. Detailed derivations and translations of the rates are deferred to Appendix~\ref{app:comparisons_ULA}. The key feature of our result is the quadratic dependence on \(\CLS(\beta,d)\) for $k$ and inverse linear dependence for the choice of $\delta$, together with the tight inverse-temperature calibration discussed in Remark~\ref{rem:exact_choice_of_beta}.

\begin{table}[H]
\centering
\begin{tabular}{l c c l}
\toprule
Work & Iterations $k$ & Precision $\delta$ & Main comparison \\
\midrule
Corollary~\ref{cor:main_I}
&
$\tilde O\!\left(\dfrac{\beta^2 d\,\CLS^2}{\epsilon^2}\right)$
&
$\tilde\Theta\!\left(\dfrac{\epsilon^2}{\beta^2\CLS}\right)$
&
our direct KL-to-objective route \\
Raginsky et al.~\cite{Raginsky2017}
&
$\tilde O\!\left(\dfrac{\beta^7\CLS^5}{\epsilon^4}\right)$
&
$\tilde\Theta\!\left(\dfrac{\epsilon^4}{\beta^6\CLS^4}\right)$
&
Wasserstein-to-objective route \\
\cite{Raginsky2017} refined
&
$\tilde O\!\left(\dfrac{\beta^5\CLS^4}{\epsilon^4}\right)$
&
$\tilde\Theta\!\left(\dfrac{\epsilon^4}{\beta^4\CLS^3}\right)$
&
still worse in $\epsilon$ and $\CLS$ \\
Xu et al.~\cite{XCZG2018}
&
--
&
--
&
constants not explicit in $\beta,d,\CLS$  \\
\bottomrule
\end{tabular}
\caption{Comparison of inexact-gradient Langevin optimization guarantees. Here $\CLS=\CLS(\beta,d)$; polynomial factors in $d$ and fixed problem-dependent constants are suppressed, except in our displayed bound.}
\label{tab:comparison_inexact_ula}
\end{table}

The comparison shows the same phenomenon as in the exact-gradient case. The standard route of \textbf{Raginsky, Rakhlin and Telgarsky~\cite{Raginsky2017}} first controls a Wasserstein distance and then converts it into objective-value error, leading to worse powers of both $\epsilon$ and $\CLS(\beta,d)$. A more refined use of their intermediate estimates, combined with the Bolley--Villani inequality~\cite{BolleyVillani}, improves the dependence but still yields a fourth power of $\CLS(\beta,d)$ and a more restrictive admissible precision $\delta$. \textbf{Xu, Chen, Zou and Gu~\cite{XCZG2018}} are closely related in spirit, since they also compare a stochastic chain with an ideal Langevin chain, but their rates are expressed through a discrete-time spectral gap and geometric-ergodicity constants which are not explicit in $\beta$, $d$, and the structural parameters of $F$. Moreover, their proof follows the standard route through Wasserstein-type controls. In contrast, Corollary~\ref{cor:main_I} gives explicit choices of $\gamma$, $k$, and $\delta$, and keeps the quadratic dependence on $\CLS(\beta,d)$ visible throughout.

\subsection{Proof ingredients and proof of the main results}\label{sec:proof_ingredients_IULA}
We recall the ULA defined in Section \ref{sec:Langevin}
\begin{equation}\label{eq:ULA-F}
y_{k + 1} = y_k - \gamma \nabla F(y_k) +\sqrt{2\gamma\beta^{-1}} z_k \qquad z_k \sim\mathcal{N}(0, I_{d}), \ y_0\sim \mu_0,\end{equation}
for which we use here the notation $\{y_k\}_k$ to avoid confusion with the sequence $\{x_k\}_k$ generated by Algorithm \ref{alg:I-ULA}. We denote again by $\pi_{\beta}$ the Gibbs distribution for $F$, i.e. $\pi_{\beta}= e^{-\beta F}/Z$, we denote the law of the variable $x_k$ by $\mu_k$ and the law of the variable $y_k$ as $\nu_k$. For the analysis of the inexact method we decompose the error in the following way
\begin{equation}\label{eq:decompose_I}
\begin{aligned}
    \mathbb{E}[F(x_k)] - \min F &= \underbrace{\mathbb{E}[F(x_k) - F(y_k)]}_{(a)} +\underbrace{\mathbb{E}[F(y_k)-F(x^{\pi_{\beta}})]}_{(b)} + \underbrace{\mathbb{E}[F(x^{\pi_{\beta}})]-\min F}_{(c)}.
    \end{aligned}
\end{equation}
We estimate $(b)$ and $(c)$ as we did in Section \ref{sec:Langevin}, while to estimate $(a)$ we use Lemma \ref{lem:BV_new} to control the expectation with the KL divergence between $\mu_k$ and $\nu_k$, which we then control with a result following from \cite[Lemma 4.4]{XCZG2018} and \cite[Lemma 7]{Raginsky2017}. This strategy is enabled by a uniform bound on the exponential moments of the sequence of distributions $\nu_k$ generated by the ULA algorithm. Crucially, unlike other results in the literature (e.g., \cite{Raginsky2017, XCZG2018}), our bound is independent of the time index $k$. Standard analyses often employ bounds that grow linearly with $k$ (see, for instance, \cite[Lemma C.7]{XCZG2018}). While such bounds are sufficient for finite-time horizon analysis, they result in suboptimal dependencies on the Log-Sobolev constant $C_{\text{LS}}(\beta, d)$ in the final convergence rates. Uniform bounds could principally be derived by combining results such \cite[Proposition 8]{DurmusMoulines2017} and \cite[Lemma 1]{DurmusMoulines2017}, but we provide a self-contained proof with explicit constants to ensure explicit dependence on problem parameters.

\begin{lemma}\label{lem:exp_bound}
Let $F$ satisfy Assumptions \ref{ass:smoothness} and \ref{ass:dissipativity}. For 
\begin{equation}\label{eq:conditions_uniform_bound}
    \beta\geq8/m, \quad \text{and} \quad\gamma \le \min\left\{1,\frac{1}{m},\frac{m}{8 M^2}\right\},
\end{equation}
we have the uniform bound
    \[\log \int \exp{(\|y\|^2)} \, d\nu_k(y) \leq \log \kappa_0 + \frac{8d+8\beta b+8\beta B^2}{m\beta},\]
    where we recall that $\kappa_0 = \int \exp(\|y\|^2) \, d\mu_0$ and $B=\|\nabla F(0)\|$.
\end{lemma}
\begin{proof}
    See Appendix \ref{app:further_proofs_lem}.
\end{proof}

Combining the uniform exponential-moment bound of Lemma~\ref{lem:exp_bound} with Lemma~\ref{lem:BV_new} gives the following comparison between the inexact and exact chains. This retains a square-root dependence on the KL divergence. By contrast, first applying the Bolley--Villani inequality~\cite[Corollary~2.3]{BolleyVillani} as in~\cite{Raginsky2017} and then using the quadratic-growth comparison $\left|\E[F(x_k)]-\E[F(y_k)]\right|
\le C W_2(\mu_k,\nu_k)$ as in~\cite[Lemma~6]{Raginsky2017}, would produce a $\KL(\mu_k|\nu_k)^{1/4}$ dependence and hence worse final rates, particularly in the dependence on $\CLS(\beta,d)$.

\begin{lemma}\label{cor:discretes_bound_ST}
For \[\beta\geq \max\left\{d,\frac{8}{m}\right\}, \qquad \gamma \le \min\left\{1,\frac{1}{m},\frac{m}{8 M^2}\right\},\] we have
    \begin{equation}
        \left|\mathbb{E}[F(x_k) - F(y_k)]\right|=\left|\int F(x) \, d(\mu_k-\nu_k)(x)\right| \leq (M+1) C_0'  \left(\sqrt{\KL(\mu_k||\nu_k)}+\frac{1}{2}\KL(\mu_k||\nu_k) \right),
    \end{equation}
    where $C_0'$ independent on $\gamma$, $\beta$, $\delta$ and $d$, and explicit in the proof.
\end{lemma}
\begin{proof}
This follows the same lines as the proof of Corollary \ref{cor:discrete_to_gibbs_bound}. In particular, we have by \ref{cor:abs-quadratic-bound}, that $|F(x)|\leq \bar L\|x\|^2 + \bar A$ with $\bar L = \frac{M+1}{2}$ and $\bar A = |F(0)|+ \frac{\|\nabla F(0)\|^2}{2}+ |\min F|$, and by Lemma \ref{lem:BV_new} and Lemma \ref{lem:exp_bound}, we obtain
\begin{equation}\label{eq:Expectation_bound_lemma}
        \left|\mathbb{E}[F(x_k) - F(y_k)]\right|=\left|\int F(x) \, d(\mu_k-\nu_k)(x)\right| \leq (M+1)\tilde K'(\beta, d) \left(\sqrt{\KL(\mu_k||\nu_k)}+\frac{1}{2}\KL(\mu_k||\nu_k) \right),
\end{equation}
where 
\begin{equation*}
    \begin{aligned}
        \tilde K'(\beta, d)=\frac{3}{2}+\frac{2\bar A}{M +1}+ \log \kappa_0 + \frac{8d+8\beta b+8\beta B^2}{m\beta} .
    \end{aligned}
\end{equation*}
Since $\beta \geq d$, and recalling the definition of $\bar A$, we have
\begin{equation}\label{eq:C0prime_def}
\tilde K'(\beta, d)\leq  \frac{3}{2}+\frac{2|F(0)|+ \|\nabla F(0)\|^2  +2 |\min F|}{M+1}+\log \kappa_0 + \frac{8(1+b+B^2)}{m}=: C_0'.
\end{equation}
\end{proof}

We use the following result to control the KL divergence between $\mu_k$ and $\nu_k$ which presents accumulating errors. This is not surprising since the errors are not supposed to be unbiased and can indeed accumulate.

\begin{proposition}\label{prop:discrete_schemes_distance_ST}
    Let $F$, $\gamma$, $\delta$, $\beta$ and $\mu_0$ as in the hypothesis of Proposition \ref{prop:iter-norm-bound-final}. Recalling that $\mu_k$ is the law of the sequence $x_k$ defined by Algorithm \ref{alg:I-ULA} and $\nu_k$ is the law of the sequence $y_k$ defined in \eqref{eq:ULA-F}, we have that 
    \[\KL(\mu_k||\nu_k)\leq \frac{\gamma k \beta  \delta}{4} (P\,\Gamma+Q).\]
\end{proposition}
\begin{proof}
    The proof is identical to the proof of \cite[Lemma 4.4]{XCZG2018}, which bounds the distance between the SGLD and GLD chains for the mini-batch estimator. Their proof only uses an upper bound on the conditional variance of the stochastic gradient, provided there by Lemma C.5. Under our bound on the variance, all subsequent steps carry over verbatim (the only step which changed is their variance estimate step in equation (C.11)). The resulting bound is the same as the one they get in the proof of \cite[Lemma 4.4]{XCZG2018}, with their mini-batch variance factor replaced by our control on the variance, obtaining our thesis. A similar analysis can be found in the proof of \cite[Lemma 7]{Raginsky2017}.
\end{proof}

\paragraph{Proof of Theorem \ref{thm:main_bound_I} and Corollary \ref{cor:main_I}.}

\begin{proof}[\hypertarget{proof_thm_IULA}{Proof} of Theorem \ref{thm:main_bound_I}] We are in the hypothesis of the lemmas and propositions of the previous section.

\textit{Bound of $(a)$.} Recall that $\mu_k$ is the law of $x_k$ and $\nu_{k}$ is the law of $y_k$. These are, respectively, the laws of the discrete stochastic Langevin iterate and of the ULA iterate on $F$. By Lemma~\ref{lem:quad_growth} (ii) applied to $F$, we have
\[
\|\nabla F(x)\|\;\le\;M\|x\|+B,\qquad \forall x\in\R^d.
\]
Combining Proposition \ref{prop:discrete_schemes_distance_ST} with Corollary \ref{cor:discretes_bound_ST} we obtain
\[|(a)|\leq C_0' (M+1) \left(\sqrt{\frac{\gamma k \beta \delta}{4} (P\Gamma+Q)}+\frac{\gamma k \beta \delta}{8} (P\Gamma+Q) \right).\]
We can therefore choose \begin{equation}\label{eq:C2_def}
    C_2:=\frac{P\Gamma + Q}{4}, \quad \text{with $\Gamma$ as in \eqref{eq:Gamma_def}}.
\end{equation}

\textit{Bound of $(b)$ and $(c)$:} These can be bounded directly using Theorem \ref{thm:main_bound}.
\end{proof}

\begin{proof}[\hypertarget{proof_cor_IULA}{Proof} of Corollary \ref{cor:main_I}] First of all, set $\epsilon_1, \epsilon_2, \epsilon_3, \epsilon_4 > 0$, with $\sum_{i=1}^4 \epsilon_i = \epsilon$. We advise choosing $\epsilon_1 \approx \epsilon$, since this will fix $\beta$ and the $\CLS(\beta, d)$ constant can depend badly on $\beta$. However, we leave the choice to the reader. For theoretical purposes we fix $\epsilon_1=\epsilon_2=\epsilon_3= \epsilon_4 = \frac{\epsilon}{4}$.

\textit{\underline{Choice of $\beta$},}
\textit{\underline{Choice of $\gamma$}} and 
\textit{\underline{Choice of $k$}} can be done exactly as in Corollary \ref{cor:main}.

\textit{\underline{Choice of $\delta$}.} Finally we set
\[\delta \leq \min\left\{\frac{m^2}{4P},\frac{\epsilon_4^2}{4 (M+1)^2 (C_0')^2 C_2\gamma k \beta}\right\},\]
with $C_0'$ as in \eqref{eq:C0prime_def} and $C_2$ as in \eqref{eq:C2_def} and with the usual convention that the first condition is void when $P=0$. Since $C_0'$ defined in \eqref{eq:C0prime_def} is greater than one, we also have $\gamma k \beta C_2 \delta \leq 1$, and therefore $\frac{\gamma k \beta C_2 \delta}{2}\leq\sqrt{\gamma k\beta C_2 \delta}$. This implies that
\[(M+1) C_0' \left(\sqrt{\gamma k \beta C_2 \delta} +\frac{\gamma k \beta C_2 \delta}{2}\right)\leq 2(M+1) C_0' \sqrt{\gamma k \beta C_2 \delta} \leq \epsilon_4.\]
\end{proof}

\section{Applications}\label{sec:Applications}

\subsection{Stochastic/Mini-batch ULA}\label{sec:Stochastic_ULA}

In this section we consider the problem of minimizing the non-convex function given by
\[\min_{x\in \R^d} F(x):= \E_{\zeta \sim \mu_\zeta}[f(x,\zeta)].\]

\begin{algorithm}[H] 
\caption{Mini-batch ULA} \label{alg:MB-ULA}
\begin{algorithmic}
  \STATE{{\bf Input:} $x_0 \sim \mu_0$, $\beta, \gamma \in \mathbb{R}_+$, $s\in \N$}
     \FOR{$k = 0,1,\dots$}
        \STATE{sample independent $\zeta_{1,k},\dots, \zeta_{s,k}$ from $\mu_\zeta$}
         \STATE{sample an independent $z_k$ from $\mathcal{N}(0, I_{d})$}
         \STATE{$x_{k + 1} = x_k - \gamma \frac{1}{s}\sum_{i=1}^s \nabla f(x_k,\zeta_{i,k}) +\sqrt{2\gamma\beta^{-1}} z_k$}
     \ENDFOR
 \end{algorithmic}    
\end{algorithm}

We work under the following assumptions and apply directly the results developed in Section \ref{sec:Inexact}.

\begin{assumption}\label{ass:MB}
    Assume that for every $z \in \supp{\mu_\zeta}$ the function $f(\cdot, z)$ is $M$-smooth and $(m,b)$-dissipative. Moreover we assume that $\sup_{z}\|\nabla f(0,z)\|=: B< +\infty$.
\end{assumption}

\begin{lemma}[Inheritance of structural properties]
\label{lem:inheritance}
Let Assumption \ref{ass:MB} be satisfied. Then the objective function $F(x) = \mathbb{E}_{\zeta \sim \mu_\zeta}[f(x, \zeta)]$ is also $M$-smooth and $(m, b)$-dissipativity holds.
\end{lemma}

\begin{proof}
First, we prove $M$-smoothness. By dominated convergence we have $\nabla F(x) = \mathbb{E}_\zeta[\nabla f(x, \zeta)]$. For any $x, y \in \mathbb{R}^d$:
\begin{align*}
\|\nabla F(x) - \nabla F(y)\| &= \|\mathbb{E}_\zeta[\nabla f(x, \zeta) - \nabla f(y, \zeta)]\| \le \mathbb{E}_\zeta[\|\nabla f(x, \zeta) - \nabla f(y, \zeta)\|] \le \mathbb{E}_\zeta[M\|x - y\|] = M\|x - y\|,
\end{align*}
where the first inequality follows from Jensen's inequality and the second from the $M$-smoothness of each $f(\cdot, \zeta)$. Next, we prove $(m, b)$-dissipativity. Using the linearity of expectation and the dissipativity of each $f(\cdot, \zeta)$ we have
\begin{align*}
\langle x, \nabla F(x) \rangle &= \langle x, \mathbb{E}_\zeta[\nabla f(x, \zeta)] \rangle = \mathbb{E}_\zeta[\langle x, \nabla f(x, \zeta) \rangle] \ge \mathbb{E}_\zeta[m\|x\|^2 - b] = m\|x\|^2 - b.
\end{align*}
This completes the proof.
\end{proof}

\begin{lemma}[Controlled growth of the gradient error]
\label{lem:MB-gradient-error}
Let Assumption \ref{ass:MB} be satisfied and let $B := \sup_{\zeta} \|\nabla f(0, \zeta)\|$. For the mini-batch gradient estimator $g(x, \xi) = \frac{1}{s}\sum_{i=1}^s \nabla f(x, \zeta^i)$ where $\zeta^i \sim \mu_\zeta$ are i.i.d., the Mean Squared Error (MSE) satisfies:
\begin{equation}
\mathbb{E}\left[ \|g(x, \xi) - \nabla F(x)\|^2 \right] \leq \delta \left( P\|x\|^2 + Q \right),
\end{equation}
with the following parameters: $\delta = \frac{1}{s}$, $P = 2M^2$ and $Q = 2B^2$.
\end{lemma}

\begin{proof}
By the independence and unbiasedness of the samples $\zeta^i$ ($\mathbb{E}_{\zeta}[\nabla f(x, \zeta)] = \nabla F(x)$), the variance of the mini-batch mean is equal to the variance of a single sample divided by $s$:
\begin{equation}
    \mathbb{E}\left[ \left\| \frac{1}{s}\sum_{i=1}^s \nabla f(x, \zeta^i) - \nabla F(x) \right\|^2 \right] = \frac{1}{s} \mathbb{E}_{\zeta \sim \mu_\zeta} \left[ \|\nabla f(x, \zeta) - \nabla F(x)\|^2 \right].
\end{equation}
Using the identity $\mathbb{E}\|X - \mathbb{E}X\|^2 = \mathbb{E}\|X\|^2 - \|\mathbb{E}X\|^2 \leq \mathbb{E}\|X\|^2$, we can upper-bound the single-sample variance by the second moment: $\mathbb{E}_{\zeta \sim \mu_\zeta} \left[ \|\nabla f(x, \zeta) - \nabla F(x)\|^2 \right] \leq \mathbb{E}_{\zeta \sim \mu_\zeta} \left[ \|\nabla f(x, \zeta)\|^2 \right]$.
By the $M$-smoothness of $f(\cdot, \zeta)$, we have the gradient growth bound $\|\nabla f(x, \zeta) - \nabla f(0, \zeta)\| \leq M\|x\|$. Using the inequality $\|a+b\|^2 \leq 2\|a\|^2 + 2\|b\|^2$, we obtain: $\|\nabla f(x, \zeta)\|^2 \leq 2\|\nabla f(x, \zeta) - \nabla f(0, \zeta)\|^2 + 2\|\nabla f(0, \zeta)\|^2 \leq 2M^2\|x\|^2 + 2B^2$.
Taking the expectation over $\zeta$ and substituting back into the mini-batch expression:
\begin{equation}
    \mathbb{E}\left[ \|g(x, \xi) - \nabla F(x)\|^2 \right] \leq \frac{1}{s} \left( 2M^2\|x\|^2 + 2B^2 \right).
\end{equation}
By setting $\delta = 1/s$, $P = 2M^2$, and $Q = 2B^2$, we recover the desired form.
\end{proof}

We can therefore apply the results of Section \ref{sec:Inexact}. In particular, we have the following result.
\begin{corollary}\label{cor:main_MB}
   Under the same assumptions of Theorem~\ref{thm:main_bound_I}, fix an accuracy $\epsilon \leq 1$. Let
    \begin{equation*}
        \beta = \tilde \Theta\left(\frac{d}{\epsilon}\right),\quad \gamma= \Theta\left(
\frac{\epsilon^2}{\beta d \CLS (\beta, d)}
\right), \quad \text{and}\quad s = \tilde \Theta \left(\frac{\beta^2 \CLS(\beta,d)}{\epsilon^2}\right).
    \end{equation*}
    Let $x_k \in \mathbb{R}^d$ be generated by Algorithm \ref{alg:MB-ULA} at timestep $k$. Fixing 
    \begin{equation}\label{eq:total_count_needed}
        k = O\left(\frac{\beta}{\gamma} \CLS (\beta, d) \log\left(\frac{\beta d}{\epsilon}\right)\right)= \tilde O\left(\frac{\beta^2 d\CLS (\beta,d)^2}{\epsilon^2}\right),
    \end{equation}
    we have $\E[F(x_k)]-\min F\leq \epsilon$. In particular, we can achieve an error of $\E[F(x_k)]-\min F\leq \epsilon \leq 1$, by performing
    \begin{equation}\label{eq:total_count_needed_ST}
    ks = \tilde O\left(\frac{\beta^4 d \CLS(\beta,d)^3}{\epsilon^4}\right),\end{equation}
    evaluations of gradients of the single functions $f(\cdot, \zeta)$.
\end{corollary}

\begin{proof}
   The proof follows from Corollary~\ref{cor:main_I}. To be even sharper in this specific case, to guarantee the error bound is actually enough to require
    \[\delta \leq \frac{\epsilon_4^2}{4 (M+1)^2 (C_0')^2 C_2\gamma k \beta},\]
    with $C_0'$ as in \eqref{eq:C0prime_def} and $C_2$ as in \eqref{eq:C2_def}. Indeed, since $C_0'\geq \frac{8}{m}$, then $\frac{1}{4(M+1)^2(C_0')^2}\leq \frac{m^2}{256 M^2}\leq \frac{m^2}{8M^2}=\frac{m^2}{4P} $ so that we don't need the additional constraint $\delta \leq \frac{m^2}{4P}$ present in Corollary \ref{cor:main_I}. This leads to the condition
    \[s\geq \frac{4 (M+1)^2 (C_0')^2 C_2\gamma k \beta}{\epsilon_4^2},\]
    and the thesis.
\end{proof}

\paragraph{Comparison with existing SGLD guarantees.}

We compare Corollary~\ref{cor:main_MB} with existing stochastic-gradient Langevin guarantees; detailed derivations are deferred to Appendix~\ref{app:comparisons_mini}. The key feature of our result is the cubic dependence on \(\CLS(\beta,d)\) for the complexity $ks$, together with the tight inverse-temperature calibration discussed in Remark~\ref{rem:exact_choice_of_beta}.

\begin{table}[H]
\centering
\begin{tabular}{l c l}
\toprule
Work & Complexity & Main comparison \\
\midrule
Corollary~\ref{cor:main_MB}
&
$\tilde O\!\left(\dfrac{\beta^4 d\,\CLS^3}{\epsilon^4}\right)$
&
our direct KL-to-objective route\\
Raginsky et al.~\cite{Raginsky2017}
&
$\tilde O\!\left(\dfrac{\beta^{13}\CLS^9}{\epsilon^8}\right)$
&
Wasserstein-to-objective route \\
\cite{Raginsky2017} refined
&
$\tilde O\!\left(\dfrac{\beta^9\CLS^7}{\epsilon^8}\right)$
&
still worse in $\epsilon$ and $\CLS$ \\
Xu et al.~\cite{XCZG2018}
&
$\tilde O\!\left(\dfrac{d^7\CLS^5}{\epsilon^5}\right)$
&
most favorable translation, still worse \\
Zou et al.~\cite{Zou2021Faster}
&
$\tilde O\!\left(\dfrac{d^6\beta^2 \CLS^4}{\epsilon^2}\right)$
&
Cheeger route; exploits unbiasedness \\
\midrule
\textbf{Require larger $\beta$:}\\
Kinoshita--Suzuki~\cite{KinoshitaSuzuki2022}
&
$\tilde O\!\left(
    \dfrac{d\beta^2\CLS^3}{\epsilon}
\right)$
&
variance-reduced \\
Chen et al.~\cite{Chen2024}
&
$\tilde O\!\left(
\max\left\{
d^3\CLS^3,\,
\dfrac{d^2\CLS^2}{\epsilon^2}
\right\}
\right)$
&
hitting-time guarantee \\
\bottomrule
\end{tabular}
\caption{Comparison of stochastic-gradient Langevin optimization guarantees. Here $\CLS=\CLS(\beta,d)$ and polynomial factors and fixed problem-dependent constants are suppressed.}
\label{tab:comparison_mb_ula}
\end{table}

The rates in Table~\ref{tab:comparison_mb_ula} should be read at fixed \(\beta\). For global optimization, however, \(\beta\) is not a harmless parameter: it is chosen to make the Gibbs bias small, and \(\CLS(\beta,d)\) may depend exponentially on \(\beta\). As recalled in Remark~\ref{rem:exact_choice_of_beta}, our expected-risk analysis uses an essentially tight temperature scale, $\beta = \frac{d}{2\epsilon_1}\,\tilde\Theta\!\left(\log\frac{d}{\epsilon_1}\right)$, $\epsilon_1\simeq \epsilon$, which is already unavoidable even for quadratic objectives. Therefore, methods that require a larger inverse temperature may suffer a much larger functional-inequality constant, even if their displayed polynomial dependence on \(\epsilon\) looks favorable.

The comparison mirrors the exact and inexact cases. The Wasserstein-based route of \textbf{Raginsky, Rakhlin and Telgarsky~\cite{Raginsky2017}} leads to substantially worse powers of both \(\epsilon\) and \(\CLS(\beta,d)\), even after refining the intermediate estimates. \textbf{Xu, Chen, Zou and Gu~\cite{XCZG2018}} obtain SGLD bounds in terms of a discrete-time spectral gap; even under the favorable translation \(\lambda(\beta,d)^{-1}\sim\CLS(\beta,d)\), the resulting dependence contains a fifth power of \(\CLS\). \textbf{Zou, Xu and Gu~\cite{Zou2021Faster}} exploit unbiasedness and obtain a bound in terms of a Cheeger constant \(\rho\), which translates into a fourth power of \(\CLS(\beta,d)\); their batch-size choice does not depend on \(\CLS\), which can be advantageous in practice. \textbf{Kinoshita and Suzuki~\cite{KinoshitaSuzuki2022}} obtain sharper rates for variance-reduced Langevin dynamics, but their setting requires periodic full-gradient computations and depends explicitly on the finite-sum size, so it is complementary to the plain mini-batch oracle considered here. Moreover, as discussed in Appendix~\ref{app:comparisons_mini}, their objective-value conversion leads to a larger inverse-temperature choice. Since \(\CLS(\beta,d)\) may scale like \(e^{c\beta}\), this can dominate the apparently better polynomial dependence on \(\epsilon\). Finally, \textbf{Chen, Sekhari and Sridharan~\cite{Chen2024}} give high-probability hitting-time guarantees for GLD and SGLD using Lyapunov potentials. Their SGLD result exploits unbiasedness directly and assumes a uniformly sub-Gaussian stochastic-gradient oracle, sample-wise smoothness and dissipativity, and a Lyapunov-potential regularity condition. Under these stronger assumptions, their oracle complexity can be sharper than the one obtained here from the general inexact-gradient framework. However, their analysis leads to a more conservative inverse-temperature calibration than our expected-risk formulation, which only requires \(\mathbb E_{\pi_\beta}F-\min F\) to be below the allocated error budget. Since \(\CLS(\beta,d)\) can depend exponentially on $\beta$, this difference is significant.

\begin{remark}[On unbiased analysis]\label{rem:toPart2}
    We believe that a dedicated treatment of the (unbiased) stochastic setting that explicitly exploits the unbiased nature of the estimator, is fundamental. We omit this approach here because it falls outside the focus of the current work, which primarily addresses the biased and inexact setting using a fundamentally different analysis. Nevertheless, it is noteworthy that our current analysis, even if not optimally sharp for the unbiased case, still yields bounds that, to the best of our knowledge, improve upon existing results for stochastic Langevin methods in optimization.
\end{remark}


\subsection{Zeroth-order optimization}\label{sec:ZO}

We now show that the inexact-ULA analysis developed in Section~\ref{sec:Inexact} naturally applies to zeroth-order optimization. In this setting, the algorithm does not have access to the gradient $\nabla F(x)$, but only to function values of $F$. The goal is therefore to construct a random surrogate $g(x,\xi)$ using finite differences and to verify Assumption~\ref{ass:gradient_error}. Once this is done, the convergence result follows directly from Theorem~\ref{thm:main_bound_I}, yielding explicit function-evaluation complexities for the resulting zeroth-order ULA schemes.

Throughout this subsection we assume that $F$ satisfies Assumptions~\ref{ass:smoothness} and \ref{ass:dissipativity}. We denote as always $B:=\|\nabla F(0)\|$ and recall that by Lemma~\ref{lem:quad_growth} it holds $\|\nabla F(x)\|\leq M\|x\|+B$, for all $x\in\mathbb{R}^d$. This linear growth estimate is the key ingredient that allows us to control the variance of the zeroth-order estimators by a quadratic function of $\|x\|$.

\paragraph{Gaussian finite-difference estimator.}

Let $u\sim \mathcal{N}(0,I_d)$ and let $h>0$ be a smoothing parameter. The one-direction Gaussian finite-difference estimator is defined by
\[
    g_h^{\mathrm{G}}(x,u)
    :=
    \frac{F(x+h u)-F(x)}{h}\,u.
\]
Given a batch of independent directions $u_1,\ldots,u_s\sim\mathcal{N}(0,I_d)$, we use the averaged estimator
\[
    g_{h,s}^{\mathrm{G}}(x)
    :=
    \frac{1}{s}\sum_{i=1}^s g_h^{\mathrm{G}}(x,u_i).
\]
This estimator is not, in general, unbiased for $\nabla F(x)$. Rather, it is unbiased for the gradient of the Gaussian-smoothed objective \cite{NesterovSpokoiny2017}
\[
    F_h^{\mathrm{G}}(x)
    :=
    \mathbb{E}_{u\sim\mathcal{N}(0,I_d)}[F(x+h u)].
\]
The error with respect to $\nabla F(x)$ is therefore the sum of a variance term and a smoothing bias term.

\begin{lemma}[Gaussian zeroth-order estimator]\label{lem:gaussian_ZO_error}
Assume that $F$ is $M$-smooth. Let $g_{h,s}^{\mathrm{G}}$ be defined as above and assume $h\leq 1$. Then, for every $x\in\mathbb{R}^d$,
\[
    \mathbb{E}\left[
        \left\|g_{h,s}^{\mathrm{G}}(x)-\nabla F(x)\right\|^2
    \right]
    \leq
    \delta_{\mathrm{G}}
    \left(P_{\mathrm{G}}\|x\|^2+Q_{\mathrm{G}}\right),
\]
where 
\begin{equation}\label{eq:deltaG_def}
    \delta_G = \frac{(d + 2)}{s} + \frac{(d + 6)^3}{s} h^2 +dh^2,
\end{equation}
$P_{\mathrm{G}}:=8M^2$ and $Q_{\mathrm{G}}:=2M^2+8B^2$. In particular, Assumption~\ref{ass:gradient_error} holds with $\delta=\delta_{\mathrm{G}}$, $P=P_{\mathrm{G}}$ and $Q=Q_{\mathrm{G}}$.
\end{lemma}

\begin{proof}
Since $g_{h,s}^{\mathrm{G}}$ is an average of $s$ independent copies of $g_h^{\mathrm{G}}$, we have
\[
\begin{aligned}
    \mathbb{E}\left[
        \left\|g_{h,s}^{\mathrm{G}}(x)-\nabla F(x)\right\|^2
    \right]
    &\leq
    \frac{2}{s}\mathbb{E}\left[
        \left\|g_h^{\mathrm{G}}(x,u)-\nabla F_h^{\mathrm{G}}(x)\right\|^2
    \right]
    +
    2\left\|\nabla F_h^{\mathrm{G}}(x)-\nabla F(x)\right\|^2        \\
    &\leq
    \frac{2}{s}\mathbb{E}\left[
        \left\|g_h^{\mathrm{G}}(x,u)\right\|^2
    \right]
    +
    2\left\|\nabla F_h^{\mathrm{G}}(x)-\nabla F(x)\right\|^2 .
\end{aligned}
\]
By $M$-smoothness,
\[
    \left\|\nabla F_h^{\mathrm{G}}(x)-\nabla F(x)\right\|
    =
    \left\|
        \mathbb{E}\left[\nabla F(x+h u)-\nabla F(x)\right]
    \right\|
    \leq
    Mh\,\mathbb{E}\|u\|
    \leq
    Mh\sqrt{d}.
\]
Therefore, $2\left\|\nabla F_h^{\mathrm{G}}(x)-\nabla F(x)\right\|^2
    \leq
    2M^2h^2d$. 
For the other term, by adding and subtracting $\langle\nabla F(x),hu \rangle u$, we have    
\begin{equation*}
    \begin{aligned}
        \mathbb{E}\left[\left\|g_h^{\mathrm{G}}(x,u)\right\|^2\right] 
        &\leq \frac{2}{h^2} (\mathbb{E} [ (F(x + hu) - F(x) - \langle\nabla F(x),hu \rangle)^2\|u\|^2]  + \mathbb{E}[(\langle\nabla F(x),hu \rangle)^2 \|u\|^2])
    \end{aligned}
\end{equation*}
Since $F$ is $M$-smooth, by the Descent Lemma,
\begin{equation*}
    \begin{aligned}
        \mathbb{E}\left[\left\|g_h^{\mathrm{G}}(x,u)\right\|^2\right] &\leq \frac{2}{h^2} \left(\mathbb{E} \left[ \frac{M^2}{4}\|hu\|^4 \|u\|^2\right]  + \mathbb{E}[(\langle\nabla F(x),hu \rangle)^2 \|u\|^2 ] \right)\\
        &= \frac{M^2}{2} h^2 \mathbb{E}[\|u\|^6]  + \frac{2}{h^2} ( \mathbb{E}[(\langle\nabla F(x),hu \rangle)^2] \|u\|^2])
    \end{aligned}
\end{equation*}
Since $\mathbb{E}_{u\sim\mathcal{N}(0,I)} [\| u \|^p] \leq (d + p)^{p/2}$ \cite[Lemma 1]{NesterovSpokoiny2017},
\begin{equation*}
    \begin{aligned}
        \mathbb{E}\left[\left\|g_h^{\mathrm{G}}(x,u)\right\|^2\right] &\leq 
        \frac{M^2}{2} (d + 6)^3 h^2  + 2 (\nabla F(x)^\top \mathbb{E}[ uu^\top u u^\top] \nabla F(x))\\
    \end{aligned}
\end{equation*}
Since $\mathbb{E}_{u\sim\mathcal{N}(0,I)} [ u u^\top u u^\top] = (d + 2)I$ 
\begin{equation*}
    \begin{aligned}
        \mathbb{E}\left[\left\|g_h^{\mathrm{G}}(x,u)\right\|^2\right] &\leq \frac{M^2}{2} (d + 6)^3 h^2  + 2 (d + 2) \| \nabla F(x)\|^2.
    \end{aligned}
\end{equation*}
By Lemma \ref{lem:quad_growth},
\begin{equation*}
    \begin{aligned}
        \mathbb{E}\left[\left\|g_h^{\mathrm{G}}(x,u)\right\|^2\right] &\leq \frac{M^2}{2} (d + 6)^3 h^2  + 4 (d + 2)M^2 \| x\|^2 + 4(d + 2)B^2.
    \end{aligned}
\end{equation*}
Therefore, we have 
\begin{equation*}
    \begin{aligned}
        \mathbb{E}\left[\left\|g_{h,s}^{\mathrm{G}}(x)-\nabla F(x)\right\|^2 \right] &\leq  8\frac{d + 2}{s}M^2 \| x\|^2 + 8 \frac{d + 2}{s}B^2 + M^2 \frac{(d + 6)^3}{s}h^2 + 2M^2 dh^2.
    \end{aligned}
\end{equation*}
Hence, by letting $\delta_G = \frac{(d + 2)}{s} + \frac{(d + 6)^3}{s} h^2+dh^2$, we have
\begin{equation*}
    \begin{aligned}
\mathbb{E}\left[\left\|g_{h,s}^{\mathrm{G}}(x)-\nabla F(x)\right\|^2 \right] &\leq \delta_G \left(\underbrace{8M^2}_{P_G} \|x\|^2 + \underbrace{8B^2 + 2M^2}_{Q_G}  \right)
    \end{aligned}
\end{equation*}

\end{proof}

\paragraph{Spherical finite-difference estimator.}

We can alternatively use a spherical estimator. Let $u$ be uniformly distributed on the Euclidean sphere $\mathbb{S}^{d-1}$ and define the one-sided spherical finite-difference estimator
\[
    g_h^{\mathrm{S}}(x,u)
    :=
    \frac{d}{h}\left(F(x+h u)-F(x)\right)u.
\]
Given independent directions $u_1,\ldots,u_s\sim \mathrm{Unif}(\mathbb{S}^{d-1})$, set
\[
    g_{h,s}^{\mathrm{S}}(x)
    :=
    \frac{1}{s}\sum_{i=1}^s g_h^{\mathrm{S}}(x,u_i).
\]
This estimator is unbiased for the gradient of the ball-smoothed objective \cite{Rando2023,flaxman2005online}
\[
    F_h^{\mathrm{S}}(x)
    :=
    \mathbb{E}_{v\sim \mathrm{Unif}(\mathbb{B}^d)}[F(x+h v)].
\]

\begin{lemma}[Spherical zeroth-order estimator]\label{lem:spherical_ZO_error}
Assume that $F$ is $M$-smooth. Let $g_{h,s}^{\mathrm{S}}$ be defined as above and assume $h\leq 1$. Then, for every $x\in\mathbb{R}^d$,
\[
    \mathbb{E}\left[
        \left\|g_{h,s}^{\mathrm{S}}(x)-\nabla F(x)\right\|^2
    \right]
    \leq
    \delta_{\mathrm{S}}
    \left(P_{\mathrm{S}}\|x\|^2+Q_{\mathrm{S}}\right),
\]
where
\begin{equation}\label{eq:deltaS_def}
    \delta_S = \frac{d}{s} + \frac{d^2 h^2}{s} + h^2,
\end{equation}
 $P_{\mathrm{S}}:=8M^2$ and $Q_{\mathrm{S}}:=2M^2+8B^2$. In particular, Assumption~\ref{ass:gradient_error} holds with $\delta=\delta_{\mathrm{S}}$, $P=P_{\mathrm{S}}$ and $Q=Q_{\mathrm{S}}$.
\end{lemma}

\begin{proof}
Since $g_{h,s}^{\mathrm{S}}$ is an average of $s$ independent copies of $g_h^{\mathrm{S}}$, we have
\[
\begin{aligned}
    \mathbb{E}\left[
        \left\|g_{h,s}^{\mathrm{S}}(x)-\nabla F(x)\right\|^2
    \right]
    &\leq
    \frac{2}{s}
    \mathbb{E}\left[
        \left\|g_h^{\mathrm{S}}(x,u)-\nabla F_h^{\mathrm{S}}(x)\right\|^2
    \right]
    +
    2\left\|\nabla F_h^{\mathrm{S}}(x)-\nabla F(x)\right\|^2        \\
    &\leq
    \frac{2}{s}
    \mathbb{E}\left[
        \left\|g_h^{\mathrm{S}}(x,u)\right\|^2
    \right]
    +
    2\left\|\nabla F_h^{\mathrm{S}}(x)-\nabla F(x)\right\|^2 .
\end{aligned}
\]
For the bias term, since $v\in\mathbb{B}^d$ and $F$ is $M$-smooth,
\[
    \left\|\nabla F_h^{\mathrm{S}}(x)-\nabla F(x)\right\|
    =
    \left\|
        \mathbb{E}_{v}\left[\nabla F(x+h v)-\nabla F(x)\right]
    \right\|
    \leq
    Mh\,\mathbb{E}\|v\|
    \leq
    Mh.
\]
On the other hand, by since $u\sim \text{Unif}(\mathbb{S}^{d-1})$ we have $\|u\| =1$. Thus, adding and subtracting $\langle \nabla F(x), hu\rangle  $, we have
\begin{equation*}
    \begin{aligned}
   \mathbb{E}\left[\left\|g_h^{\mathrm{S}}(x,u)\right\|^2\right]&= \frac{d^2}{h^2} \mathbb{E} \left[(F(x + hu) - F(x))^2 \underbrace{\|u\|^2}_{=1} \right]\\
        &\leq 2\frac{d^2}{h^2} \bigg( \mathbb{E}[(F(x + hu) - F(x)  - \langle \nabla F(x), hu\rangle)^2] + \mathbb{E}[( \langle \nabla F(x), hu\rangle)^2]\bigg)
    \end{aligned}
\end{equation*}
Since $F$ is $M$-smooth and $\mathbb{E}_{u\sim \text{Unif}(\mathbb{S}^{d-1})}[uu^\top] = \frac{1}{d} I$ \cite[Lemma 2]{rando2025structured}, by the Descent Lemma we have
\begin{equation*}
    \begin{aligned}
    \mathbb{E}\left[\left\|g_h^{\mathrm{S}}(x,u)\right\|^2\right] 
    &\leq\frac{d^2 M^2}{2}h^2 + 2\frac{d^2}{h^2} \bigg( \mathbb{E}[( \langle \nabla F(x), hu\rangle)^2]\bigg)\\
    &=\frac{d^2 M^2}{2}h^2 + 2d^2 \nabla F(x)^\top \mathbb{E}[uu^\top] \nabla F(x)\\ 
    &=\frac{d^2 M^2}{2}h^2 + 2d \| \nabla F(x)\|^2. 
    \end{aligned}
\end{equation*}
By Lemma \ref{lem:quad_growth},
\begin{equation*}
    \begin{aligned}
        \mathbb{E}\left[\left\|g_h^{\mathrm{S}}(x,u)\right\|^2\right] &\leq \frac{d^2 M^2}{2}h^2 + 4d M^2 \| x\|^2 + 4dB^2.
    \end{aligned}
\end{equation*}
Therefore, 
\begin{equation*}
    \begin{aligned}
        \mathbb{E}\left[\left\|g_{h,s}^{\mathrm{S}}(x)-\nabla F(x)\right\|^2\right]&\leq 8 \frac{d M^2}{s} \| x\|^2 + \left(\frac{d^2 }{s} + 2\right)M^2 h^2 + \frac{8d B^2}{s} 
    \end{aligned}
\end{equation*}
Hence, by letting $\delta_S = \frac{d}{s} + (\frac{d^2 }{s} +1) h^2$, we have
\begin{equation*}
    \begin{aligned}
\mathbb{E}\left[\left\|g_{h,s}^{\mathrm{S}}(x)-\nabla F(x)\right\|^2\right]&\leq \delta_S\left( \underbrace{8M^2}_{P_S} \|x\|^2 + \underbrace{2M^2 + 8B^2}_{Q_S} \right)
    \end{aligned}
\end{equation*}

\end{proof}

\paragraph{Zeroth-order Langevin algorithm.}

The zeroth-order Langevin scheme is obtained by replacing the gradient in Algorithm~\ref{alg:I-ULA} with either $g_{h,s}^{\mathrm{G}}$ or $g_{h,s}^{\mathrm{S}}$. Namely,
\[
    x_{k+1}
    =
    x_k-\gamma g_{h,s}^{\mathrm{ZO}}(x_k)
    +
    \sqrt{2\gamma\beta^{-1}}\,z_k,
    \qquad
    z_k\sim\mathcal{N}(0,I_d),
\]
where $g_{h,s}^{\mathrm{ZO}}$ denotes either the Gaussian or the spherical estimator.

\begin{algorithm}[H]
\caption{Zeroth-order ULA}\label{alg:ZO-ULA}
\begin{algorithmic}
  \STATE{{\bf Input:} $x_0\sim\mu_0$, $\beta,\gamma,h\in\mathbb{R}_+$, batch size $s\in\mathbb{N}$}
     \FOR{$k = 0,1,\dots$}
        \STATE{sample independent random directions $u_{k,1},\ldots,u_{k,s}$ from $\mathcal{N}(0, I_d)$ or $\mathrm{Unif}(\mathbb{S}^{d - 1})$}
        \STATE{construct $g_{h,s}^{\mathrm{ZO}}(x_k)$ using either the Gaussian or spherical estimator}
        \STATE{sample an independent $z_k\sim\mathcal{N}(0,I_d)$}
        \STATE{$x_{k+1}=x_k-\gamma g_{h,s}^{\mathrm{ZO}}(x_k)+\sqrt{2\gamma\beta^{-1}}z_k$}
     \ENDFOR
\end{algorithmic}
\end{algorithm}

The previous lemmas show that Algorithm~\ref{alg:ZO-ULA} is a particular instance of Algorithm~\ref{alg:I-ULA}. Hence, Theorem~\ref{thm:main_bound_I} applies directly.

\begin{theorem}[Zeroth-order ULA]\label{thm:ZO_ULA}
Assume that $F$ satisfies Assumptions~\ref{ass:smoothness} and \ref{ass:dissipativity}, and let $\mu_0$ satisfy Assumption~\ref{ass:mu_0}. Assume moreover that $h\leq1$ and that $\beta$ and $\gamma$
satisfy the conditions of Theorem~\ref{thm:main_bound_I}. Let $x_k$ be generated by Algorithm~\ref{alg:ZO-ULA} with either the Gaussian or spherical estimator. Then, if $\delta_{\mathrm{ZO}}\leq \frac{m^2}{32M^2}$ (with $\delta_{\mathrm{ZO}}=\delta_{\mathrm{G}}$ for the Gaussian estimator or $\delta_{\mathrm{ZO}}=\delta_{\mathrm{S}}$ for the spherical one), the bound of Theorem~\ref{thm:main_bound_I} holds with
$\delta=\delta_{\mathrm{ZO}}$, $P=8M^2$, and
$Q=2M^2+8B^2$.
\end{theorem}

\begin{proof}
By Lemma~\ref{lem:gaussian_ZO_error}, the Gaussian estimator satisfies Assumption~\ref{ass:gradient_error} with $\delta=\delta_{\mathrm{G}}$. By Lemma~\ref{lem:spherical_ZO_error}, the spherical estimator satisfies Assumption~\ref{ass:gradient_error} with $\delta=\delta_{\mathrm{S}}$. The claim follows by applying Theorem~\ref{thm:main_bound_I}.
\end{proof}

\begin{corollary}[Complexity of zeroth-order ULA]\label{cor:ZO_complexity}
Under the same assumptions of Theorem~\ref{thm:ZO_ULA}, fix an accuracy $\epsilon \in (0,1)$. Let
\[
    \beta=\tilde\Theta\left(\frac{d}{\epsilon}\right),
    \qquad
    \gamma=
    \Theta\left(
        \frac{\epsilon^2}{\beta d\CLS(\beta,d)}
    \right),
\]
and
\[
    k=
    \tilde O\left(
        \frac{\beta^2 d\CLS(\beta,d)^2}{\epsilon^2}
    \right).
\]
For the Gaussian estimator, choose
\begin{equation}\label{eq:sandh_ZO_G}
    s=
    \tilde\Theta\left(
        \frac{\beta^2 d\CLS(\beta,d)}{\epsilon^2}
    \right),
    \qquad
    h=
    \tilde\Theta\left(
        \frac{\epsilon}{\beta d\sqrt{\CLS(\beta,d)}}
    \right).
\end{equation}
For the spherical estimator, choose
\begin{equation}\label{eq:sandh_ZO_S}
     s=
    \tilde\Theta\left(
        \frac{\beta^2 d\CLS(\beta,d)}{\epsilon^2}
    \right),
    \qquad
    h=
    \tilde\Theta\left(
        \frac{\epsilon}{\beta\sqrt{d\CLS(\beta,d)}}
    \right).
\end{equation}
With these choices, $\mathbb{E}[F(x_k)]-\min F\leq \epsilon$. Consequently, the total number of function evaluations is
\[
    ks=\tilde O\left(
        \frac{\beta^4 d^2\CLS(\beta,d)^3}{\epsilon^4}
    \right).
\]
\end{corollary}

\begin{proof}
The proof follows from Corollary~\ref{cor:main_I}. The exact-gradient part of the error is controlled by the same choices of $\beta$, $\gamma$, and $k$ used there. In particular we have $\gamma k\beta  =\tilde O\left(\beta^2\CLS(\beta,d)\right)$. By increasing $k$ if necessary, without changing its stated
asymptotic order, we may assume that $\gamma k\beta\ge1$.
Moreover, enlarging $C_2$ if necessary, we assume that $C_2\ge1$. It remains to choose $s$ and $h$ so that the zeroth-order precision parameter satisfies
\begin{equation}\label{eq:ZO_delta_requirement}
    \delta_{\mathrm{ZO}}
    \leq
    \frac{\epsilon_4^2}
    {4(M+1)^2(C_0')^2 C_2\gamma k \beta},
\end{equation}
with $C_0'$ as in \eqref{eq:C0prime_def} and $C_2$ as in \eqref{eq:C2_def}. Notice that, in both the Gaussian and spherical cases, the constant $P$ in Assumption~\ref{ass:gradient_error} can be taken to be $P=8 M^2$. Therefore, since $\epsilon_4\le1$, $C_2\ge1$,
$\gamma k\beta\ge1$ and $C_0'\geq 8/m$, condition \eqref{eq:ZO_delta_requirement} implies
\[
    \delta_{\mathrm{ZO}}
    \leq
    \frac{1}{4(M+1)^2(C_0')^2}
    \leq
    \frac{m^2}{256(M+1)^2}
    \leq
    \frac{m^2}{256M^2}
    \leq
    \frac{m^2}{32M^2}
    =
    \frac{m^2}{4P}.
\]
Thus the stability condition required in Corollary~\ref{cor:main_I} is automatically satisfied. Set
\[
    \Delta_{\mathrm{ZO}}
\leq \frac{\epsilon_4^2}
    {4(M+1)^2(C_0')^2 C_2\gamma k \beta}.
\]
Since $\gamma k\beta
    =
    \tilde O\left(\beta^2\CLS(\beta,d)\right)$,
we can choose $\Delta_{\mathrm{ZO}}
    =
    \tilde \Theta\left(
        \frac{\epsilon^2}{\beta^2\CLS(\beta,d)}
    \right)$. For the Gaussian estimator, Lemma~\ref{lem:gaussian_ZO_error} gives $\delta_G = \frac{(d + 2)}{s} + \frac{(d + 6)^3}{s} h^2+dh^2$. Thus it is enough to impose $\frac{(d+2)}{s}\leq \frac{\Delta_{\mathrm{ZO}}}{3}$, $
    \frac{(d + 6)^3h^2}{s}\leq \frac{\Delta_{\mathrm{ZO}}}{3}$ and $
    dh^2\leq \frac{\Delta_{\mathrm{ZO}}}{3}$. In particular, it is sufficient to choose $s\geq \frac{3(d+2)}{\Delta_{\mathrm{ZO}}}$ and $
    h\leq \frac{1}{\sqrt{3}(d+6)}\min\{1, \sqrt{\Delta_{\mathrm{ZO}}}\}$. Using the expression of $\Delta_{\mathrm{ZO}}$, this gives \eqref{eq:sandh_ZO_G}. For the spherical estimator, Lemma~\ref{lem:spherical_ZO_error} gives $\frac{d}{s} + \frac{d^2 h^2}{s} + h^2$. Hence it is enough to impose $\frac{d}{s}\leq \frac{\Delta_{\mathrm{ZO}}}{3}$, $ \frac{d^2h^2}{s}\leq \frac{\Delta_{\mathrm{ZO}}}{3}$ and $h^2\leq \frac{\Delta_{\mathrm{ZO}}}{3}$. In particular, it is sufficient to choose $s\geq \frac{3d}{\Delta_{\mathrm{ZO}}}$ and $h\leq \min\left\{\frac{1}{\sqrt{d}},\sqrt{\frac{\Delta_{\mathrm{ZO}}}{3}}\right\}$. Using the expression of $\Delta_{\mathrm{ZO}}$, this gives \eqref{eq:sandh_ZO_S}. With these choices, condition \eqref{eq:ZO_delta_requirement} holds, and therefore Corollary~\ref{cor:main_I} gives
\[
    \mathbb{E}[F(x_k)]-\min F\leq \epsilon.
\]
Finally, both the Gaussian and the one-sided spherical estimator require $s+1$ function evaluations per Langevin step, since the value $F(x_k)$ can be computed once and reused for all directions. Hence the total number of function evaluations is, up to leading order,
\[
    k s
    =
    \tilde O\left(
        \frac{\beta^2 d\CLS(\beta,d)^2}{\epsilon^2}
    \right)
    \tilde O\left(
        \frac{\beta^2 d\CLS(\beta,d)}{\epsilon^2}
    \right)
    =
    \tilde O\left(
        \frac{\beta^4 d^2\CLS(\beta,d)^3}{\epsilon^4}
    \right),
\]
which concludes the proof.
\end{proof}

\begin{remark}[Gaussian versus spherical smoothing]
Both the Gaussian and the one-sided spherical estimators considered above have the same leading-order oracle cost. Indeed, for a batch of $s$ directions, the value $F(x)$ can be computed once and reused, so both estimators require $s+1$ function evaluations per Langevin step in the deterministic zeroth-order setting. The main difference between the two estimators is the dependence of the smoothing error on the dimension. The spherical estimator allows a larger smoothing radius $h$ for the same target precision, which can be preferable in practice, especially when function evaluations are noisy or affected by numerical precision errors.
\end{remark}

\paragraph{Comparison with existing zeroth-order Langevin guarantees.}

To the best of our knowledge, there are no directly comparable expected-excess-risk bounds for zeroth-order Langevin optimization under the same smooth dissipative assumptions. We therefore compare with the closest zeroth-order Langevin sampling results in Appendix~\ref{app:comparisons_ZO}.

\section{Experiments} \label{sec:experiments}

In this section, we provide numerical experiments to investigate how the parameters of Algorithm \ref{alg:ZO-ULA} affect its practical performance. We consider three standard synthetic non-convex target functions, namely the Ackley function, the Rastrigin function, and the Levy function\footnote{\label{fn:benchmark_functions} These are standard benchmark functions and their definition with typical evaluation domains can be found for instance at the following url \url{https://www.sfu.ca/~ssurjano/optimization.html}.}. For each target function, we include a regularization term $\lambda \|x\|^2$ with $\lambda = 10^{-3}$. We use spherical finite-difference estimator, fix a budget of $10^6$ function evaluations and consider different dimensions $d = 10, 50,$ and $100$.  Each experiment is repeated $5$ times with different initializations $x_0$ sampled uniformly from the typical evaluation domain of each function$^\text{\ref{fn:benchmark_functions}}$, and we report the mean and standard deviation across repetitions. 
Figure~\ref{fig:performance_gamma} reports the normalized average optimality gap $\frac{F(x_k)-\min F}{F(x_0)-\min F}$, computed at the best $x_k$ observed over the last $100$ iterations, as a function of the stepsize $\gamma$ for different numbers of directions $s$ and different input dimensions $d$. The gap is clipped to $1$ in case of divergence. In every experiment, the exploration parameter $\beta$ is selected via a grid search over the set $\{10^{-2}, 10^{-1}, 1.0, 10, 10^2, 10^3, 10^4\}$.
\begin{figure}[H]
    \centering
    \includegraphics[width=0.75\linewidth]{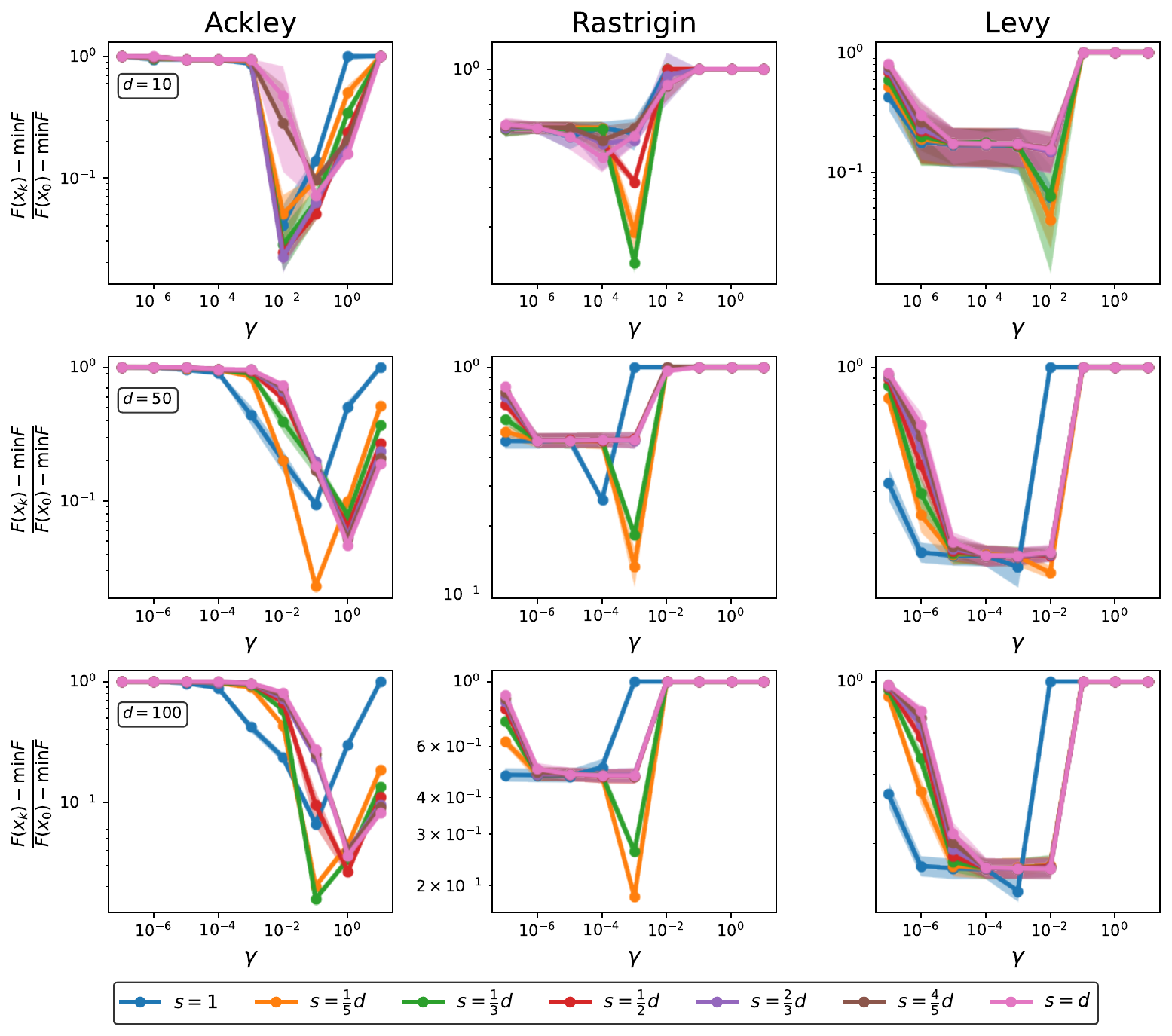}
    \caption{Averaged normalized optimality gap at the best iterate over the last $100$ iterations, as a function of the stepsize $\gamma$, for varying number of directions $s$. Rows correspond to increasing input dimension $d$ (top to bottom: $d=10, 50, 100$), while columns correspond to the different target functions (left to right: Ackley, Rastrigin, Levy).}
    \label{fig:performance_gamma}
\end{figure}
We observe that both extremes of $\gamma$ provide bad performance: for very small stepsizes the algorithm makes negligible progress within the fixed evaluation budget, so the optimality gap remains close to $1$, whereas for very large stepsizes the algorithm diverges, and the gap is clipped to $1$. 
Increasing the number of directions $s$ increases the largest admissible stepsize, since the gradient estimator becomes more stable, and consequently improves overall performance. However, if $s$ is too large, performance may deteriorate. As the budget of function evaluations is fixed, a larger $s$ implies a higher number of evaluation per iteration, which in turn reduces the total number of iterations the algorithm can perform. Since the target functions are non-convex, the algorithm can become trapped near stationary points and may require many iterations to escape them via the exploration component (such an effect is particularly visible for the Rastrigin function). This is coherent with previous observations on finite-difference methods, where increasing the number of directions similarly improves gradient estimate quality, allows for larger stepsizes, and comes at the expense of the number of iterations affordable under a fixed evaluation budget. Figure~\ref{fig:performance_beta} reports the normalized average optimality gap as a function of the exploration parameter $\beta$ for different numbers of directions $s$ and different input dimensions $d$. In every experiment, the stepsize $\gamma$ is selected via a grid search over the set $\{ 10^{-7}, 10^{-6}, 10^{-5}, 10^{-4}, 10^{-3}, 10^{-2}, 10^{-1}, 1.0, 10\}$. 
\begin{figure}[H]
    \centering
    \includegraphics[width=0.75\linewidth]{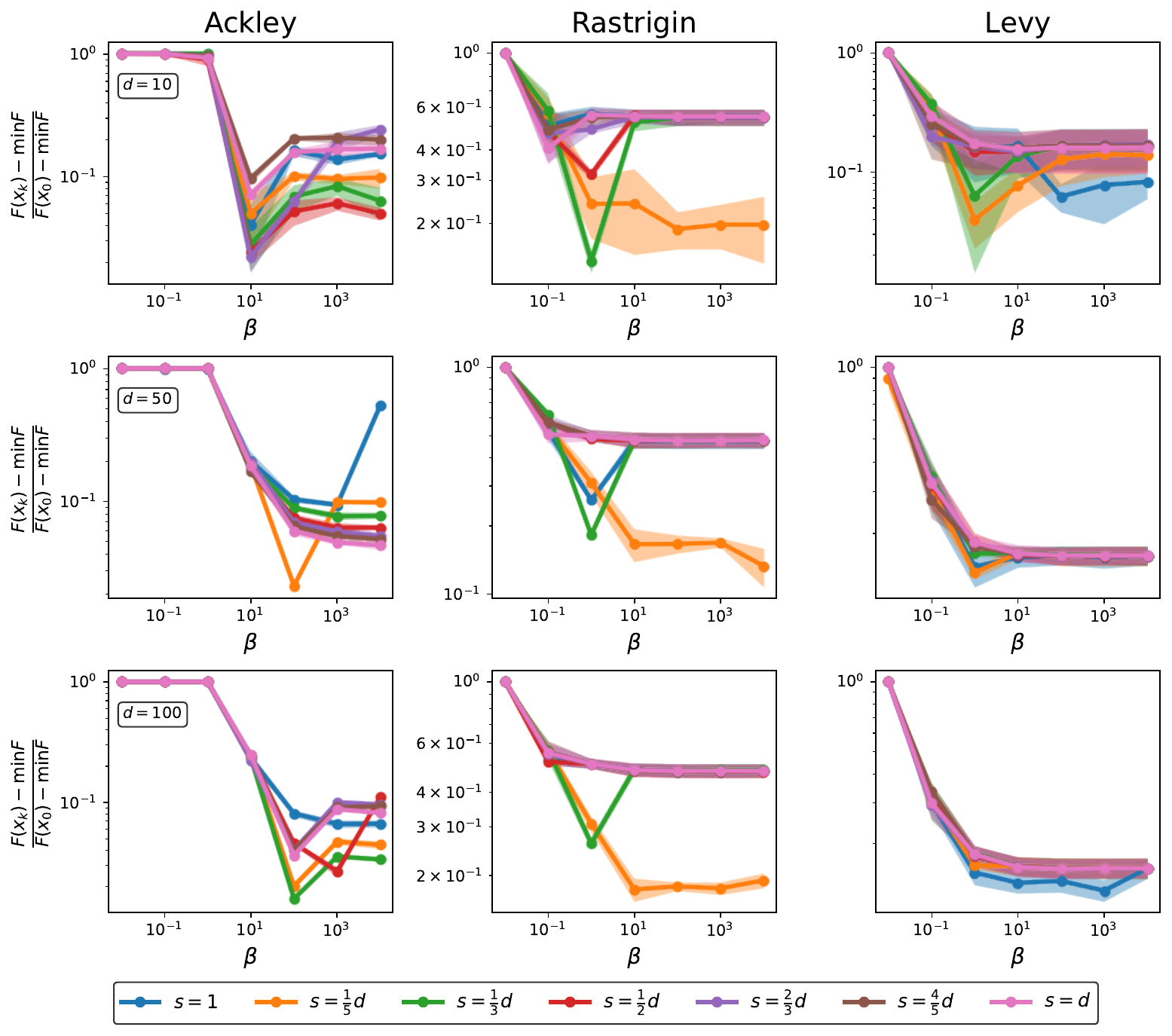}
    \caption{Averaged normalized optimality gap at the best iterate over the last $100$ iterations, as a function of the exploration parameter $\beta$, for varying number of directions $s$. Rows correspond to increasing input dimension $d$ (top to bottom: $d=10, 50, 100$), while columns correspond to the different target functions (left to right: Ackley, Rastrigin, Levy).}
    \label{fig:performance_beta}
\end{figure}
We observe that for very small $\beta$, Algorithm \ref{alg:ZO-ULA} provides bad performance. This is because with such a choice the stochastic component of the dynamics is too strong relative to the drift. As a result, the algorithm either diverges, or, when it converges, the stationary distribution it samples from has a too large variance, yielding a poor approximation of the minimizer. As $\beta$ increases, the target Gibbs measure becomes more peaked around the global minimum, so that sampling from (or near) the stationary distribution yields a more precise approximation of the minimizer, and performance improves. Notice that such a behavior is expected from the theory. However, for excessively large $\beta$, performance may deteriorate. A more peaked Gibbs measure requires more iterations for the algorithm to reach (or approach) its stationary distribution, so that, under a fixed evaluation budget, an insufficient number of iterations are performed to benefit from this increased precision. Notice that this effect is amplified when $s$ is large, since larger $s$ further reduces the number of iterations affordable within the same budget. Notably, performance also deteriorates for $s=1$ when $\beta$ is too large. In this case, in addition to requiring more iterations to approach the stationary distribution, the admissible stepsize $\gamma$ is smaller for $s=1$, as observed in Figure \ref{fig:performance_gamma}, and thus the algorithm progresses even more slowly toward the stationary distribution, further limiting the precision achievable within the fixed evaluation budget.

\section{Conclusions}

We studied Langevin-based algorithms for non-convex optimization under smoothness and dissipativity assumptions, focusing on expected excess-risk bounds. We derived explicit and improved bounds for exact-gradient ULA, inexact-gradient ULA, mini-batch ULA, and zeroth-order ULA. As a further contribution, the KL-to-observable transfer we develop is independent of ULA and can be combined with other sampling results stated in relative entropy whenever suitable exponential moments are available. Our inexact-gradient framework covers a broad class of biased and unbiased gradient estimators. However, exploiting the additional cancellations available in the unbiased case could yield sharper mini-batch ULA bounds, for instance along the lines of~\cite{DasDheerajRaj2023}. Another direction is to relax smoothness and dissipativity. Recent works treat tail-growth and weak-smoothness conditions~\cite{ErdogduHosseinzadeh2021,NguyenDangChen2023}, stationarity under mild non-log-concavity~\cite{BCESZ2022}, non-smooth sampling~\cite{Johnston2025}, and optimization under weaker dissipativity or H\"older regularity~\cite{Chen2024}. A further direction is to improve the dependence on functional-inequality constants through perturbative or problem-specific estimates, for example when the objective has a strongly confining part plus a bounded or controlled perturbation~\cite{Cattiaux2022}. Such refinements would be complementary to the present analysis and could substantially improve the practical relevance of the resulting bounds. Finally, our analysis may also be relevant beyond global optimization, in particular for Langevin-based Plug-and-Play (PnP)~\cite{Laumont2022PnPULA,Laumont2023MAPPnP,KlatzerMelidonisPereyraZygalakis2026,BoutonThouveninRepettiChainais2026} and proximal sampling methods in imaging inverse problems~\cite{HabringHollerPock2024,EhrhardtKugerSchoenlieb2024,HabringFalkZachPock2026,RenaudDeBortoliLeclairePapadakis2025}. The techniques developed here may also provide new insights into noisy zeroth-order optimization~\cite{NesterovSpokoiny2017,GhadimiLan2013,DuchiJordanWainwrightWibisono2015,rando2025structured,rando2026zobaefficientsingleloopzerothorder,Rando2023,Rando2024} and can be applied in future work on standard applications of zeroth-order optimization~\cite{zoo,rando2025new,malladi_llm_fine_tuning,sartore2024automatic}.

\paragraph*{Acknowledgments.}
\sloppy E.N., S.V. and L.R. acknowledge the support of the US Air Force Office of Scientific Research (FA8655-22-1-7034). The research by E.N. and S.V. has been supported by the MUR Excellence Department Project awarded to Dipartimento di Matematica, Università di Genova, CUP D33C23001110001. L.R. acknowledges the financial support of the European Commission (Horizon Europe grant ELIAS 101120237) and the Ministry of Education, University and Research (FARE grant ML4IP R205T7J2KP). L.R. is affiliated to the Istituto Italiano di Tecnologia (IIT). E.N. and S.V. are members of the Gruppo Nazionale per l’Analisi Matematica, la Probabilità e le loro Applicazioni (GNAMPA) of the Istituto Nazionale di Alta Matematica (INdAM). The research of M.R. has been partially funded by a France 2030 support managed by the Agence Nationale de la Recherche, under the reference ANR-23-PEIA-0004 (PDE-AI project). Experiments presented in this paper were carried out using the Grid'5000 testbed, supported by a scientific interest group hosted by Inria and including CNRS, RENATER and several Universities as well as other organizations (see \url{https://www.grid5000.fr}). This work represents only the view of the authors. The European Commission and the other organizations are not responsible for any use that may be made of the information it contains.

\paragraph*{Code Availability.} The code used to perform the experiments is available at \url{https://github.com/MarcoRando/szlan}. The version corresponding to the experiments reported in this paper is archived under the release associated to the tag szlan-v1.0.

\bibliographystyle{plain}
\bibliography{bibliography}

\appendix

\setcounter{theorem}{0}

\renewcommand{\thelemma}{\thesection.\arabic{lemma}}
\renewcommand{\thedefinition}{\thesection.\arabic{definition}}

\section{Measure inequalities}\label{app:inequalities}

In this section, we recall fundamental functional inequalities satisfied by the target measure $\pi_{\beta}$ under the assumptions of smoothness and dissipativity. These inequalities are central to establishing the convergence of the corresponding Langevin dynamics. First, recall the Kullback--Leibler divergence.

\begin{definition}[Kullback--Leibler divergence]
Let $\mu$ and $\nu$ be probability measures on $\mathbb{R}^d$. The Kullback--Leibler divergence, or relative entropy, of $\mu$ with respect to $\nu$ is defined by
\[
    \KL(\mu\|\nu)
    :=
    \begin{cases}
    \displaystyle
    \int_{\mathbb{R}^d}
    \log\left(\frac{d\mu}{d\nu}\right)\,d\mu,
    & \text{if } \mu\ll\nu, \\[1.2em]
    +\infty,
    & \text{otherwise}.
    \end{cases}
\]
Equivalently, if $r=\frac{d\mu}{d\nu}$, then
\[
    \KL(\mu\|\nu)
    =
    \int_{\mathbb{R}^d} r\log r\,d\nu,
\]
with the convention $0\log 0=0$.
\end{definition}

The main inequality we consider and exploit is the logarithmic Sobolev inequality. For a probability measure $\pi$ on $\mathbb{R}^d$, consider
the quadratic form
\[
    \mathcal{E}_{\pi}(g,g)
    :=
    \int_{\mathbb{R}^d} \|\nabla g\|^2 \, d\pi,
    \qquad
    g \in C_c^\infty(\mathbb{R}^d).
\]
Whenever this form is closable in $L^2(\pi)$, we denote by
$H^1(\pi)$ the domain of its closure and by $\nabla g$ the associated
weak gradient. All the Gibbs measures considered in this paper have
strictly positive continuous densities, locally bounded above and below, and therefore the corresponding quadratic forms are closable. For $g \in L^2(\pi)$, define
\[
    \operatorname{Ent}_{\pi}(g^2)
    :=
    \int_{\mathbb{R}^d} g^2 \log(g^2) \, d\pi
    -
    \left(\int_{\mathbb{R}^d} g^2 \, d\pi\right)
    \log\left(\int_{\mathbb{R}^d} g^2 \, d\pi\right),
\]
with the convention $0\log 0=0$.

\begin{definition}[Logarithmic Sobolev inequality]
A probability measure $\pi$ is said to satisfy a logarithmic Sobolev
inequality with constant $\CLS>0$ if
\begin{equation}\label{eq:lsi}
    \operatorname{Ent}_{\pi}(g^2)
    \leq
    2\CLS
    \int_{\mathbb{R}^d} \|\nabla g\|^2 \, d\pi,
    \qquad
    g \in H^1(\pi).
\end{equation}
\end{definition}

For a probability measure $\mu \ll \pi$, we define its relative Fisher information with respect to $\pi$ by
\[
    I(\mu\|\pi)
    :=
    \begin{cases}
    4\displaystyle\int_{\mathbb{R}^d}
        \left\|\nabla\sqrt{\frac{d\mu}{d\pi}}\right\|^2\,d\pi,
        & \sqrt{d\mu/d\pi}\in H^1(\pi),\\
    +\infty,
        & \sqrt{d\mu/d\pi}\notin H^1(\pi).
    \end{cases}
\]
With this extended-value definition, the logarithmic Sobolev inequality
is equivalently written as
\[
    \KL(\mu\|\pi)
    \leq
    \frac{\CLS}{2}\,I(\mu\|\pi),
    \qquad
    \mu\ll\pi.
\]
Indeed, the measure formulation follows from~\eqref{eq:lsi} by taking
$g=\sqrt{d\mu/d\pi}$, while the converse follows by normalizing
$g^2\pi$ to be a probability measure. If $d\mu/d\pi$ is strictly positive and sufficiently regular for the chain rule
to apply, for instance if $d\mu/d\pi\in C^1(\mathbb{R}^d)$ and the quantities
below are finite, then
\[
    I(\mu\|\pi)
    =
4\int_{\mathbb{R}^d}\left\|\nabla\sqrt{\frac{d\mu}{d\pi}}\right\|^2\,d\pi
    =
    \int_{\mathbb{R}^d}\left\|\nabla\log \frac{d\mu}{d\pi}\right\|^2\,d\mu.
\]
In general, the square-root expression together with the extended-value
convention above is the definition used throughout the paper.

To place the logarithmic Sobolev inequality within the standard hierarchy
of functional inequalities, we next recall Talagrand's transportation
inequality and the Poincaré inequality. Let $\mathcal{P}_2(\mathbb{R}^d)$ denote the set of probability measures
on $\mathbb{R}^d$ with finite second moment.

\begin{definition}[Talagrand's Transportation Inequality]
    A probability measure $\pi \in \mathcal{P}_2(\mathbb{R}^d)$ is said
    to satisfy Talagrand's $T_2$ transportation inequality with constant
    $C_T > 0$ if, for every
    $\mu \in \mathcal{P}_2(\mathbb{R}^d)$ such that $\mu \ll \pi$, it holds that
    \begin{equation}\label{eq:talagrand}
        W_2^2(\mu,\pi)
        \leq
        2 C_T \KL(\mu\|\pi),
    \end{equation}
    where $W_2(\mu,\pi)$ denotes the $L^2$-Wasserstein distance between
    $\mu$ and $\pi$ \cite{Villani2009}.
\end{definition}

\begin{definition}[Poincaré Inequality] A probability measure $\pi$ is said to satisfy a Poincaré inequality with constant $C_P > 0$ if, for all functions $g \in C_c^\infty(\mathbb{R}^d)$, it holds that \begin{equation}\label{eq:poincare} \operatorname{Var}_{\pi}(g) \leq C_P \int_{\mathbb{R}^d} \|\nabla g\|^2 \, d\pi, \end{equation} where $\operatorname{Var}_{\pi}(g) = \int_{\mathbb{R}^d} g^2 d\pi- \left(\int_{\mathbb{R}^d} g d\pi\right)^2$. \end{definition}

The logarithmic Sobolev, Talagrand, and Poincaré inequalities are linked
by a classical hierarchy of implications. In particular, by the
Otto--Villani theorem \cite[Theorem~9.6.1]{Bakry2014}, if $\pi_\beta$
satisfies a logarithmic Sobolev inequality with constant
$\CLS$, then it satisfies Talagrand's $T_2$ inequality with
$C_T \leq \CLS$. Talagrand's $T_2$ inequality, in turn,
implies the Poincaré inequality.

\begin{remark}[Connection to the Spectral Gap]\label{rem:spectral_gap} In the context of the Langevin diffusion $dX_t = -\nabla F(X_t)dt + \sqrt{2}dW_t$, the Poincaré constant $C_P$ is intimately tied to the spectral gap of the associated infinitesimal generator $\mathcal{L} = \Delta - \langle \nabla F, \nabla \rangle$. Specifically, if $\lambda^*$ denotes the global spectral gap of the operator $-\mathcal{L}$ in $L^2(\pi)$ (as defined, for instance, in Raginsky et al.\ \cite{Raginsky2017}), the Poincaré constant is exactly its inverse: $C_P = \frac{1}{\lambda^*}$. Thus, bounding $C_P$ is equivalent to establishing an exponential rate of convergence to equilibrium in $L^2(\pi)$ variance. \end{remark}

Relative entropy also provides a general control of total variation, independently of the functional-inequality assumptions above.

\begin{definition}[Pinsker's Inequality]
    For any two probability measures $\mu$ and $\pi$, Pinsker's inequality
    gives
    \begin{equation}\label{eq:pinsker}
        \|\mu-\pi\|_{\mathrm{TV}}
        \leq
        \sqrt{2\,\KL(\mu\|\pi)},
    \end{equation}
    where, for a finite signed measure $\eta$, $\|\eta\|_{\mathrm{TV}}
        :=
        |\eta|(\mathbb{R}^d)$
    denotes its total variation norm.
\end{definition}

\subsection{Bound on the Logarithmic Sobolev Constant}\label{sec:CLS_bounds} 

To bound the Log-Sobolev constant, we rely on the Poincaré inequality defined in Appendix \ref{app:inequalities}. As established by Raginsky et al.\ \cite[Appendix E]{Raginsky2017}, $\CLS(\beta, d)$ can be explicitly bounded in terms of the problem parameters and the Poincaré constant $C_P(\beta, d)$ of the measure $\pi_{\beta}$. For a $C^2$, $M$-smooth and $(m,b)$-dissipative potential $F$, the Gibbs measure satisfies a Log-Sobolev inequality with the constant
\begin{equation}\label{eq:CLS_bound}
    \CLS (\beta, d) \leq \frac{2m^2+8M^2}{\beta m^2 M} + C_P(\beta, d)\left(\frac{6M(d+\beta b)}{m}+2\right)=:C_{\mathrm{R}}(\beta,d).
    \end{equation}
Without further geometric assumptions, such as global strong convexity, the magnitude of $\CLS(\beta, d)$ is dominated by the Poincaré constant $C_P(\beta,d)$ (which directly corresponds to the inverse of the spectral gap, $\lambda_*^{-1}$, in the notation of \cite{Raginsky2017}). Slightly correcting the parameter assignment immediately after
\cite[Equation~(B.1)]{Raginsky2017}, the Lyapunov criterion
\cite[Theorem~1.4]{BakryBartheCattiauxGuillin2008}
(see also \cite[Proposition~13]{Raginsky2017}) applies with
$\lambda_0=\kappa$ and $\kappa_0=2\kappa$ (in the notation of \cite{Raginsky2017}), and gives
\[
C_P(\pi_\beta)
\leq
\frac1\kappa
\left(1+2\kappa C_P(\pi_{\beta,R})\right),
\]
where $\pi_{\beta,R}$ is the normalized restriction of $\pi_\beta$
to $B(0,R)$. The Payne--Weinberger inequality \cite{PayneWeinberger1960}, followed
by the standard bounded-density comparison, gives
\[
C_P(\pi_{\beta,R})
\leq
\frac{4R^2}{\pi^2}
e^{\operatorname{Osc}_R(\beta F)}.
\]
Moreover, $M$-smoothness and $2R\leq R^2+1$ imply
\[
\operatorname{Osc}_R(\beta F)
\leq
\beta\left(
\frac{M+B}{2}R^2+A+B-C
\right)
=
\frac{2(M+B)}{m}(d+\beta b)+\beta(A+B-C),
\]
where $A=|F(0)|$, $B:=\|\nabla F(0)\|$ and
$C:=\min_{x\in\mathbb R^d}F(x)$. Consequently,
\begin{equation}\label{eq:CP_bound}
C_P(\beta,d)
\leq
\frac{2}{m\beta(d+\beta b)}
+
\frac{32(d+\beta b)}{\pi^2m\beta}
\exp\left(
\frac{2(M+B)}{m}(d+\beta b)
+\beta(A+B-C)
\right).
\end{equation}

Combining \eqref{eq:CLS_bound} with \eqref{eq:CP_bound} reveals that the dependence of $\CLS(\beta, d)$ on $\beta$ and $d$ is exponential. Asymptotically, as $\beta, d \to \infty$, we derive the following general scaling behavior:
\[\CLS (\beta, d)\leq  \tilde{\mathcal{O}}\left(\exp\left(\left[\frac{2b}{m}(M+B)+A+B\right]\beta + \frac{2}{m}(M+B) d\right)\right).\]

As highlighted in the main text, while Raginsky et al.\ \cite{Raginsky2017} state these bounds for $M$-smooth potentials, their underlying reliance on the Bakry-Cattiaux criteria implicitly assumes that $F$ is $C^2$ (see \cite[Proposition 15]{Raginsky2017} or \cite[Lemma B.6]{habring2026inhomogeneous} where this gap is still present). Because our framework avoids this stronger regularity assumption, we cannot invoke their LSI result directly. To bridge this gap and ensure the constants hold for functions that are merely $C^{1,1}$, we provide a proof of the Logarithmic Sobolev Inequality via a mollification argument.

\begin{proof}[Proof of Proposition \ref{prop:LSI}]
We argue by mollification. Let $\rho\in C_c^\infty(\mathbb R^d)$ be nonnegative,
even, supported in the unit ball, and normalized by $\int_{\mathbb R^d}\rho(y)\,dy=1$. Since $\rho$ is even, we also have $\int_{\mathbb R^d} y\,\rho(y)\,dy=0$. For $\epsilon>0$, set
\[
    \rho_\epsilon(y):=\epsilon^{-d}\rho(y/\epsilon),
    \qquad
    F_\epsilon(x):=(\rho_\epsilon*F)(x)
    =
    \int_{\mathbb R^d}\rho_\epsilon(y)F(x-y)\,dy.
\]
Finally define $s_\rho^2:=\int_{\mathbb R^d}\|y\|^2\rho(y)\,dy$, which implies $\int_{\mathbb R^d}\|y\|^2\rho_\epsilon(y)\,dy
    =
    \epsilon^2s_\rho^2$. Since $F$ is $C^1$ and $\rho_\epsilon$ is $C^{\infty}_c$,
$F_\epsilon\in C^\infty(\mathbb R^d)$, and it holds
\[
    \nabla F_\epsilon(x)
    =
    \int_{\mathbb R^d}\rho_\epsilon(y)\nabla F(x-y)\,dy .
\]
Indeed, this follows by differentiating under the integral sign, which is justified
because $\nabla F$ is continuous and the integration is over the compact set
$x-\operatorname{supp}\rho_\epsilon$. For every $x,z\in\mathbb R^d$, we have
\[
\begin{aligned}
    \|\nabla F_\epsilon(x)-\nabla F_\epsilon(z)\|
    &=
    \left\|
        \int_{\mathbb R^d}\rho_\epsilon(y)
        \bigl(\nabla F(x-y)-\nabla F(z-y)\bigr)\,dy
    \right\|\le
    \int_{\mathbb R^d}\rho_\epsilon(y)
    \|\nabla F(x-y)-\nabla F(z-y)\|\,dy          \\
    &\le
    M\|x-z\|\int_{\mathbb R^d}\rho_\epsilon(y)\,dy
    =
    M\|x-z\|.
\end{aligned}
\]
Thus $F_\epsilon$ is $M$-smooth. Since $F_\epsilon$ is $C^{\infty}$, this implies $-MI\preceq \nabla^2F_\epsilon(x)\preceq MI$, for all $x\in\mathbb R^d$. In particular, the potential $\beta F_\epsilon$ satisfies $\nabla^2(\beta F_\epsilon)(x)\succeq -\beta M I$. We next prove dissipativity. For every $x\in\mathbb R^d$,
\[
\begin{aligned}
    \langle x,\nabla F_\epsilon(x)\rangle
    &=
    \int_{\mathbb R^d}\rho_\epsilon(y)
    \langle x,\nabla F(x-y)\rangle\,dy=
    \int_{\mathbb R^d}\rho_\epsilon(y)
    \langle x-y,\nabla F(x-y)\rangle\,dy + \int_{\mathbb R^d}\rho_\epsilon(y)
    \langle y,\nabla F(x-y)\rangle\,dy .
\end{aligned}
\]
The first integral is bounded from below by dissipativity of $F$:
\[
\begin{aligned}
    \int_{\mathbb R^d}\rho_\epsilon(y)
    \langle x-y,\nabla F(x-y)\rangle\,dy
    &\ge
    \int_{\mathbb R^d}\rho_\epsilon(y)
    \bigl(m\|x-y\|^2-b\bigr)\,dy=m\int_{\mathbb R^d}\rho_\epsilon(y)\|x-y\|^2\,dy-b.
\end{aligned}
\]
Since $\int y\rho_\epsilon(y)\,dy=0$,
\[
    \int_{\mathbb R^d}\rho_\epsilon(y)\|x-y\|^2\,dy
    =
    \|x\|^2+\int_{\mathbb R^d}\rho_\epsilon(y)\|y\|^2\,dy
    =
    \|x\|^2+\epsilon^2s_\rho^2 .
\]
Therefore
\[
    \int_{\mathbb R^d}\rho_\epsilon(y)
    \langle x-y,\nabla F(x-y)\rangle\,dy
    \ge
    m\|x\|^2
    +
    m\epsilon^2s_\rho^2
    -
    b.
\]
For the second integral, using again $\int y\rho_\epsilon(y)\,dy=0$, we write
\[
\begin{aligned}
    \int_{\mathbb R^d}\rho_\epsilon(y)
    \langle y,\nabla F(x-y)\rangle\,dy
    &=
    \int_{\mathbb R^d}\rho_\epsilon(y)
    \langle y,\nabla F(x-y)-\nabla F(x)\rangle\,dy .
\end{aligned}
\]
By the $M$-Lipschitz continuity of $\nabla F$,
\[
\begin{aligned}
    \int_{\mathbb R^d}\rho_\epsilon(y)
    \langle y,\nabla F(x-y)-\nabla F(x)\rangle\,dy
    &\ge
    -
    \int_{\mathbb R^d}\rho_\epsilon(y)
    \|y\|\,\|\nabla F(x-y)-\nabla F(x)\|\,dy       \\
    &\ge
    -
    M\int_{\mathbb R^d}\rho_\epsilon(y)\|y\|^2\,dy =
    -M\epsilon^2s_\rho^2 .
\end{aligned}
\]
Combining the two estimates gives
\[
    \langle x,\nabla F_\epsilon(x)\rangle
    \ge
    m\|x\|^2-b-(M-m)\epsilon^2s_\rho^2 .
\]
Thus $F_\epsilon$ is $(m_\epsilon,b_\epsilon)$-dissipative with $m_\epsilon=m$ and $b_\epsilon=b+(M-m)\epsilon^2s_\rho^2$. In particular we have $M_\epsilon=M$, $m_\epsilon = m$ and $b_\epsilon\to b$. We now compare $F_\epsilon$ and $F$ uniformly. By $M$-smoothness, for all
$x,y\in\mathbb R^d$,
\[
    \left|
        F(x-y)-F(x)+\langle \nabla F(x),y\rangle
    \right|
    \le
    \frac M2\|y\|^2 .
\]
Therefore, using $\int y\rho_\epsilon(y)\,dy=0$,
\[
\begin{aligned}
    |F_\epsilon(x)-F(x)|
    &=
    \left|
        \int_{\mathbb R^d}\rho_\epsilon(y)
        \bigl(F(x-y)-F(x)\bigr)\,dy
    \right| = \left|
        \int_{\mathbb R^d}\rho_\epsilon(y)
        \bigl(F(x-y)-F(x)+\langle\nabla F(x),y\rangle\bigr)\,dy
    \right|                                                   \\
    &\le
    \frac M2
    \int_{\mathbb R^d}\rho_\epsilon(y)\|y\|^2\,dy
    =
    \frac M2\,\epsilon^2s_\rho^2 .
\end{aligned}
\]
Set $D_\epsilon:=\frac M2\,\epsilon^2s_\rho^2$. Then $\|F_\epsilon-F\|_{L^\infty(\mathbb R^d)}
    \le D_\epsilon$ with $D_\epsilon\to0$. Let
\[
    Z_\epsilon
    :=
    \int_{\mathbb R^d}e^{-\beta F_\epsilon(x)}\,dx,
    \qquad
    d\pi_\epsilon(x)
    :=
    Z_\epsilon^{-1}e^{-\beta F_\epsilon(x)}\,dx.
\]
Since $|F_\epsilon-F|\le D_\epsilon$, we have
\[
    e^{-\beta D_\epsilon}e^{-\beta F(x)}
    \le
    e^{-\beta F_\epsilon(x)}
    \le
    e^{\beta D_\epsilon}e^{-\beta F(x)} .
\]
Hence $e^{-\beta D_\epsilon}Z
    \le
    Z_\epsilon
    \le
    e^{\beta D_\epsilon}Z$. In particular, $Z_\epsilon<\infty$. Moreover, the density ratio satisfies
\[
\begin{aligned}
    \frac{d\pi_\epsilon}{d\pi_{\beta}}(x)
    &=
    \frac{Z}{Z_\epsilon}
    e^{-\beta(F_\epsilon(x)-F(x))}.
\end{aligned}
\]
Using the previous bounds on $Z_\epsilon/Z$ and on $F_\epsilon-F$, we obtain
\[
    e^{-2\beta D_\epsilon}
    \le
    \frac{d\pi_\epsilon}{d\pi_{\beta}}(x)
    \le
    e^{2\beta D_\epsilon},
    \qquad x\in\mathbb R^d.
\]
Since
\[\pi_\epsilon
=
\frac{e^{-\beta(F_\epsilon-F)}}{\int e^{-\beta(F_\epsilon-F)(y)}\, d\pi_{\beta}(y)}\,\pi_{\beta},
\]
the measure $\pi_\epsilon$ is a bounded perturbation of $\pi_{\beta}$ with perturbation
potential $G_\epsilon=\beta(F_\epsilon-F)$. As
\[
\operatorname{Osc}(G_\epsilon) := \sup G_\epsilon - \inf G_\epsilon
\le
2\beta\|F_\epsilon-F\|_\infty
\le
2\beta D_\epsilon,
\]
the Holley--Stroock perturbation criterion \cite{holley1987} (see also \cite[Theorem 1.1]{Cattiaux2022}) yields
\[
C_P(\pi_\epsilon)
\le
e^{\operatorname{Osc}(G_\epsilon)}C_P(\pi_{\beta})
\le
e^{2\beta D_\epsilon}C_P(\pi_{\beta}).
\]
In particular, $\limsup_{\epsilon\to 0} C_P(\pi_\epsilon)
\le C_P(\pi_{\beta})=:C_P(\beta,d)$. We may now apply the result in Raginsky et al. \cite[Appendix A]{Raginsky2017} to the $C^{\infty}$ potential $F_\epsilon$. The function $F_\epsilon$ is $C^\infty$,
$M_\epsilon$-smooth with $M_\epsilon=M$, and
$(m_\epsilon,b_\epsilon)$-dissipative with $m_\epsilon=m$ and $b_\epsilon=b+(M-m)\epsilon^2s_\rho^2$. Moreover, $\pi_\epsilon$ satisfies a Poincaré inequality with constant at most
$e^{2\beta D_\epsilon}C_P(\beta,d)$. Hence $\pi_\epsilon$ satisfies a logarithmic
Sobolev inequality with constant that can be bounded by
\[
    C_\epsilon
    =
    \frac{2m_\epsilon^2+8M_\epsilon^2}
         {m_\epsilon^2M_\epsilon\beta}
    +
    \left(
        \frac{6M_\epsilon}{m_\epsilon}
        (d+\beta b_\epsilon)
        +2
    \right)
    e^{2\beta D_\epsilon}C_P(\beta,d) .
\]
Since $m_\epsilon=m$, $M_\epsilon=M$,
$b_\epsilon\to b$, and $D_\epsilon\to0$, $C_\epsilon$ converges to the constant $C_{\mathrm{R}}(\beta,d)$ defined in \eqref{eq:CLS_bound}.
It remains to transfer the logarithmic Sobolev inequality from
$\pi_\epsilon$ to $\pi_\beta$. Recall that
\[
    \frac{d\pi_\beta}{d\pi_\epsilon}(x)
    =
    \frac{Z_\epsilon}{Z}
    \exp\bigl(\beta(F_\epsilon(x)-F(x))\bigr).
\]
Thus, $\pi_\beta$ is a bounded perturbation of $\pi_\epsilon$, and the oscillation of the corresponding perturbation potential is bounded by
\[\operatorname{Osc}\bigl(\beta(F- F_\epsilon)\bigr)
    \leq
    2\beta\|F_\epsilon-F\|_\infty
    \leq
    2\beta D_\epsilon.\]
Applying the Holley--Stroock perturbation principle to the logarithmic
Sobolev inequality satisfied by $\pi_\epsilon$ (see again \cite{holley1987} or \cite[Theorem 1.1]{Cattiaux2022}) shows that $\pi_\beta$
satisfies a logarithmic Sobolev inequality with constant $e^{2\beta D_\epsilon}C_\epsilon$. Since $D_\epsilon\to0$ and
$C_\epsilon\to C_{\mathrm{R}}(\beta,d)$ as $\epsilon\to0$, we conclude that
$\pi_\beta$ satisfies a logarithmic Sobolev inequality with constant
$C_{\mathrm{R}}(\beta,d)$. This proves Proposition~\ref{prop:LSI}.\end{proof}

\section{Proofs for Section \ref{sec:Langevin}}\label{app:proofs_Lan}

\subsection{Proofs of preliminary results}\label{app:proof_preliminaries}

\begin{proof}[Proof of Lemma \ref{lem:quad_growth}]\label{proof:quad_growth}
    (i) Let $x$ with $\|x\|=1$. Then, for every $t>0$, 
    \[\frac{d}{dt}f(tx) = \langle x,\nabla f(tx) \rangle\geq \frac{mt^2-b}{t}.\]
    Therefore, for any $t\geq 1$ if holds
    \[f(tx)-f(x)= \int_1^t \frac{d}{ds}f(sx)\, ds \geq \frac{m}{2} (t^2-1) - b\ln(t).\]
    Since $f$ is continuous it attains a minimum on the sphere $\{x\in \R^d\mid\|x\|=1\}$, so that
    \[f(tx)\geq \min_{\|x\|=1}f(x) + \frac{m}{2} (t^2-1) - b\ln(t).\]
    This implies that $f$ is coercive and therefore attains its minimum in $\R^d$.\\
    (ii) For any $x\in \R^d$ we have $\|\nabla f(x)\|\leq \|\nabla f(x)-\nabla f(0)\|+\|\nabla f(0)\|\leq L \|x\| + \|\nabla f(0)\| $. 
    (iii) The proof follows the same line of \cite[Lemma 2]{Raginsky2017}.
\end{proof}

\begin{proof}[Proof of Corollary \ref{cor:abs-quadratic-bound}]
By Lemma~\ref{lem:quad_growth} (iii), we have $f(x)\le \frac{L}{2}\|x\|^2 + B\|x\| + A$, for every $x\in\R^d$.
By Young's inequality we have  $B\|x\|\le \frac12\|x\|^2+\frac{B^2}{2}$. Hence,
\begin{equation*}
f(x)\le \frac{L+1}{2}\|x\|^2 + A+\frac{B^2}{2}.    
\end{equation*}
On the other hand, let $C  := \min\limits_{x \in \mathbb{R}^d} f(x)$. Therefore, for every $x \in \mathbb{R}^d$ we have $f(x) \geq C$ and so $-f(x) \leq -C \leq|C|$. 
Therefore,
\begin{equation*}
|f(x)|=\max\{f(x),-f(x)\}
\le \max\Big\{\frac{L+1}{2}\|x\|^2 + A+\frac{B^2}{2},\,|C|\Big\}
\le \frac{L+1}{2}\|x\|^2 + A+\frac{B^2}{2} + |C|,    
\end{equation*}
which concludes the proof.
\end{proof}

\subsection{Proof of Lemma~\ref{lem:exp_bound_pi}}\label{app:proof_exp_bound_pi}

\begin{proof}[Proof of Lemma~\ref{lem:exp_bound_pi}]
For any smooth, sufficiently integrable test function $g: \mathbb{R}^d \to \mathbb{R}$, integration by parts against the Gibbs measure $\pi_{\beta}$ yields the standard identity
\begin{equation}\label{eq:integration_by_parts_pi}
    \int_{\mathbb{R}^d} \Delta g(y) \, d\pi_{\beta}(y) = \frac{1}{Z}\int \Delta g \, e^{-\beta F} \, dy =\frac{\beta}{Z}\int \langle \nabla F(y), \nabla g(y) \rangle \, e^{-\beta F(y)} \, dy= \beta \int_{\mathbb{R}^d} \langle \nabla F(y), \nabla g(y) \rangle \, d\pi_{\beta}(y).
\end{equation}
We apply this identity to the function $g(y) = e^{\alpha \|y\|^2}$ for a parameter $\alpha \in [0, 1]$. This is justified by Lemma~\ref{lem:quad_growth}(iii), which implies that there exists $c\in\R$ such that
$F(y)\ge \frac{m}{3}\|y\|^2-c$ for all $y$, hence
$e^{-\beta F(y)}\le e^{\beta c} e^{-(\beta m/3)\|y\|^2}$.
Therefore, for every $\alpha\in[0,1]$ and $\beta\ge 4/m$, we have
\[\int_{\R^d} e^{\alpha\|y\|^2}\,d\pi_{\beta}(y) \leq \frac{e^{\beta c}}{Z}\int e^{(\alpha-\beta m /3)\|y\|^2} \, dy\leq \frac{e^{\beta c}}{Z}\int e^{-\|y\|^2/3} \, dy<\infty,\]
and
\[\int_{\R^d}\|y\|^2 e^{\alpha\|y\|^2}\,d\pi_{\beta}(y) \leq \frac{e^{\beta c}}{Z}\int \|y\|^2 e^{(\alpha-\beta m /3)\|y\|^2} \, dy\leq \frac{e^{\beta c}}{Z}\int \|y\|^2 e^{-\|y\|^2/3} \, dy<\infty.\]
Consequently, since the gradient and Laplacian of $g$ are given respectively by
\begin{equation}\label{eq:Grad_Lap_of_g}
    \nabla g(y) = 2\alpha y e^{\alpha \|y\|^2}, \quad \Delta g(y) = 2\alpha(d + 2\alpha \|y\|^2) e^{\alpha \|y\|^2},
\end{equation}
and $\|\nabla F(y)\|\leq M\|y\|+B$ by Lemma~\ref{lem:quad_growth}(ii), we may apply \eqref{eq:integration_by_parts_pi} to $g_R(y)=\chi_R(y)e^{\alpha\|y\|^2}$ with $\chi_R\in C_c^\infty$ a standard cutoff,
and let $R\to\infty$ to obtain \eqref{eq:integration_by_parts_pi} for $g(y)=e^{\alpha\|y\|^2}$. To rigorously justify this, let $\chi \in C_c^\infty(\mathbb{R}^d)$ be a base cutoff function such that $\chi \equiv 1$ on $B_1$ (where $B_r$ denotes the ball of radius $r$) and $\chi \equiv 0$ on $\mathbb{R}^d \setminus B_2$. We define the scaled cutoff $\chi_R(y) = \chi(y/R)$, which ensures $\chi_R \equiv 1$ on $B_R$, $\chi_R \equiv 0$ outside $B_{2R}$, and by the chain rule, $\|\nabla\chi_R\|_\infty \le C_1 R^{-1}$ and $\|\Delta\chi_R\|_\infty \le C_2 R^{-2}$. Set the compactly supported test function $g_R = \chi_R g$. Applying \eqref{eq:integration_by_parts_pi} to $g_R$ yields
\begin{equation}\label{eq:ibp_cutoff}
    \int_{\mathbb{R}^d} \left( \chi_R \Delta g + 2\langle \nabla \chi_R, \nabla g \rangle + g \Delta \chi_R \right) d\pi_\beta= \beta \int_{\mathbb{R}^d} \left( \chi_R \langle \nabla F, \nabla g \rangle + g \langle \nabla F, \nabla \chi_R \rangle \right) d\pi_\beta.
\end{equation}
Because $\chi_R$ is constant everywhere except on $\mathcal{A}_R = \{y \in \mathbb{R}^d : R \le \|y\| \le 2R\}$, the derivatives $\nabla \chi_R$ and $\Delta \chi_R$ are supported entirely on $\mathcal{A}_R$. Using the gradient bound $\|\nabla F(y)\| \le M\|y\| + B$, we can bound the absolute value of the error integrands on $\mathcal{A}_R$:
\begin{align*}
    |g \Delta \chi_R| &\lesssim R^{-2} e^{\alpha\|y\|^2} \le e^{\alpha\|y\|^2}, \\
    |\langle \nabla \chi_R, \nabla g \rangle| &\lesssim R^{-1} (2\alpha\|y\|) e^{\alpha\|y\|^2} \lesssim e^{\alpha\|y\|^2}, \\
    |g \langle \nabla F, \nabla \chi_R \rangle| &\lesssim R^{-1} (M\|y\|+B) e^{\alpha\|y\|^2} \lesssim e^{\alpha\|y\|^2},
\end{align*}
where we used the fact that $\|y\| \le 2R$ on the support $\mathcal{A}_R$ to cancel the respective $R^{-1}$ dependencies. Consequently, all error terms are uniformly dominated by $C e^{\alpha\|y\|^2}$. Since $\int e^{\alpha\|y\|^2} d\pi_\beta< \infty$ and the measure of $\mathcal{A}_R$ vanishes as $R \to \infty$, the dominated convergence theorem guarantees that the integrals of these error terms converge to $0$. For the main terms $\chi_R \Delta g$ and $\chi_R \langle \nabla F, \nabla g \rangle$, we observe they converge pointwise to $\Delta g$ and $\langle \nabla F, \nabla g \rangle$. To apply the dominated convergence theorem, we bound them uniformly in $R$ using \eqref{eq:Grad_Lap_of_g}
\begin{align*}
    |\chi_R \Delta g| &\le 2\alpha(d + 2\alpha \|y\|^2) e^{\alpha \|y\|^2}, \\
    |\chi_R \langle \nabla F, \nabla g \rangle| &\le \|\nabla F\| \|\nabla g\| \le (M\|y\|+B)(2\alpha \|y\|) e^{\alpha \|y\|^2} = 2\alpha(M\|y\|^2 + B\|y\|) e^{\alpha \|y\|^2}.
\end{align*}
Using the elementary inequality $\|y\| \le \frac{1}{2}(1+\|y\|^2)$, both integrands are globally dominated by $C(1 + \|y\|^2) e^{\alpha \|y\|^2}$ for some constant $C > 0$. Because we have already established that both $\int e^{\alpha\|y\|^2} d\pi_\beta< \infty$ and $\int \|y\|^2 e^{\alpha\|y\|^2} d\pi_\beta< \infty$, this dominating function is $\pi_{\beta}$-integrable. Thus, the dominated convergence theorem ensures that the integrals of the main terms converge to their un-truncated counterparts, fully recovering \eqref{eq:integration_by_parts_pi} for $g(y) = e^{\alpha\|y\|^2}$. Substituting the expressions in \eqref{eq:Grad_Lap_of_g} into \eqref{eq:integration_by_parts_pi} and dividing by $2\alpha > 0$, we obtain
\[
    \int_{\mathbb{R}^d} (d + 2\alpha \|y\|^2) e^{\alpha \|y\|^2} \, d\pi_{\beta}(y) = \beta \int_{\mathbb{R}^d} \langle \nabla F(y), y \rangle e^{\alpha \|y\|^2} \, d\pi_{\beta}(y).
\]
Applying the dissipativity condition $\langle \nabla F(y), y \rangle \ge m\|y\|^2 - b$ to the right-hand side yields
\[
    \int_{\mathbb{R}^d} (d + 2\alpha \|y\|^2) e^{\alpha \|y\|^2} \, d\pi_{\beta}(y) \ge \beta \int_{\mathbb{R}^d} (m\|y\|^2 - b) e^{\alpha \|y\|^2} \, d\pi_{\beta}(y).
\]
Rearranging the terms to group the integrals with respect to $\|y\|^2 e^{\alpha \|y\|^2}$ and $e^{\alpha \|y\|^2}$, we get the inequality
\begin{equation}\label{eq:diff_ineq_alpha}
    (\beta m - 2\alpha) \int_{\mathbb{R}^d} \|y\|^2 e^{\alpha \|y\|^2} \, d\pi_{\beta}(y) \le (d + \beta b) \int_{\mathbb{R}^d} e^{\alpha \|y\|^2} \, d\pi_{\beta}(y).
\end{equation}
Let $H(\alpha) = \int_{\mathbb{R}^d} e^{\alpha \|y\|^2} \, d\pi_{\beta}(y)$. We recognize the integral on the left side as the derivative $H'(\alpha)$. Since $\beta \ge 4/m$ and $\alpha \le 1$, we have $\beta m - 2\alpha \ge 4 - 2 = 2 > 0$. We can thus safely divide both sides by $(\beta m - 2\alpha)$ to form a differential inequality
\[
    \frac{H'(\alpha)}{H(\alpha)} \le \frac{d + \beta b}{\beta m - 2\alpha}.
\]
Integrating this inequality with respect to $\alpha$ from $0$ to $1$ gives
\[
    \log H(1) - \log H(0) \le \int_0^1 \frac{d + \beta b}{\beta m - 2\alpha} \, d\alpha.
\]
By definition, $H(0) = \int_{\mathbb{R}^d} 1 \, d\pi_{\beta}(y) = 1$, so $\log H(0) = 0$. Evaluating the integral on the right-hand side, we find
\[
    \log H(1) \le \left[ -\frac{d + \beta b}{2} \log(\beta m - 2\alpha) \right]_0^1 = \frac{d + \beta b}{2} \log\left( \frac{\beta m}{\beta m - 2} \right).
\]
Using the elementary inequality $\log(1 + x) \le x$ for $x > 0$, we can bound the logarithmic term:
\[
    \log\left(\frac{\beta m}{\beta m - 2} \right)=\log\left( 1 + \frac{2}{\beta m - 2} \right) \le \frac{2}{\beta m - 2}.
\]
Substituting this into our bound yields
\[
    \log \int_{\mathbb{R}^d} e^{\|y\|^2} \, d\pi_{\beta}(y) \le \frac{d + \beta b}{2} \cdot \frac{2}{\beta m - 2} = \frac{d + \beta b}{\beta m - 2}.
\]
Finally, since the assumption $\beta \ge 4/m$ implies $\beta m \ge 4$, it follows that $\beta m - 2 \ge \frac{1}{2}\beta m$. Using this lower bound on the denominator, we arrive at the final result:
\[
    \log \int_{\mathbb{R}^d} e^{\|y\|^2} \, d\pi_{\beta}(y) \le \frac{d + \beta b}{\beta m / 2} = \frac{2(d + \beta b)}{\beta m}.
\]
This completes the proof.
\end{proof}

\subsection{Proof of Lemma~\ref{lem:KL_mu0_pi}}\label{app:proof_KL_mu0_pi}

\begin{proof}[Proof of Lemma~\ref{lem:KL_mu0_pi}]
The proof follows \cite[Lemma~5]{Raginsky2017}; we include the details because their result is stated for an empirical risk formed from nonnegative, smooth, and dissipative loss functions, whereas we impose smoothness and dissipativity directly on a general objective $F$, whose lower quadratic bound contains the additional constant $C=\min F$. Let $p$ denote the density of the Gibbs measure with respect to the Lebesgue measure on $\mathbb{R}^d$, i.e., $p(x) = e^{-\beta F(x)}/Z$, where $Z = \int_{\mathbb{R}^d} e^{-\beta F(x)} dx$. Since $p > 0$ everywhere, we can write
\begin{equation}\label{eq:KL_decomp}
\begin{aligned}
\KL(\mu_0||\pi_{\beta}) &= \int_{\mathbb{R}^d} p_0(x) \log \frac{p_0(x)}{p(x)} dx= \int_{\mathbb{R}^d} p_0(x) \log p_0(x) dx + \log Z + \beta \int_{\mathbb{R}^d} p_0(x) F(x) dx \\
&\le \log \|p_0\|_\infty + \log Z + \beta \int_{\mathbb{R}^d} p_0(x) F(x) dx.
\end{aligned}
\end{equation}
We first upper-bound the partition function $Z$
\begin{align*}
Z &= \int_{\mathbb{R}^d} e^{-\beta F(x)} dx \le \int_{\mathbb{R}^d} \exp\left( -\beta \left( \frac{m}{3}\|x\|^2 - \frac{b}{2}\log 3 + C \right) \right) dx \\
&= e^{\frac{\beta b}{2}\log 3 - \beta C} \int_{\mathbb{R}^d} e^{-\frac{m\beta}{3}\|x\|^2} dx = 3^{\beta b/2} e^{-\beta C} \left( \frac{3\pi}{m\beta} \right)^{d/2},
\end{align*}
where the inequality follows from Lemma~\ref{lem:quad_growth}(iii) applied to $F$. Thus,
\begin{equation}\label{eq:log_Z_bound}
\log Z \le \frac{d}{2} \log \frac{3\pi}{m\beta} + \frac{\beta b}{2}\log 3 - \beta C.
\end{equation}

Moreover, invoking Lemma~\ref{lem:quad_growth}(iii) once again, we have $F(x) \le \frac{M}{2}\|x\|^2 + B\|x\| + A$ for all $x\in \R^d$. Therefore,
\begin{align}
\int_{\mathbb{R}^d} F(x)p_0(x) dx &\le \int_{\mathbb{R}^d} p_0(x) \left( \frac{M}{2}\|x\|^2 + B\|x\| + A \right) dx \le \frac{M}{2} \kappa_0 + B\sqrt{\kappa_0} + A, \label{eq:energy_bound}
\end{align}
where we used $\mathbb{E}_{\mu_0}[\|x\|^2] \leq \mathbb{E}_{\mu_0}[e^{\|x\|^2}]= \kappa_0$ and we used Jensen's inequality $\mathbb{E}[\|x\|] \le \sqrt{\mathbb{E}[\|x\|^2]}$.

Substituting \eqref{eq:log_Z_bound} and \eqref{eq:energy_bound} into \eqref{eq:KL_decomp}, and grouping the constants $A$ and $C$, we get
\begin{equation}\label{eq:final_bound_on_KL_0}
\begin{aligned}
\KL(\mu_0||\pi_{\beta}) &\leq \log \|p_0\|_\infty + \frac{d}{2} \log \left(\frac{3\pi}{m\beta}\right) + \beta \left(\frac{M\kappa_0}{2} + B \sqrt{\kappa_0}+ (A - C) + \frac{b}{2}\log 3\right)\\
& \leq \beta\left(\frac{\left|\log \|p_0\|_\infty\right|}{\beta} + \frac{d}{2\beta} \left[\log \left(\frac{3\pi}{m\beta}\right)\right]_+ +  \frac{M\kappa_0}{2} + B \sqrt{\kappa_0}+ (A - C) + \frac{b}{2}\log 3\right).
\end{aligned}
\end{equation}
Using the fact that $\beta \geq d\geq 1$, we can set 
\begin{equation}\label{eq:Ctilde_def}
    \tilde C := \left|\log\left( \|p_0\|_\infty\right)\right| + \frac{1}{2}\left[\log \left(\frac{3\pi}{m}\right)\right]_+ + \frac{M\kappa_0}{2} + B \sqrt{\kappa_0}+ (A - C) + \frac{b}{2}\log 3,
    \end{equation}
    with $A = |F(0)|$, $B= \|\nabla F(0)\|$ and $C = \min F$.
\end{proof}

\begin{remark}[A dimension-independent Gaussian initialization]
\label{rem:gaussian_initialization}
The bound above is stated for a general initialization and may therefore
inherit a dependence on $d$ through $\kappa_0$ and
$\|p_0\|_\infty$. Nevertheless, this dependence can be eliminated by a suitable choice of $\mu_0$. For example, let
$\mu_0=\mathcal N\left(0,\frac{\tau}{d}I_d\right)$, with
$0<\tau<\frac12$. Then $\log\|p_0\|_\infty
= \frac d2\log\left(\frac{d}{2\pi\tau}\right)$,
and hence, using $\beta\ge d$,
\[
\begin{aligned}
\log\|p_0\|_\infty
+\frac d2\log\left(\frac{3\pi}{m\beta}\right)
&=
\frac d2\log\left(\frac{3d}{2m\tau\beta}\right)\le
\frac{\beta}{2}
\left[\log\left(\frac{3}{2m\tau}\right)\right]_+.
\end{aligned}
\]
Moreover,
\[
\int_{\mathbb R^d}e^{\|x\|^2}\,d\mu_0(x)
=
\left(1-\frac{2\tau}{d}\right)^{-d/2}
\le
\exp\left(\frac{\tau}{1-2\tau}\right),
\]
so $\kappa_0$ is uniformly bounded in $d$. Consequently,
\eqref{eq:final_bound_on_KL_0} yields
$\KL(\mu_0\|\pi_\beta)\le \widetilde C_\tau\beta$, where
$\widetilde C_\tau$ is independent of both $d$ and $\beta$.
\end{remark}

\section{Proofs for Section \ref{sec:Inexact}}\label{app:proofs_Inexact}

\subsection{Proof of Proposition \ref{prop:iter-norm-bound-final}}\label{app:proof_moment_bounds_IULA}

\begin{proof}[Proof of Proposition \ref{prop:iter-norm-bound-final}] Let $\mathcal{F}_k = \sigma(x_0, \dots, x_k)$ and $g_k = g(x_k, \xi_k)$. From the update rule $x_{k + 1} = x_k - \gamma g_k + \sqrt{2\gamma \beta^{-1}} z_k$, the independence of the noise $z_k \sim \mathcal{N}(0, I_d)$ and the fact that $\E[\|z_k\|^2]=d$, we have
\begin{equation}\label{eq:prec_cond_exp}
    \mathbb{E}[\| x_{k + 1} \|^2 | \mathcal{F}_k] = \mathbb{E}_{\xi_k}[\| x_k - \gamma g_k \|^2 | \mathcal{F}_k] + 2 \gamma \beta^{-1} d.
\end{equation}
Using the weighted Young's inequality $2\langle u, v \rangle \le \alpha \|u\|^2 + \frac{1}{\alpha}\|v\|^2$ with $\alpha = \gamma m$, we decompose the first term without using unbiasedness
\begin{align*}
    \| x_k - \gamma g_k \|^2 &= \| (x_k - \gamma \nabla F(x_k)) + \gamma (\nabla F(x_k) - g_k) \|^2 \\
    &\le (1 + \gamma m) \| x_k - \gamma \nabla F(x_k) \|^2 + \left( \gamma^2 + \frac{\gamma}{m} \right) \| g_k - \nabla F(x_k) \|^2.
\end{align*}
Applying the conditional expectation $\mathbb{E}_{\xi_k}[\cdot \mid \mathcal{F}_k]$ and the bound $\mathbb{E}[\|g - \nabla F\|^2 \mid \mathcal{F}_k] \le \delta(P\|x_k\|^2 + Q)$
\begin{equation}\label{eq:prec_refined_step}
    \mathbb{E}_{\xi_k}[\| x_k - \gamma g_k \|^2\mid \mathcal{F}_k] \le (1 + \gamma m) \| x_k - \gamma \nabla F(x_k) \|^2 + \left( \gamma^2 + \frac{\gamma}{m} \right) \delta (P\|x_k\|^2 + Q).
\end{equation}
From $(m, b)$-dissipativity and $M$-smoothness, we have 
\[\begin{aligned}
\| x_k - \gamma \nabla F(x_k) \|^2 &= \|x_k\|^2 - 2\gamma \langle x_k, \nabla F(x_k)\rangle + \gamma^2 \|\nabla F(x_k)\|^2 \\
& \leq \|x_k\|^2 - 2\gamma m \|x_k\|^2 + 2\gamma b + 2\gamma^2 M^2 \|x_k\|^2 + 2 \gamma^2 B^2  \\
&= (1 - 2\gamma m + 2\gamma^2 M^2)\|x_k\|^2 + 2\gamma b + 2\gamma^2 B^2.\end{aligned}\] On the other hand
\begin{align*}
    (1 + \gamma m)(1 - 2\gamma m + 2\gamma^2 M^2) &= 1 - 2\gamma m + 2\gamma^2 M^2 + \gamma m - 2\gamma^2 m^2 + 2\gamma^3 m M^2 \\
    &= 1 - \gamma m + 2\gamma^2 M^2 - 2\gamma^2 m^2 + 2\gamma^3 m M^2.
\end{align*}
Since $\gamma \le \frac{m}{M^2}$, then $2\gamma^3 m M^2 \le 2\gamma^2 m^2$, ensuring $2\gamma^3 m M^2 - 2\gamma^2 m^2 \le 0$, the term simplifies to $1 - \gamma m + 2\gamma^2 M^2$. Substituting into \eqref{eq:prec_refined_step} and then into \eqref{eq:prec_cond_exp}
\begin{align*}
    \mathbb{E}[\|x_{k+1}\|^2 | \mathcal{F}_k] &\le \left( 1 - \gamma m + 2\gamma^2 M^2 + \frac{\gamma \delta P}{m} + \gamma^2 \delta P \right)\|x_k\|^2 + \gamma \tilde C(\gamma, \beta, d) \\
    &= \left( 1 - \gamma m + \frac{\gamma \delta P}{m} + \gamma^2 (2M^2 + \delta P) \right)\|x_k\|^2 + \gamma \tilde C(\gamma, \beta, d).
\end{align*}
where $\tilde C(\gamma, \beta, d) = (1+\gamma m) (2b + 2\gamma B^2)+ \gamma Q\delta + \frac{Q\delta}{m} + 2\beta^{-1} d$. Exploiting $\beta \geq d$, $\gamma \leq 1$ and the hypothesis on $\delta$, we have
\begin{equation*}\label{eq:Ctemporary_def}
\tilde C(\gamma, \beta, d)\le C,\quad \text{with} \quad C := (1+m)(2b + 2 B^2) + \frac{m^2 Q}{4P} + \frac{mQ}{4P} + 2.
\end{equation*}
Notice that by hypothesis we also have $\gamma \le \frac{m}{3(2M^2+\delta P)}$, so that by $\delta \le \frac{m^2}{4P}$ we have $\frac{\gamma \delta P}{m} \le \frac{\gamma m}{3}$. From $\gamma \le \frac{m}{3(2M^2 + \delta P)}$ we also have $\gamma^2 (2M^2 + \delta P) \le \frac{\gamma m}{3}$. Substituting these yields the final coefficient: $1 - \gamma m + \frac{\gamma m}{3} + \frac{\gamma m}{3} = 1 - \frac{\gamma m}{3}$. Taking total expectations gives $\mathbb{E}[\|x_{k+1}\|^2] \le (1 - \frac{m\gamma}{3})\mathbb{E}[\|x_k\|^2] + \gamma C$. By $\frac{m\gamma}{3} \leq \frac{m^2}{24 M^2}\leq \frac{1}{24}$, then $0<(1 - \frac{m\gamma}{3}) < 1$ and the sequence is uniformly bounded. In particular we have
\begin{equation*}
    \mathbb{E}[\|x_{k+1}\|^2] \le \left(1 - \frac{m\gamma}{3}\right)^k\mathbb{E}[\|x_0\|^2] + \gamma C\sum_{i=1}^k \left(1-\frac{m\gamma}{3}\right)^i\leq \mathbb{E}[\|x_0\|^2]+\frac{3C}{m}.    \end{equation*}
    We can therefore set
    \begin{equation}\label{eq:Gamma_def}
    \Gamma:=\kappa_0+\frac{3(1+m)(2b+2B^2)}{m}+\frac{3(m+1)Q}{4P}+\frac{6}{m}.
    \end{equation}
\end{proof}

\subsection{Proof of Lemma~\ref{lem:exp_bound}.}\label{app:further_proofs_lem}

\begin{lemma}\label{lem:sph_exp}
    Let $I \in \mathbb{R}^{d \times d}$ be the identity matrix and $\sigma$ be the normalized measure on the sphere. Then
    \begin{equation*}
        \int_{\mathbb{S}^{d - 1}} v v^\intercal \, d\sigma(v) = \frac{1}{d} I.
    \end{equation*}
\end{lemma}
\begin{proof}
    This result follows the same line of \cite[Lemma 7.3, point (b)]{Gao2018}.
\end{proof}

\begin{lemma}\label{lem:gaussian_bound}
    If $Z\sim\mathcal{N}(0,I_d)$, $u\in\mathbb{R}^d$, and $v<\tfrac12$, then
\begin{equation}\label{eq:gauss_id}
\mathbb{E}\big[\exp(\langle u, Z\rangle+v\|Z\|^2)\big]
=(1-2v)^{-d/2}\exp\Big(\frac{\|u\|^2}{2(1-2v)}\Big).
\end{equation}
\end{lemma}
\begin{proof}[Proof of Lemma~\ref{lem:gaussian_bound}]
Let $Z \sim \mathcal{N}(0, I_d)$. Writing the expectation as an integral over the density of $Z$:
\begin{align*}
\mathbb{E}\left[e^{\langle u, Z\rangle+v\|Z\|^2}\right] &= (2\pi)^{-d/2} \int_{\mathbb{R}^d} \exp\left( -\frac{1}{2}\|z\|^2 + \langle u, z\rangle + v\|z\|^2 \right) \, dz \\
&= (2\pi)^{-d/2} \int_{\mathbb{R}^d} \exp\left( -\frac{1}{2}(1-2v)\|z\|^2 + \langle u, z\rangle \right) \, dz
\end{align*}
Completing the square in the exponent for $z$, we observe
\[
-\frac{1-2v}{2}\|z\|^2 + \langle u, z\rangle = -\frac{1-2v}{2} \left\| z - \frac{u}{1-2v} \right\|^2 + \frac{\|u\|^2}{2(1-2v)}.
\]
Substituting this back into the integral and factoring out the constant term we obtain
\[
\mathbb{E}\left[e^{\langle u, Z\rangle+v\|Z\|^2}\right]  = \exp\left( \frac{\|u\|^2}{2(1-2v)} \right) (2\pi)^{-d/2} \int_{\mathbb{R}^d} \exp\left( -\frac{1-2v}{2} \left\| z - \frac{u}{1-2v} \right\|^2 \right) \, dz.
\]
Using the Gaussian integral identity $\int_{\mathbb{R}^d} e^{-\frac{\lambda}{2}\|y\|^2} dy = (2\pi)^{d/2}\lambda^{-d/2}$ with $\lambda = 1-2v$ we arrive to
\begin{align*}
\mathbb{E}\left[e^{\langle u, Z\rangle+v\|Z\|^2}\right]  &= \exp\left( \frac{\|u\|^2}{2(1-2v)} \right) (2\pi)^{-d/2} \cdot (2\pi)^{d/2} (1-2v)^{-d/2} \\
&= (1-2v)^{-d/2} \exp\left( \frac{\|u\|^2}{2(1-2v)} \right). \qedhere
\end{align*}
\end{proof}

\bigskip

\begin{proof}[Proof of Lemma~\ref{lem:exp_bound}]
By Lemma \ref{lem:quad_growth} (ii) applied to $F$ we have that $\|\nabla F(x)\| \leq M\|x\| + B$, for all $x\in \R^d$. First of all, notice that the hypothesis in \eqref{eq:conditions_uniform_bound} on $\beta$ and $\gamma$ imply the following
\begin{equation}\label{eq:weaker_conditions_uniform_bound}
\beta>4/m, \quad \text{and} \quad \gamma \le \min\left\{1,\frac{\beta}{8},\frac{m\beta-4}{4\beta M^2}, \frac{1}{m-4/\beta}\right\},
\end{equation}
To show this we just need to notice that $\beta \geq \frac{8}{m}>\frac{4}{m}$, $\gamma \leq \frac{1}{m}\leq \frac{\beta}{8}$ and since $\frac{m\beta - 4}{\beta}= m - \frac{4}{\beta} \geq \frac{m}{2}$ we have $\gamma \leq \frac{m}{8M^2}\leq \frac{m\beta -4}{4\beta M^2}$. Finally, since $m-\frac{4}{\beta}<m$, we have $\frac{1}{m}<\frac{1}{m-4/\beta}$, and therefore $\gamma\leq\frac1m < \frac{1}{m-4/\beta}$. From now on in the proof we will work solely on these weaker conditions, producing in this way a stronger result than the one needed. In the statement of the Lemma we adopted instead the stronger conditions \eqref{eq:conditions_uniform_bound} just for simplicity.

Let $(z_k)_{k\ge 0}$ be i.i.d.\ standard Gaussians in $\mathbb{R}^d$, $z_k\sim\mathcal{N}(0,I_d)$, and consider the ULA update
\[
y_{k+1}=y_k-\gamma\nabla F(y_k)+\sqrt{\frac{2\gamma}{\beta}}\,z_k.
\]
Define the function $V(x):=\exp(\|x\|^2)$ and set $a_k:=y_k-\gamma\nabla F(y_k)$.
Then
\[
\|y_{k+1}\|^2=\|a_k\|^2+\sqrt{\frac{8\gamma}{\beta}}\langle a_k,z_k\rangle+\frac{2\gamma}{\beta}\|z_k\|^2,
\]
hence, by tower property,
\begin{equation}\label{eq:tower}
\mathbb{E}[V(y_{k+1})]
=\mathbb{E}\Big[\exp(\|a_k\|^2)\,
\mathbb{E}\Big[\exp\Big(\sqrt{\tfrac{8\gamma}{\beta}}\langle a_k,z_k\rangle+\tfrac{2\gamma}{\beta}\|z_k\|^2\Big)\,\Big|\,y_k\Big]\Big].
\end{equation}

Now we apply Lemma \ref{lem:gaussian_bound} conditionally on $y_k$ with $u=\sqrt{\frac{8\gamma}{\beta}}\,a_k$ and $v=\frac{2\gamma}{\beta}$. The condition $v<\tfrac12$ is $\beta>4\gamma$. We obtain
\[
\mathbb{E}\left[\exp\left(\sqrt{\tfrac{8\gamma}{\beta}}\langle a_k,z_k\rangle+\tfrac{2\gamma}{\beta}\|z_k\|^2\right)\,\Big|\,y_k\right]
=\left(1-\tfrac{4\gamma}{\beta}\right)^{-d/2}\exp\left(\frac{4\gamma}{\beta-4\gamma}\|a_k\|^2\right).
\]
Plugging this into \eqref{eq:tower} and combining the $\|a_k\|^2$-terms yields
\begin{equation}\label{eq:D25_analogue}
\mathbb{E}[V(y_{k+1})]
=(1-\tfrac{4\gamma}{\beta})^{-d/2}\,
\mathbb{E}\Big[\exp\Big(\frac{\beta}{\beta-4\gamma}\|y_k-\gamma\nabla F(y_k)\|^2\Big)\Big].
\end{equation}

Expand and use dissipativity of $F$ (i.e.\ $\langle y,\nabla F(y)\rangle\ge m\|y\|^2-b$)
\begin{equation}\label{eq:euler_1}
\begin{aligned}
\|y_k-\gamma\nabla F(y_k)\|^2
&=\|y_k\|^2-2\gamma\langle y_k,\nabla F(y_k)\rangle+\gamma^2\|\nabla F(y_k)\|^2\\
&\le (1-2m\gamma)\|y_k\|^2+2\gamma b+\gamma^2\|\nabla F(y_k)\|^2.
\end{aligned}
\end{equation}
Using the linear growth bound $\|\nabla F(x)\|\le M\|x\|+B$ and $(a+b)^2\le 2a^2+2b^2$, we get
\[
\|\nabla F(y_k)\|^2\le 2M^2\|y_k\|^2+2B^2.
\]
Substitute into \eqref{eq:euler_1}:
\begin{equation}\label{eq:euler_2}
\|y_k-\gamma\nabla F(y_k)\|^2
\le \big(1-2m\gamma+2\gamma^2M^2\big)\|y_k\|^2+2\gamma b+2\gamma^2B^2.
\end{equation}

Combining \eqref{eq:D25_analogue} and \eqref{eq:euler_2} gives
\begin{equation}\label{eq:cond_bound_raw}
\mathbb{E}[V(y_{k+1})]
\le (1-\tfrac{4\gamma}{\beta})^{-d/2}\,
\exp\Big(\frac{\beta}{\beta-4\gamma}(2\gamma b+2\gamma^2B^2)\Big)\,
\mathbb{E}\Big[\exp\Big(\frac{\beta}{\beta-4\gamma}(1-2m\gamma+2\gamma^2M^2)\|y_k\|^2\Big)\Big].
\end{equation}
Since $4\gamma/\beta\in(0,1/2]$ we can use the inequality $-\log(1-x)\le 2x$ that holds for $x\in(0,1/2]$, and deduce
\begin{equation}\label{eq:prefactor_bd}
    \left(1-\frac{4\gamma}{\beta}\right)^{-d/2}=\exp\left(-\frac{d}{2}\log\left(1-\frac{4\gamma}{\beta}\right)\right)\le \exp\left(\frac{4d\gamma}{\beta}\right).
\end{equation}
Moreover, since $\frac{4\gamma}{\beta}\le 1/2$, we have $\frac{\beta}{\beta-4\gamma}=\left(1-\frac{4\gamma}{\beta}\right)^{-1}\le 2$, and using $\gamma\le 1$,
\begin{equation}\label{eq:const_bd}
\frac{\beta}{\beta-4\gamma}(2\gamma b+2\gamma^2B^2)
\le 2(2\gamma b+2\gamma^2B^2)\le 4\gamma b+4\gamma B^2.
\end{equation}

Define now
\[
Q(\gamma):=\frac{\beta}{\beta-4\gamma}(1-2m\gamma+2M^2\gamma^2).
\]
By assumption we have that $\gamma\le \beta/8$, so that $4\gamma/\beta\le 1/2$. Then $\frac{\beta}{\beta-4\gamma}
=\left(1-\frac{4\gamma}{\beta}\right)^{-1}
\le 1+\frac{8\gamma}{\beta}$. Consequently,
\begin{align*}
Q(\gamma)
&\le (1-2m\gamma+2M^2\gamma^2)\Big(1+\frac{8\gamma}{\beta}\Big)= 1-(2m-\tfrac{8}{\beta})\gamma
   +2M^2\gamma^2
   -\frac{16m}{\beta}\gamma^2
   +\frac{16M^2}{\beta}\gamma^3 \\
&\le 1-(2m-\tfrac{8}{\beta})\gamma
   +2M^2\gamma^2
   +\frac{16M^2}{\beta}\gamma^3 .
\end{align*}
Using again $\gamma\le \beta/8$, we have $\frac{16M^2}{\beta}\gamma^3 \le 2M^2\gamma^2$, and therefore $Q(\gamma)
\le 1-(2m-\tfrac{8}{\beta})\gamma +4M^2\gamma^2$. Since $\beta>4/m$ and $\gamma \le \frac{m\beta-4}{4\beta M^2}$, then $4M^2\gamma^2
\le \Big(m-\frac{4}{\beta}\Big)\gamma$, and we conclude that
\begin{equation}\label{eq:Q_contract}
Q(\gamma)
\le 1-\Big(m-\frac{4}{\beta}\Big)\gamma.
\end{equation}
Combining \eqref{eq:cond_bound_raw} with \eqref{eq:prefactor_bd}, \eqref{eq:const_bd}, and \eqref{eq:Q_contract}, we obtain
\[
V_{k+1}:=\mathbb E\!\left[V(y_{k+1})\right]
\le
\exp\!\Big(\tfrac{4d\gamma}{\beta}\Big)\,
\exp\!\big(4\gamma b+4\gamma B^2\big)\,\mathbb{E}\left[\exp\left((1-\tilde m\gamma)\|y_k\|^2\right)\right],
\]
where $\tilde m := \left(m-\frac{4}{\beta}\right)$. Let $K:=\frac{4d}{\beta}+4b+4B^2$. Then
\begin{equation}\label{eq:cond_final}
V_{k+1}\le
\exp(K\gamma)\,\mathbb{E}\left[\exp\left((1-\tilde m\gamma)\|y_k\|^2\right)\right]=\exp(K\gamma)\,\mathbb{E}\!\left[V(y_k)^{\,1-\tilde m\gamma}\right].
\end{equation}
Since $0<1-\tilde m\gamma\le 1$, the map $x\mapsto x^{1-\tilde m\gamma}$ is concave on $[0,\infty)$,
and Jensen's inequality gives
\[
\mathbb E\!\left[V(y_k)^{\,1-\tilde m\gamma}\right]\le \big(\mathbb E[V(y_k)]\big)^{1-\tilde m\gamma}=V_k^{\,1-\tilde m\gamma}.
\]
Hence $ V_{k+1}\le e^{K\gamma}\,V_k^{\,1-\tilde m\gamma }$, and taking logs yields $\log V_{k+1}\le K\gamma + (1-\tilde m \gamma)\log V_k$. Iterating we obtain
\[
\log V_k\le (1-\tilde m\gamma)^k \log V_0 + K\gamma\sum_{j=0}^{k-1}(1-\tilde m\gamma)^j
\le \log V_0 + \frac{\beta K}{m\beta-4}.
\]
This is a general bound we achieved by using the weaker conditions \eqref{eq:weaker_conditions_uniform_bound}. Coming back to the original conditions \eqref{eq:conditions_uniform_bound}, since $\beta \geq \frac{8}{m}$, we obtain $\frac{\beta}{m\beta-4}\leq \frac{2}{m}$ and thus
\[
\log V_k\le \log V_0 + \frac{2 K}{m} =  \log V_0 + \frac{8d+8\beta b+8\beta B^2}{m\beta}.
\]
\end{proof}

\section{Details of the comparison with existing works}
\label{app:comparisons}

In this appendix we provide the derivations behind the comparison statements in Sections~\ref{sec:comparisons_ULA} and~\ref{sec:comparisons_I}. We focus on the dependence on \(\beta\), \(\epsilon\), and \(\CLS(\beta,d)\), since these are the quantities that dominate the complexity in the non-convex regime.

\subsection{Exact and Inexact ULA}\label{app:comparisons_ULA}

\paragraph{Raginsky, Rakhlin and Telgarsky (2017)~\cite{Raginsky2017}.}

Raginsky, Rakhlin and Telgarsky~\cite{Raginsky2017} obtain optimization guarantees by first deriving sampling bounds in Wasserstein distance and then converting these bounds into expected objective-value estimates. In our notation, combining \cite[Proposition~10]{Raginsky2017} with \cite[Lemma~6]{Raginsky2017} gives a bound of the form
\begin{equation}\label{eq:RRT_W2_rate_app}
    \E[F(x_k)]-\E[F(x^\pi_{\beta})]
    \leq
    \left(
        \sqrt{\beta}\,C_0^{\RRT}\delta^{1/4}
        +
        \sqrt{\beta}\,C_1^{\RRT}\gamma^{1/4}
    \right)k\gamma
    +
    \sqrt{\beta\CLS(\beta,d)}\,C_2^{\RRT}
    e^{-k\gamma/(\beta\CLS(\beta,d))},
\end{equation}
where \(x^\pi\sim\pi_\beta\), and \(C_i^{\RRT}\) denote constants independent of the algorithmic parameters. The Gibbs concentration term is the same as in our analysis, so one first chooses $\beta=\tilde\Theta\left(\frac{d}{\epsilon}\right)$. To control the exponential term in \eqref{eq:RRT_W2_rate_app}, one needs
\begin{equation}\label{eq:RRT_time_horizon_app}
    k\gamma
    =
    O\left(
        \beta\CLS(\beta,d)
        \log\left(
            \frac{\sqrt{\beta\CLS(\beta,d)}}{\epsilon}
        \right)
    \right)
    =
    \tilde O\left(\beta\CLS(\beta,d)\right).
\end{equation}

In the exact-gradient case, \(\delta=0\). Then the discretization term in \eqref{eq:RRT_W2_rate_app} requires $\sqrt{\beta}\,\gamma^{1/4}k\gamma
    \lesssim
    \epsilon$, and therefore
\[
    \gamma
    =
    O\left(
        \frac{\epsilon^4}{\beta^2(k\gamma)^4}
    \right)
    =
    \tilde O\left(
        \frac{\epsilon^4}{\beta^6\CLS(\beta,d)^4}
    \right).
\]
Combining this with \eqref{eq:RRT_time_horizon_app} gives
\[
    k
    =
    \tilde O\left(
        \frac{\beta\CLS(\beta,d)}{\gamma}
    \right)
    =
    \tilde O\left(
        \frac{\beta^7\CLS(\beta,d)^5}{\epsilon^4}
    \right).
\]
This is worse than our exact-gradient complexity in \eqref{eq:total_count_needed_ULA}. The same estimate also gives an inexact-gradient comparison. In this case, the term involving \(\delta\) in \eqref{eq:RRT_W2_rate_app} requires $\sqrt{\beta}\,\delta^{1/4}k\gamma
    \lesssim
    \epsilon$, and hence
\[
    \delta
    =
    O\left(
        \frac{\epsilon^4}{\beta^2(k\gamma)^4}
    \right)
    =
    \tilde O\left(
        \frac{\epsilon^4}{\beta^6\CLS(\beta,d)^4}
    \right).
\]
Thus, the black-box Wasserstein route based on \cite{Raginsky2017} gives
\[
    k
    =
    \tilde O\left(
        \frac{\beta^7\CLS(\beta,d)^5}{\epsilon^4}
    \right),
    \qquad
    \delta
    =
    \tilde \Theta\left(
        \frac{\epsilon^4}{\beta^6\CLS(\beta,d)^4}
    \right),
\]
which is worse than our inexact-ULA result. One can obtain a sharper comparison by going back to the intermediate estimates in \cite{Raginsky2017}. In particular, combining \cite[Lemma~7]{Raginsky2017} with the Bolley--Villani inequality \cite[Corollary~2.3]{BolleyVillani}, one obtains
\[
    W_2(\mu_k,\nu^c_{k\gamma})
    \leq
    C_\nu^{\RRT}
    \left(
        \sqrt{\KL(\mu_k\|\nu^c_{k\gamma})}
        +
        \KL(\mu_k\|\nu^c_{k\gamma})^{1/4}
    \right),
\]
where \(\nu^c_t\) denotes the law of the continuous-time Langevin diffusion at time \(t\). In the relevant accuracy regime, the dominant contribution is the \(\KL^{1/4}\) term. Tracking the dependence on \(\beta\), \(d\), and \(\CLS(\beta,d)\), this leads to the improved stepsize condition $\gamma
    =
    O\left(
        \frac{\epsilon^4}{\beta(k\gamma)^3}
    \right)
    =
    \tilde O\left(
        \frac{\epsilon^4}{\beta^4\CLS(\beta,d)^3}
    \right)$.
Together with \eqref{eq:RRT_time_horizon_app}, this gives, in the exact-gradient case,
\[
    k
    =
    \tilde O\left(
        \frac{\beta^5\CLS(\beta,d)^4}{\epsilon^4}
    \right).
\]
In the inexact-gradient case, the same refined argument requires $\delta
    =
    O\left(
        \frac{\epsilon^4}{\beta(k\gamma)^3}
    \right)
    =
    \tilde O\left(
        \frac{\epsilon^4}{\beta^4\CLS(\beta,d)^3}
    \right)$,
and therefore yields
\[
    k
    =
    \tilde O\left(
        \frac{\beta^5\CLS(\beta,d)^4}{\epsilon^4}
    \right),
    \qquad
    \delta
    =
    \tilde O\left(
        \frac{\epsilon^4}{\beta^4\CLS(\beta,d)^3}
    \right).
\]
Even this refined reading remains worse than our bounds, both in the number of iterations and in the admissible inexactness level.

\paragraph{Xu, Chen, Zou and Gu (2018)~\cite{XCZG2018}.}

A direct quantitative comparison with Xu, Chen, Zou and Gu~\cite{XCZG2018} is delicate. Their GLD and SGLD rates are expressed in terms of a discrete-time spectral gap \(\lambda(\beta,d)\), together with constants arising from geometric-ergodicity and Poisson-equation estimates. These constants are inherited from the approach of Mattingly, Stuart and Higham~\cite{Mattingly2002}, and are not tracked explicitly in terms of \(\beta\), \(d\), and the structural constants of \(F\). This lack of explicit dependence is important in our setting. The inverse temperature must scale as \(\beta=\tilde\Theta(d/\epsilon)\), and constants such as \(\CLS(\beta,d)\) may depend exponentially on \(\beta\). Consequently, a statement such as $ \lambda(\beta,d)=e^{-\tilde O(d)}$
is too coarse for the comparison performed here. Indeed, the quantities \(\CLS(\beta,d)^{-1}\), \(\CLS(\beta,d)^{-2}\), and \(\CLS(\beta,d)^{-3}\) may all be absorbed into such notation, while leading to substantially different final complexities. Moreover, the proof strategy in \cite{XCZG2018} follows the standard route of first controlling a Wasserstein sampling distance and then converting it into a function-value estimate. This is precisely the route that our analysis avoids. Therefore, although the bounds of \cite{XCZG2018} are qualitatively informative, they do not provide a fully explicit parameter calibration directly comparable to Corollary~\ref{cor:main} or Corollary~\ref{cor:main_I} without additional explicit estimates for \(\lambda(\beta,d)\) and for the Poisson-equation constants.

\paragraph{Zou, Xu and Gu (2021)~\cite{Zou2021Faster}.}

In~\cite{Zou2021Faster} the authors derive a total gradient-evaluation complexity of order
\(k=\tilde O\left(
        \frac{d^4\beta^2}{\rho^4\epsilon^2}
    \right),\)
where \(\rho\) is a Cheeger constant. To compare this expression with our bounds, recall that, as noted in \cite[Remark~4.6]{Zou2021Faster}, the Cheeger constant is related to the Poincaré constant appearing in their notation by $\rho=O\left(d^{-1/2}c_p^{\mathrm{ZXG}}\right)$. The notation \(c_p^{\mathrm{ZXG}}\) corresponds to the reciprocal of the Poincaré constant \(C_P(\beta,d)\) used in \cite{Raginsky2017} and by us, i.e., $c_p^{\mathrm{ZXG}}
    =
    \frac{1}{C_P(\beta,d)}$. Using \eqref{eq:CLS_bound}, the Log-Sobolev constant in our setting satisfies $\CLS(\beta,d)
    =
    O(C_P(\beta,d))
    =
    O\left((c_p^{\mathrm{ZXG}})^{-1}\right)$. Consequently, $\rho
    =
    O\left(
        \frac{1}{d^{1/2}\CLS(\beta,d)}
    \right)$. Substituting this into the bound of \cite{Zou2021Faster} we obtain
    \[k=\tilde O\!\left(\dfrac{\beta^2d^6\CLS^4}{\epsilon^2}\right).\] While they have appearing the term \(\CLS(\beta,d)^4\), our exact and inexact ULA bounds depend instead on \(\CLS(\beta,d)^2\) in the iteration complexity. Since \(\CLS(\beta,d)\) may scale exponentially with \(\beta\) and \(d\), this difference is significant in the non-convex regime.

\paragraph{Kinoshita and Suzuki (2022)~\cite{KinoshitaSuzuki2022}.}

Kinoshita and Suzuki~\cite{KinoshitaSuzuki2022} provide one of the sharpest existing iteration complexities for exact-gradient ULA in this setting:
\[
    k
    =
    \tilde O\left(
        \frac{\beta^2d\,\CLS(\beta,d)^3}{\epsilon}
    \right).
\]
Compared with our bound in \eqref{eq:total_count_needed_ULA},
their result has a better direct dependence on \(\epsilon\), but a worse dependence on \(\CLS(\beta,d)\). Since \(\beta=\tilde\Theta(d/\epsilon)\) and \(\CLS(\beta,d)\) may depend exponentially on \(\beta\), the power of \(\CLS(\beta,d)\) is the dominant term in the non-convex regime. There is also an important structural difference in the treatment of the Gibbs concentration term. Their argument uses a comparison of the form
\[
    \mathbb E[F(x_k)]-\mathbb E[F(x^\pi_{\beta})]
    \leq
    W_2^2(\mu_k,\pi_{\beta})
    +
    \mathbb E[F(x^\pi_{\beta})]-\min F,
\]
which leads to $\mathbb E[F(x_k)]-\min F
    \leq
    W_2^2(\mu_k,\pi_{\beta})
    +
    2\left(\mathbb E[F(x^\pi_{\beta})]-\min F\right)$. The factor \(2\) in front of the Gibbs concentration term is relevant because the inverse temperature \(\beta\) enters the Log-Sobolev constant $\CLS(\beta,d)$. In contrast, our decomposition keeps this term with coefficient \(1\), which allows us to use the sharp leading scale \(d/(2\beta)\) for the Gibbs concentration error. As discussed in Remark~\ref{rem:exact_choice_of_beta}, even constant-factor increases in the required inverse temperature may have an exponential effect through \(\CLS(\beta,d)\).

\paragraph{Zhang, Akyildiz, Damoulas and Sabanis (2023) \cite{Zhang2023}.}
Zhang et al.\ improve the Wasserstein-based analysis of \cite{Raginsky2017} by replacing its time-growing discretization term with one independent of the iteration number; their full-gradient result applies to ULA and relies on quantitative Harris-type contraction estimates \cite{EberleGuillinZimmer2019}. As in the case of Xu et al., a direct comparison in terms of $\CLS(\beta,d)$ is delicate, although for a different reason: their constants are explicit, but are expressed through different contraction quantities rather than through a logarithmic Sobolev constant. Comparing the explicit worst-case exponential envelopes can yield effective powers larger than one of the Raginsky-type envelope for $\CLS(\beta,d)$, but no uniform power relation holds over the full parameter class; we therefore do not report a direct $\CLS$-based rate. At fixed $\beta$, $d$, and structural parameters, their displayed dependence on the target accuracy is less favorable than ours; moreover,
their analysis still proceeds through a $W_2$-to-objective conversion.

\paragraph{Chen, Sekhari and Sridharan (2024).}
We compare our bounds with the recent Lyapunov-potential approach of Chen, Sekhari and Sridharan~\cite{Chen2024}. Their Theorem~8 is stated in the smooth dissipative setting and, in addition, their theorem assumes that the Lyapunov potential associated with the hitting time of the \(\epsilon\)-sublevel set satisfies a self-bounding regularity condition up to third order. This assumption is not a standard regularity assumption on \(F\) itself, but rather an assumption on the hitting-time Lyapunov object used in their proof. Their result gives, for both GLD and SGLD, a high-probability guarantee of reaching a point \(x\) such that \(F(x)-\min F\leq \epsilon\) with oracle complexity
\[
    \tilde O\left(
        \max\left\{
            d^3\max(C_P(\beta,d),1)^3,\,
            \frac{d^2\max(C_P(\beta,d),1)^2}{\epsilon^2}
        \right\}
        \log(1/\delta)
    \right).
\]
This guarantee is of a different nature from ours: it is a hitting-time guarantee, whereas our main results control the expected excess risk of the iterate. For exact gradients, our Corollary~\ref{cor:main} gives quadratic dependence on \(\CLS(\beta,d)\), while the stated Lyapunov-potential bound contains the term $d^3\max(C_P(\beta,d),1)^3$. Since the logarithmic Sobolev inequality implies the Poincar\'e inequality and in our setting $C_P(\beta,d)$ is of the same order of $ \CLS(\beta,d)$ (see \eqref{eq:CLS_bound}), our result improves the power of the dominant geometric constant from three to two. A further important difference concerns the choice of the inverse temperature, which in our bound is \(d/(2\epsilon)\), up to logarithmic factors. By contrast, their Lemma~14 requires $\epsilon
    \geq
    \frac{2d}{\beta}
    \log\left(
        4\pi_{\beta}e\,\beta Ld\,S
    \right)$. Thus, as stated, their theorem uses a more conservative inverse-temperature calibration than the one made explicit in our Corollary~\ref{cor:main}. This distinction is significant in the non-convex regime, because the constants \(C_P(\beta,d)\) and \(\CLS(\beta,d)\) may depend exponentially on \(\beta\). We emphasize that we do not claim that the constant in their temperature condition is intrinsic to the Lyapunov-potential method. Their condition is a sufficient one, obtained by enforcing a fixed lower bound on \(\pi_\beta(A_\epsilon)\). The point is that our expected-risk formulation does not require such a sublevel-mass condition: it only requires the Gibbs bias to be smaller than the allocated error budget, and therefore makes the sharper \(d/(2\epsilon)\) leading scale explicit.

\subsection{Mini-batch stochastic ULA}\label{app:comparisons_mini}

\paragraph{Raginsky, Rakhlin and Telgarsky (2017) \cite{Raginsky2017}.}
       The bounds already seen in the previous section translate to a total number of gradient evaluations of the single $f(\cdot, \zeta)$ of
    \[k s = O\left(\frac{k}{\delta}\right) = \tilde O \left(\frac{\beta^{13} \CLS(\beta,d)^9}{\epsilon^8}\right).\]
    This is clearly much worse than our estimate in \eqref{eq:total_count_needed_ST}. Even allowing for the more refined analysis discussed above, the number of gradient evaluations of the single $f(\cdot, \zeta)$ remains
    \[k s = O\left(\frac{k}{\delta}\right) = \tilde O \left(\frac{\beta^{9} \CLS(\beta,d)^7}{\epsilon^8}\right),\]
    which is still clearly much worse than our estimate found in \eqref{eq:total_count_needed_ST}.

\paragraph{Xu, Chen, Zou and Gu (2018) \cite{XCZG2018}.} We send the reader to Appendix \ref{app:comparisons_ULA}, where we discussed the impossibility to truly compare the two works. However, \cite{XCZG2018} reports for SGLD a gradient complexity of order
\[
ks = \tilde O\!\left(\frac{d^7}{\lambda(\beta,d)^5 \epsilon^5}\right),
\]
where $\lambda(\beta,d)$ denotes the exponential convergence rate (referred to there as the discrete spectral gap) of the discrete-time Markov chain generated by GLD or SGLD. In \cite[Remark~3.5]{XCZG2018}, the authors state that this quantity is of the same order as Raginsky et al.'s continuous-time spectral gap $\lambda^*(\beta,d)$ (and thus also of $\CLS(\beta,d)$); however, this comparison is not made quantitative in a way that yields explicit dependence on $\beta$ and $d$ for parameter tuning. Still, even under the most favorable interpretation (namely, replacing $\lambda(\beta,d)^{-1}$ by a quantity of the same order as our complexity scale $\CLS(\beta,d)$), one would obtain at best
\[
ks = \tilde O\!\left(\frac{d^7\,\CLS(\beta,d)^5}{\epsilon^5}\right),
\]
which is still worse than our bound in \eqref{eq:total_count_needed_ST}.

\paragraph{Zou, Xu, and Gu (2021) \cite{Zou2021Faster}.} While our analysis requires a small parameter $\delta$, i.e., a large batch-size $s$ (as the already discussed works \cite{Raginsky2017,XCZG2018}), Zou et al. \cite{Zou2021Faster} provide a convergence analysis for SGLD with unbiased estimators that avoids this requirement. By combining \cite[Corollary 4.7]{Zou2021Faster} and \cite[Corollary 4.8]{Zou2021Faster}, they establish rates in Total Variation distance without requiring vanishing steps. However, their complexity bound depends heavily on the Cheeger constant $\rho$. Specifically, they derive a total gradient evaluation complexity of order $ks = \tilde{O}\left(\frac{d^4\beta^2 }{\rho^4\epsilon^2}\right)$. As described also in Appendix \ref{app:comparisons_ULA}, we can translate this into \[ks=\tilde O\!\left(\dfrac{\beta^2d^6\CLS^4}{\epsilon^2}\right).\]
Since $\CLS(\beta, d)$ typically scales exponentially with $d$ and, more importantly, with $\beta$, which we must set as $O(d/\epsilon)$, it is clearly the dominant bottleneck in non-convex optimization. Therefore, our result in \eqref{eq:total_count_needed_ST} (which achieves a dependence of $O(\CLS(\beta, d)^3)$) offers a tighter bound in this regime. Extending our techniques to exploit the unbiasedness of the gradient estimator, as done in \cite{Zou2021Faster}, remains a promising avenue for future research.

\paragraph{Kinoshita and Suzuki (2022) \cite{KinoshitaSuzuki2022}.}
A related analysis is provided by \cite{KinoshitaSuzuki2022}, who study variance-reduced Langevin dynamics. Their setting is structurally different from the plain mini-batch oracle considered here: variance reduction requires a finite-sum objective and periodic full-gradient computations, so the total complexity depends explicitly on the dataset size \(n\). In particular, their finite-sum complexity contains the additional cost of variance reduction, of order \(\tilde O(n+\sqrt n/\epsilon)\). Ignoring this difference in oracle structure, their optimization guarantee leads to a bound of the form
\[
    ks
    =
    \tilde O\left(
        \frac{d\beta^2\CLS(\beta,d)^3}{\epsilon}
    \right).
\]
Thus, at the level of the displayed polynomial factors, their result has a better dependence on \(\epsilon^{-1}\) than Corollary~\ref{cor:main_MB}, while both bounds contain a cubic dependence on the functional-inequality constant. However, this comparison is misleading if one ignores the required inverse temperature. As explained in Appendix~\ref{app:comparisons_ULA}, their analysis requires a more conservative choice of \(\beta\) than ours. Consequently, although their bound has a better dependence on \(\epsilon\), our mini-batch result can be substantially sharper in the metastable non-convex regime because it keeps the Gibbs-temperature calibration optimal. Their bound is superior only in the convex setting, but in the convex setting better bounds are anyway available, see \cite{Dalalyan2017,DurmusMoulines2017,durmus2019W2,Dalalyan2019}.

\paragraph{Zhang, Akyildiz, Damoulas and Sabanis (2023) \cite{Zhang2023}.} We send the reader to Appendix \ref{app:comparisons_ULA}, where we discussed the impossibility to truly compare the two works.

\paragraph{Chen, Sekhari and Sridharan (2024).}
The stochastic-gradient part of \cite{Chen2024} is particularly relevant for comparison with Corollary~\ref{cor:main_MB}. Their Theorem~8 gives the same displayed oracle complexity for GLD and SGLD, provided the stochastic-gradient oracle is unbiased and uniformly sub-Gaussian and the stochastic losses satisfy sample-wise smoothness/dissipativity assumptions. Under these assumptions, their stochastic-gradient result is sharper than the mini-batch bound obtained in Corollary~\ref{cor:main_MB}, which gives
\[
    ks
    =
    \tilde O\left(
        \frac{\beta^4 d\,\CLS(\beta,d)^3}{\epsilon^4}
    \right).
\]
This difference is expected. Their proof exploits the unbiased stochastic-gradient structure directly, whereas our analysis treats mini-batch gradients through the general inexact-oracle condition $\mathbb E\|g(x,\xi)-\nabla F(x)\|^2
    \leq
    \delta(P\|x\|^2+Q)$. This framework allows biased and state-dependent approximation errors, and is therefore suited to the zeroth-order finite-difference estimators considered in Section~\ref{sec:ZO}. It is not intended to be optimal for the purely unbiased stochastic-gradient case. Nevertheless, we stress again that their stated inverse-temperature calibration is more conservative than ours, and this heavily affects the final bound through the exponential dependence of \(C_P(\beta,d)\).

\subsection{Zeroth order ULA}\label{app:comparisons_ZO}

\paragraph{Roy, Shen, Balasubramanian and Ghadimi (2022) \cite{Roy2022}.}

A particularly relevant comparison is Theorem~3.2 of \cite{Roy2022}, which studies stochastic zeroth-order Langevin under smoothness and a Log-Sobolev inequality. Their convention for the logarithmic Sobolev inequality is $\KL(\mu\|\pi_\beta)\leq \frac{1}{2\lambda}\,I(\mu\|\pi_\beta)$, so that, in our notation, $\lambda=\CLS(\beta,d)^{-1}$. Their target is of the form $\pi\propto e^{-f}$, whereas in our optimization setting $\pi_\beta(dx)\propto e^{-\beta F(x)}\,dx$. Thus the correct change of variables is $f=\beta F$, $M_{\rm Roy}=\beta M$,
    $\lambda=\CLS(\beta,d)^{-1}$, and $\gamma_{\rm Roy}=\frac{\gamma}{\beta}$,
where $\gamma_{\rm Roy}$ denotes the Langevin stepsize in \cite{Roy2022}. We keep the notation $h$ for the zeroth-order smoothing radius and $s$ for the number of random directions, corresponding to their $\nu$ and $b$, respectively. Theorem~3.2 of \cite{Roy2022} is a Wasserstein sampling guarantee. In the non-convex regime, where typically $\CLS(\beta,d)>1$, equivalently $\lambda<1$, and tracking carefully the dependence on $M_{\rm Roy}$, which depends on $\beta$, their parameter choices for obtaining $W_2(\mu_k,\pi_\beta)\leq \epsilon_{\mathrm W}$, scale as
\[
    \gamma_{\rm Roy}
    =
    \tilde\Theta\left(
        \frac{\epsilon_{\mathrm W}^2}
        {M_{\rm Roy}^2 d\,\CLS(\beta,d)^2}
    \right),
    \quad
    h
    =
    \tilde\Theta\left(
        \frac{\epsilon_{\mathrm W}}
        {M_{\rm Roy}d^{3/2}\CLS(\beta,d)}
    \right),\quad  s
    =
    \tilde\Theta\left(
        \frac{M_{\rm Roy}^2d^2\CLS(\beta,d)^2}
        {\epsilon_{\mathrm W}^2}
    \right).
\]
Equivalently, since $M_{\rm Roy}=\beta M$, and suppressing fixed powers of $M$ and oracle-noise constants, this gives
\[
    \gamma
    =
    \beta\gamma_{\rm Roy}
    =
    \tilde\Theta\left(
        \frac{\epsilon_{\mathrm W}^2}
        {\beta d\,\CLS(\beta,d)^2}
    \right),
    \quad
    h
    =
    \tilde\Theta\left(
        \frac{\epsilon_{\mathrm W}}
        {\beta d^{3/2}\CLS(\beta,d)}
    \right), \quad s
    =
    \tilde\Theta\left(
        \frac{\beta^2d^2\CLS(\beta,d)^2}
        {\epsilon_{\mathrm W}^2}
    \right).
\]
The resulting number of iterations is $k
    =
    \tilde O\left(
        \frac{\beta^2 d\CLS(\beta,d)^3}
        {\epsilon_{\mathrm W}^2}
    \right),$ and hence the total number of function evaluations is
\[
    ks
    =
    \tilde O\left(
        \frac{\beta^4d^3\CLS(\beta,d)^5}
        {\epsilon_{\mathrm W}^4}
    \right).
\]
Here we have set the function-value noise level of \cite{Roy2022} to zero; keeping it only adds the oracle-noise factor appearing in their theorem. If one uses Theorem~3.2 of \cite{Roy2022} as a black box for optimization, one first obtains the Wasserstein guarantee and then converts it into an objective-value guarantee, for instance using Lemma~6 of \cite{Raginsky2017}. This conversion does not remove the powers of $\CLS(\beta,d)$ already paid in the Wasserstein sampling theorem. Therefore, the black-box route through \cite{Roy2022} inherits a $\CLS(\beta,d)^3$ dependence in the number of iterations and a $\CLS(\beta,d)^5$ dependence in zeroth-order oracle complexity. Notice also that the same phenomenon appears at the level of the smoothing radius. The Wasserstein route of \cite{Roy2022}, after tracking $M_{\rm Roy}=\beta M$, requires
\[
    h
    =
    \tilde\Theta\left(
        \frac{\epsilon_{\mathrm W}}
        {\beta d^{3/2}\CLS(\beta,d)}
    \right),
\]
whereas our Gaussian zeroth-order analysis allows $h=\tilde\Theta\left(
        \frac{\epsilon}
        {\beta d^{3/2}\sqrt{\CLS(\beta,d)}}
    \right)$. Thus, for comparable accuracies, our admissible smoothing radius is larger by a factor $\sqrt{\CLS(\beta,d)}$. This is consistent with the fact that \cite{Roy2022} targets a full Wasserstein approximation, while our final goal is the expected excess risk $\E[F(x_k)]-\min F\leq \epsilon$. By working directly with relative entropy and converting KL control into objective-value error, our zeroth-order bound requires only $k
    =
    \tilde O\left(
        \frac{\beta^2 d\CLS(\beta,d)^2}
        {\epsilon^2}
    \right)$ and $ks
    =
    \tilde O\left(
        \frac{\beta^4d^3\CLS(\beta,d)^3}
        {\epsilon^4}
    \right)$. The improvement should therefore not be interpreted as contradicting \cite{Roy2022}: their theorem gives a stronger distributional guarantee in $W_2$, while ours gives a sharper guarantee for the different objective of expected excess risk. Interestingly, if one inspects the proof of \cite[Theorem~3.2]{Roy2022} and stops at the intermediate entropy estimate, before the final use of Talagrand's inequality, then combining that estimate with our KL-to-objective argument recovers the same leading dependence on $\CLS(\beta,d)$ as our result, providing a useful sanity check.

\paragraph{Comparison with Liu and Wang (2020) \cite{LiuWang2020}.} The work \cite{LiuWang2020} is also motivated by zeroth-order Langevin methods for derivative-free global optimization, but its assumptions make the result difficult to compare with ours. Their main theorem requires
\[
    \frac mM>
    \mathbb E_{u\sim p}[\|u\|\,\|\nabla\log p(u)\|]+1,
\]
where $p$ is the law of the probing direction $u$, $m$ is the dissipativity constant and $M$ is the smoothness constant. However, under the usual smooth dissipative assumptions,
\[
    \langle x,\nabla F(x)\rangle\ge m\|x\|^2-b,
    \qquad
    \|\nabla F(x)-\nabla F(y)\|\le M\|x-y\|,
\]
one necessarily has $m\le M$. Hence the above condition cannot hold. In particular, for their Gaussian probing example it becomes $m>M(d+1)$, which is incompatible with $m\le M$ for every $d\ge1$. For this reason, \cite{LiuWang2020} does not provide a directly applicable benchmark in the standard smooth dissipative regime considered here. Moreover, their main result is a Wasserstein tracking bound to the Gibbs measure, rather than an explicit expected-excess-risk oracle complexity. Our zeroth-order result applies under the usual smoothness and dissipativity assumptions and gives explicit choices of the stepsize, smoothing radius, and number of random directions, leading directly to an expected excess-risk bound.

\section{Experimental Details} \label{app:exp_details}

In this appendix, we provide the implementation and experimental details for the numerical results presented in Section~\ref{sec:experiments}. All scripts were implemented in Python~3.11.15 using PyTorch~2.6.0 \cite{pytorch}, NumPy~2.4.6 \cite{numpy}, and Matplotlib~3.10.9 \cite{matplotlib}. All experiments were conducted on the Grid'5000 testbed. The hardware specifications of the machine used in the experiments are reported in Table~\ref{tab:machine_details}.

\begin{table}[H]
    \centering
    \caption{Hardware specifications of the machine used for the experiments.}
    \label{tab:machine_details}
    \begin{tabular}{ll}
        \toprule
        Feature & Specification \\
        \midrule
        CPU & $64 \times$ Intel(R) Xeon(R) Silver 4215 CPU @ 2.50\,GHz \\
        GPU & $2 \times$ NVIDIA Quadro RTX 6000 \\
        RAM & 256\,GB \\
        \bottomrule
    \end{tabular}
\end{table}
\noindent In the experiments of Section~\ref{sec:experiments}, we investigated how the numerical performance of Algorithm~\ref{alg:ZO-ULA} is affected by its parameters. Specifically, we considered three standard nonconvex benchmark functions: the Ackley function ($F_A$), the Rastrigin function ($F_R$), and the Levy function ($F_L$), defined as follows.

\begin{equation*}
    \begin{aligned}
    F_A(x) &=-20 \exp\left(-0.2\sqrt{\frac{1}{d}\sum_{i=1}^{d}x_i^2}\right)-\exp\left(\frac{1}{d}\sum_{i=1}^{d}\cos(2\pi x_i)\right)+20+e + \lambda\|x\|^2,\\
    F_R(x) &=10d+\sum_{i=1}^{d}\left(x_i^2-10\cos(2\pi x_i)\right) + \lambda\|x\|^2,\\
    F_L(x) &=\sin^2\left(\pi\phi(x_1)\right)+\sum_{i=1}^{d-1}(\phi(x_i)-1)^2\left(1+10\sin^2\!\left(\pi\phi(x_i)+1\right)\right)\\
    & \hspace{2cm}+(\phi(x_d)-1)^2\left(1+\sin^2\left(2\pi\phi(x_d)\right)\right) + \lambda\|x\|^2,
    \end{aligned}
\end{equation*}
where $\phi(x_i)=1+\frac{x_i-1}{4}$. For each benchmark function, we considered problem dimensions $d \in \{10,50,100\}$ and repeated every experiment five times. We added to each function a regularizer term $\lambda\|x\|^2$, with $\lambda = 10^{-3}$. In each repetition, Algorithm~\ref{alg:ZO-ULA} was initialized from a point $x_0 \in \mathbb{R}^d$ sampled uniformly at random from the hypercube $[-32.768,32.768]^d$ for the Ackley function, $[-5.12,5.12]^d$ for the Rastrigin function, and $[-10,10]^d$ for the Levy function.\footnote{These hypercubes correspond to the standard evaluation domains commonly adopted for these benchmark functions, see \url{https://www.sfu.ca/~ssurjano/optimization.html}.} The values of the step size $\gamma$, the exploration parameter $\beta$, and the number of sampled directions used to assess the performance of Algorithm~\ref{alg:ZO-ULA} were selected as follows. The parameters $\gamma$ and $\beta$ were chosen via grid-search over the sets $\{ 10^{-7}, 10^{-6}, 10^{-5}, 10^{-4}, 10^{-3}, 10^{-2}, 10^{-1}, 1.0, 10\}$ and $\{10^{-2}, 10^{-1}, 1.0, 10, 10^2, 10^3, 10^4\}$ respectively. The number of sampled directions $s$ was selected from the set $\{1,\lceil d/5\rceil,\lceil d/3\rceil,\lceil d/2\rceil,\lceil 2d/3\rceil,d\}$. The directions used to construct the finite-difference surrogates were sampled independently and uniformly from the unit sphere.

\end{document}

%% file: MacroBook.tex
\def\argmin{\operatornamewithlimits{arg\,min}}

\DeclareMathOperator{\E}{\mathbb{E}}

\newtheorem{theorem}{Theorem}[section]
\newtheorem{lemma}[theorem]{Lemma}
\newtheorem{proposition}[theorem]{Proposition}
\newtheorem{corollary}[theorem]{Corollary}
\newtheorem{definition}[theorem]{Definition}
\newtheorem{remark}[theorem]{Remark}

\newtheorem{assumption}{Assumption}[section]

\providecommand{\supp}[1]{\text{supp}( {#1} )}

\newcommand{\R}{\mathbb R}

\newcommand{\N}{\mathbb N}